\documentclass[oneside, 11pt]{amsart}
\usepackage{setspace} 
\usepackage[utf8]{inputenc}
\usepackage[T1]{fontenc}
\usepackage{lmodern}
\usepackage{amsmath}
\usepackage{amsthm}
\usepackage{amssymb}
\usepackage{units}
\usepackage{amscd}
\usepackage{geometry}
\usepackage{graphicx}
\usepackage{tikz}
\usepackage{tikz-cd}
\usepackage{pst-plot}
\usepackage{enumitem}
\usetikzlibrary{arrows,automata}
\usepackage{graphicx}
\usepackage{url}
\usepackage{pgfplots}
\usepackage{mathtools}
\usepackage{hyperref}
\usepackage{comment}
\usepackage[capitalize]{cleveref}

{\newtheorem{theorem}{Theorem}[section]}
{\newtheorem{corollary}[theorem]{Corollary}}
{}
{}
{}
{\theoremstyle{definition} \newtheorem{definition}[theorem]{Definition}
	
	\newtheorem{recollection}[theorem]{Recollection}
	\newtheorem{proposition}[theorem]{Proposition}
	\newtheorem{lemma}[theorem]{Lemma}}
{\theoremstyle{remark} \newtheorem{remark}[theorem]{Remark}
	
	\newtheorem{example}[theorem]{Example}}

\crefname{exemple}{Example}{Examples}
\crefname{lemma}{Lemma}{Lemmas}
\crefname{corollary}{Corollary}{Corollaries}
\crefname{proposition}{Proposition}{Propositions}
\crefname{conjecture}{Conjecture}{Conjectures}
\crefname{theo}{Theorem}{Theorems}
\crefname{construction}{Construction}{Constructions}
\crefname{recollection}{Recollection}{Recollections}

	\newcommand{\Id}{1}	
\newcommand{\Set}{\text{Set}}
\newcommand{\op}{op}

\newcommand{\sS}{\mathrm{sSet}}
\newcommand{\Hom}{\mathrm{Hom}}

\newcommand{\Top}{\mathrm{Top}}
\newcommand{\Real}[1]{{}|#1|}
\newcommand{\RealP}[1]{{}\Real{#1}_P}
\newcommand{\Sing}{\mathrm{Sing}}

\newcommand{\pr}{pr}
\newcommand{\colim}{\operatornamewithlimits{colim}}
\newcommand{\hocolim}{\operatornamewithlimits{hocolim}}
\newcommand{\Ex}{\mathrm{Ex}}
\newcommand{\sd}{\mathrm{sd}}
\newcommand{\sdn}{\sd^{\text{naiv}}_P}
\newcommand{\Exi}{\mathrm{Ex}^{\infty}}
\newcommand{\Exn}{\mathrm{Ex}_P^{\text{naiv}}}
\newcommand{\lv}{\mathrm{l.v}}

\newcommand{\Map}{\mathrm{Map}}
\newcommand{\Fun}{\mathrm{Fun}}

\newcommand{\nd}{\mathrm{n.d.}}

\newcommand{\ev}{\mathrm{ev}}

\newcommand{\R}{\mathbb{R}}

\newcommand{\fil}[1]{(#1,\varphi_{#1})}

\newcommand{\C}{\textbf{C}}

\newcommand{\N}{\mathbb{N}}
\newcommand{\Diag}{\mathrm{Diag}}

\newcommand{\RealNP}[1]{\Real{#1}_{N(P)}}
\newcommand{\Int}{\mathrm{Int}}

\newcommand{\cof}{\mathrm{cof}}
\newcommand{\Ho}{\mathrm{Ho}}

\newcommand{\Strat}{\mathrm{Strat}}

\newcommand{\sStrat}{\mathrm{sStrat}}

\newcommand{\Poset}{\mathrm{Poset}}

\newcommand{\sSJK}{\sS^{\mathrm{Joyal-Kan}}}

\newcommand{\TopNP}{\Top_{N(P)}}

\newcommand{\fib}{\mathrm{fib}}

\newcommand{\CW}{\mathrm{CW}}

\newcommand{\lab}[1]{(#1,\lambda_{#1})}

\renewcommand{\Im}{\mathrm{Im}}
\newcommand{\define}[1]{{\bf \boldmath{#1}}}

\newcommand{\Link}[1]{\mathrm{Link}_{#1}}
\newcommand{\redu}{\text{red}}
\newcommand{\nondegen}[1]{\widehat{#1}}

\newcommand{\Hol}{\mathrm{HoLink}}
\newcommand{\HolINP}{\textit{H\!\hspace{1pt}o}\mathrm{Link}_{\I}}
\newcommand{\HolIP}{\mathcal{H}\!\textit{o}\mathrm{Link}_{\I}}
\newcommand{\hol}{\mathcal{H}\!\textit{o}\mathrm{Link}}
\newcommand{\RealStrat}[1]{\Real{#1}_{\Strat}}
\newcommand{\SingStrat}{\Sing_{\Strat}}
\newcommand{\TopP}{\Top_P}
\newcommand{\Con}{\mathrm{Con}}
\newcommand{\PsS}{P\text{-}\sS}
\newcommand{\PCW}{P\text{-}\CW}
\pgfplotsset{compat=1.16}
\newcommand{\spaceperiod}{\makebox[0pt][l]{\,.}}
\newcommand{\spacecomma}{\makebox[0pt][l]{\,,}}

\newcommand{\I}{\mathcal{I}}
\newcommand{\J}{\mathcal{J}}

\author{Sylvain Douteau}

\address{Sylvain Douteau \\ Matematiska institutionen, Stockholm universitet}
\email{sylvain.douteau@gmail.com}
\author{Lukas Waas}

\address{Lukas Waas
 \\ Mathematisches Institut, Universität Heidelberg, Heidelberg, Germany}
\email{lwaas@mathi.uni-heidelberg.de}
\title{From homotopy links to 
stratified homotopy theories}

\begin{document}
\maketitle
\begin{abstract}
In previous work, the first author defined homotopy theories for stratified spaces from a simplicial and a topological perspective. In both frameworks stratified weak-equivalences are detected by suitable generalizations of homotopy links. These two frameworks are connected through a stratified version of the classical adjunction between the realization and the functor of singular simplices. Using a modified version of this adjunction, the first author showed that over a fixed poset of strata the two homotopy theories were equivalent. Building on this result we now show that the unmodified adjunction induces an equivalence between the global homotopy theories of stratified spaces and of stratified simplicial sets. We do so through an in depth study of the homotopy links.  As a consequence, we prove that the classical homotopy theory of conically stratified spaces embeds fully-faithfully into the homotopy theory of all stratified spaces.
\end{abstract}

	\tableofcontents

\section{Introduction}\label{sec:introduction}

Stratifications and stratified spaces were first introduced by Whitney \cite{Whitney}, Thom \cite{Thom} and Mather to describe manifolds with singularities. In this context, stratifications correspond to decompositions into strata that are themselves manifolds and satisfy certain compatibility conditions. Stratified spaces of this kind are usually called pseudo-manifolds, since they provide an extension of the class of manifolds over which many invariants can be extended while retaining their key properties. Most notably, intersection cohomology, introduced by Goresky and MacPherson in \cite{IntersectionHomologyI} is an extension of singular cohomology to pseudo-manifolds that still satisfies Poincaré Duality. Intersection cohomology is only an invariant up to stratum preserving homotopy equivalence, not arbitrary homotopy equivalences, which motivates the developpment of a homotopy theory for stratified spaces. This question was already asked by Goresky and MacPherson  (see for example \cite[Problems 4 and 11]{Borel}).

 Motivated by the study of the homotopical properties of stratified objects, Quinn \cite{quinn1988homotopically} introduced a more general notion of stratified spaces: Homotopically stratified sets. While for pseudo-manifolds compatibility conditions between the strata are characterized by geometric links, homotopically stratified sets make use of homotopy links (or holinks), which are spaces of paths going from one stratum to another. 
As Miller later showed in \cite[Theorem 6.3]{miller2013}, stratified homotopy equivalences between homotopically stratified sets can be fully characterized as those maps inducing weak equivalences on all strata and homotopy links. This inspired the idea that homotopy links and strata should be the basic blocks for defining a stratified homotopy type.

Let us now recall a few classical results of homotopy theory: In \cite{Quillen} Quillen introduced the notion of a model category and immediately gave two seminal examples. The model category of topological spaces, $\Top$, and the model category of simplicial sets, $\sS$. He also showed that the adjunction relating those two categories $\Real{-}\colon \sS\leftrightarrow\Top\colon\Sing$ was in fact a Quillen equivalence, meaning that the two homotopy theories are equivalent. This is a key result in homotopy theory since it allows one to prove results about the homotopy theory of spaces while working in a purely combinatorial setting.

In this paper, we consider stratified spaces in the broadest sense - those are just spaces with a continuous map towards a poset, $X\to P$, as defined by Woolf in \cite{Woolf} and popularized by Lurie in \cite{HigherAlgebra}. This point of view gives rise to two kind of categories. First, the categories of stratified objects and stratum preserving map over a fixed poset. Here, objects can be taken to mean either topological spaces or simplicial sets, producing the categories $\Top_P$ (resp $\sS_P$), of topological spaces (resp simplicial sets) stratified over the poset $P$. Secondly, there are the categories of stratified objects over all posets, where maps are given by commutative squares. Those are denoted $\Strat$ and $s\Strat$ for the topological and simplicial versions. Note that the categories $\Top_P$, correspond to the fibers of the functor $\Strat\to\text{Poset}$ over the discrete categories $\{P\}$, and similarly for $\sS_P$ and $s\Strat$.

The emerging field of stratified homotopy theory has seen a lot of recent activity (See for example \cite{AFRStratifiedHomotopyHypothesis,nand2019simplicial,haine2018homotopy}, and \cite[Appendix A]{HigherAlgebra}) with applications to algebraic geometry \cite{Exodromy}, $G$-isovariant homotopy theory \cite{IsovariantHomotopyTheory}, and to the study of classical invariants of stratified spaces \cite{NaturalOperationsIntersectionCohomology}.
In this context, the first author showed that there exist two independent model structures on the category $\Top_P$, of spaces \cite{douteauEnTop} and  $\sS_P$, of simplicial sets \cite{douSimp}, stratified over a fixed poset $P$. See also \cite{douteauFren}. Both are defined from appropriate notions of homotopy links (and generalization of those to tuples of strata), i.e. weak-equivalences between stratified objects are maps inducing weak-equivalences between all strata and homotopy links. In addition, it is shown in \cite{douteauEnTop} that the model structures on $\Top_P$ and $\sS_P$, for varying $P$, assemble to form model structures on $\Strat$ and $s\Strat$ respectively. Furthermore, the classical adjunction $\Real{-}\colon\sS\leftrightarrow\Top\colon\Sing$ admits  stratified versions, $\RealP{-}\colon\sS_P\leftrightarrow\Top_P\colon\Sing_P$, for all posets $P$. Those can be glued together, producing a global adjunction $\RealStrat{-}\colon s\Strat\leftrightarrow\Strat\colon\SingStrat$.

 The adjunction $\RealStrat{-}\colon s\Strat\leftrightarrow\Strat\colon\SingStrat$ is not a Quillen equivalence, however. In fact, if $P$ is not discrete, the adjunction $\RealP{-}\colon\sS_P\leftrightarrow\Top_P\colon\Sing_P$ is not even a Quillen adjunction. Nevertheless, by composing the above adjunction with a suitably defined stratified subdivision, the first author showed in \cite{douteau2021stratified} that there exists a modified Quillen equivalence
\begin{equation}\label{eq:QE_sd_intro}
    \RealP{\sd_P(-)}\colon\sS_P\leftrightarrow\Top_P\colon\Ex_P\Sing_P,
\end{equation}
meaning that the homotopy theory of spaces, and simplicial sets, stratified over the same poset, coincide.
On the other hand the adjunction \eqref{eq:QE_sd_intro} is no longer compatible with the gluing process producing $\Strat$ and $s\Strat$, which means that it does not allow for a direct comparison between the homotopy theory associated to the two global model categories.

Relatedly, one would want to interpret Miller's theorem \cite[Theorem 6.3]{miller2013} - which characterizes stratum-preseving homotopy equivalences between suitably nice stratified spaces - as a statement about cofibrant-fibrant objects in a model category. But in fact, the objects appearing in Miller's theorem are almost never cofibrant as objects of $\Strat$. 

Both of those observations might motivate one to consider some other model structure for the category of stratified spaces, to remedy these problems. It turns out however that such a structure does not exist, without making major changes to the topological setup, as we show in \cref{section:appendix}. What we show instead is that even though the adjunctions $(\RealP{-},\Sing_P)$ and $(\RealStrat{-},\SingStrat)$ are not Quillen equivalences, they still induce - in a very strong sense - isomorphisms between the homotopy theories of stratified simplicial sets and stratified spaces (see \cref{theo:Equivalence_Simplicial_Homotopy_Category} and \cref{rem:Equiv_Simplicial_Localization}).
\begin{theorem}
The adjoint pairs $\RealP{-}\dashv \Sing_P$ and $\RealStrat{-}\dashv \SingStrat$ descend to equivalences of homotopy categories
\begin{align*}
    \RealP{-}\colon \Ho\sS_P&\leftrightarrow\Ho\Top_P\colon \Sing_P,\\ 
    \RealStrat{-}\colon \Ho s\Strat&\leftrightarrow\Ho\Strat \colon \SingStrat .
\end{align*}
\end{theorem}

Note that an even stronger statement holds. Instead of the homotopy categories, one can consider the simplicial localizations in the sense of Dwyer and Kan \cite{DwyerKanCalculating}. Then, one sees that in addition to inducing equivalences between the homotopy categories, the pairs $\RealP{-}\dashv\Sing_P$ and $\RealStrat{-}\dashv\SingStrat$ induce Dwyer-Kan equivalences between the simplicial localizations. In particular, this means that the $\infty$-categories associated to stratified spaces and stratified simplicial sets are equivalent.

We also consider the case of triangulable conically stratified spaces (see \cref{rem:conical_pman_homotopically_stratified}). As mentioned earlier, stratum preserving homotopy equivalences between those spaces can be explicitly characterized (see \cite[Theorem 6,3]{miller2013}). In particular, for those objects weak-equivalences and stratum preserving homotopy equivalences coincide. In fact, even though those objects are not fibrant-cofibrant, they behave as if they were. What this means is that in order to study their homotopy theory, one does not need to work with their hom-set in the homotopy category $\Ho\Strat$, but one can work instead with the much more explicit set of stratified maps up to stratum preserving homotopies, because the two coincide. This can be phrased more rigorously as follows (\cref{Cor:Conic_embed,rem:Embedding_Global}):

\begin{theorem}\label{theo:intro_Con_embeds_HoStrat}
Let $\mathrm{Con}\subset \Strat$ be the full subcategory of triangulable conically stratified spaces, and $\simeq$ the relation of stratified homotopy. Then the induced functor 
\begin{equation*}
    \mathrm{Con}/{\simeq}\hookrightarrow\Ho\Strat
\end{equation*}
is a fully faithful embedding. \\
Let $\mathrm{Con}_P\subset \Top_P$ be the full subcategory of triangulable conically stratified spaces over $P$, and $\simeq_P$ be the relation of stratum preserving homotopy. Then the induced functor 
\begin{equation*}
    \mathrm{Con}_P/{\simeq_P}\hookrightarrow\Ho\Top_P
\end{equation*}
is a fully faithful embedding.
\end{theorem}
Again, this statement can also be strengthened to a fully faithful inclusion of infinity categories. 
In the proofs contained in this paper, we mainly focus on stratified objects over a fixed poset, $P$. In this context,
our approach is twofold. On one hand, we give a precise comparison of three model categories, and in particular of their classes of weak-equivalences. Those model categories are those of spaces and simplicial sets stratified over $P$, as well as the category of diagrams of simplicial sets, indexed by strictly increasing chains in $P$ (which we call regular flags). An object in this last category corresponds to the data of the strata and (generalized) homotopy links of a stratified object. This is summed up in the following theorem (\cref{Cor:Realization_Preserve_Weak_Equivalences,cor:Real_CP_reflects_WE,cor:CP_reflects_WE,theo:SingP_Characterize_WeakEquivalences,Cor:Diagram_Preserve_WE}).

\begin{theorem}\label{theo:Intro_triangle}
All functors in the following commutative diagram preserve and characterize weak-equivalences between arbitrary objects.
\begin{equation}\label{eq:Triangle_Functors}
    \begin{tikzcd}
         & \arrow[ld, "D_P", shift left = 4pt ]\sS_P \arrow[rd, "\RealP{-}"]& \\
         \Diag_P \arrow[ru, "C_P"] \arrow[rr]&& \arrow[ll, shift left = 4pt] \arrow[lu, "\Sing_P",  shift left = 4 pt ]\Top_P
    \end{tikzcd}
    \end{equation}
Furthermore, all of these functors descend to equivalences of the underlying homotopy categories.
\end{theorem}
Again, \cref{theo:Intro_triangle} strengthens to a statement about the corresponding infinity categories. 
\cref{theo:Intro_triangle} suggests a somewhat richer picture than what one has in the non-stratified case. Indeed, in the case where $P=\{*\}$, the categories $\Diag_P$ and $\sS_P$ are both canonically equivalent to $\sS$, turning this triangle into the familiar adjoint pair $\sS\leftrightarrow\Top$. In general however, the categories $\Diag_P$ and $\sS_P$ are very different, and the usefulness of thinking in terms of diagrams is illustrated in \cref{Section:Vertical,section:appendixB}, through the lens of vertical objects.

Furthermore, the fact that the functors in Diagram \eqref{eq:Triangle_Functors} preserve and reflect all weak-equivalences is a stronger property than one might expect. Indeed, four of those six functors are part of Quillen equivalences. One expects such a functor to only reflect weak-equivalences either between cofibrant or between fibrant objects. In fact, weak-equivalences in $\sS_P$ can even be defined as those maps that are sent to weak-equivalences in $\Diag_P$, \textbf{after} a suitable fibrant replacement. \cref{theo:Intro_triangle} then implies that this fibrant replacement is not needed. This gives some insight as to why the model structure on $\sS_P$ described by Henriques in \cite{Henriques} and the one studied here coincide (see \cref{rem:Henriques}).

On the other hand, we also have a more topological approach. Recall that pseudo-manifolds (and more generally conically stratified objects) can be described locally via the data of their strata and their (local) links. Indeed, in such a space, any stratum has a neighborhood that is locally homeomorphic to the product of the stratum and a cone on the (local) link. This is the phenomenon that we want to generalize to be able to reconstruct a stratified homotopy type, from the data of strata and homotopy links.

First note that the geometric definitions for the (global) link in terms of a boundary of a regular neighborhood readily extend to arbitrary stratified simplicial objects. This can be done either through the simplicial structure (by making use of the subdivision), or by working topologically. On the other hand, the homotopy links - which are nothing more than spaces of exit-paths - can be defined, as certain mapping spaces, for arbitrary stratified objects. Though, \textit{a priori}, the result might depend on the category in which we compute those mapping spaces.

Note also that while the historical definition of homotopy links, given by Quinn in \cite{quinn1988homotopically}, is only concerned with pairs of strata $[p<q]$, we study a generalized version of those homotopy links, which is defined for any increasing chain of strata $\I=[p_0<\dots <p_n]$. In this case, instead of exit-paths, the elements of the $\I$-th homotopy link of a stratified object can be thought of as "exit-simplices".

We show that, for a stratified simplicial set, all definitions of links and homotopy links coincide up to weak-equivalence (see \cref{theo:Link_Summary,theo:simHolink_v_Top_Hol,rem:Henriques}). This is summed up in the following theorem, where the geometric interpretations given for the different definitions hold when investigating pairs of strata, i.e. when $\I=[p<q]$. The two notions of links are illustrated in \cref{Fig:Link_And_Fiber}  and elements of the three different homotopy links are represented in \cref{Fig:Holinks}.

\begin{figure}[h]
\begin{tikzpicture}

\draw (0,0)--(3,0)--(-2,-2)--(-2,2)--(3,0);
\draw (-2,2)--(0,0)--(-2,-2);
\filldraw[red] (0,0) circle (2pt);
\filldraw[blue] (3,0) circle (2pt) (-2,-2) circle (2pt) (-2,2) circle (2pt);

\draw[shift={(7,0)}] (0,0)--(3,0)--(-2,-2)--(-2,2)--(3,0);
\draw[shift={(7,0)}] (-2,2)--(0,0)--(-2,-2);
\filldraw[shift={(7,0)},red] (0,0) circle (2pt);
\filldraw[shift={(7,0)},blue] (3,0) circle (2pt) (-2,-2) circle (2pt) (-2,2) circle (2pt);

\draw[shift={(7,0)}, thick, green] (1.5,0)--(0.33,0.66)--(-1,1)--(-1.33,0)--(-1,-1)--(0.33,-0.66)--(1.5,0);
\draw[shift={(7,0)}, dotted] (0.5,1)--(0.33,0.66)--(1.5,0)--(0.33,-0.66)--(0,0)--(-1.33,0)--(-1,1)--(0.33,0.66)--(0,0);
\draw[shift={(7,0)}, dotted] (-2,2)--(0.33,0.66)--(3,0);
\draw[shift={(7,0)}, dotted] (3,0)--(0.33,-0.66)--(0.5,-1);
\draw[shift={(7,0)}, dotted] (-2,-2)--(0.33,-0.66)--(-1,-1)--(-1.33,0)--(-2,-2);
\draw[shift={(7,0)}, dotted] (-2,0)--(-1.33,0)--(-2,2);

\filldraw[shift={(3.5,-4)}, opacity=0.3, blue] (3,0)--(-2,-2)--(-2,2)--(3,0);
\draw[shift={(3.5,-4)},blue](0,0)--(3,0) (0,0)--(-2,-2) (0,0)--(-2,2);
\filldraw[shift={(3.5,-4)},red] (0,0) circle (2pt);

\draw[shift={(3.5,-4)},thick, green] (1.5,0)--(-1,1)--(-1,-1)--(1.5,0);

\end{tikzpicture}

\caption{
 A simplicial set $K$ stratified over $P=\{0<1\}$, its simplicial link, $\Link{[0<1]}(K)$, and the topological link, $\Real{K}_b$, with $b$ in the interior of the interval $N(\{0<1\})$.
} \label{Fig:Link_And_Fiber}
\end{figure}
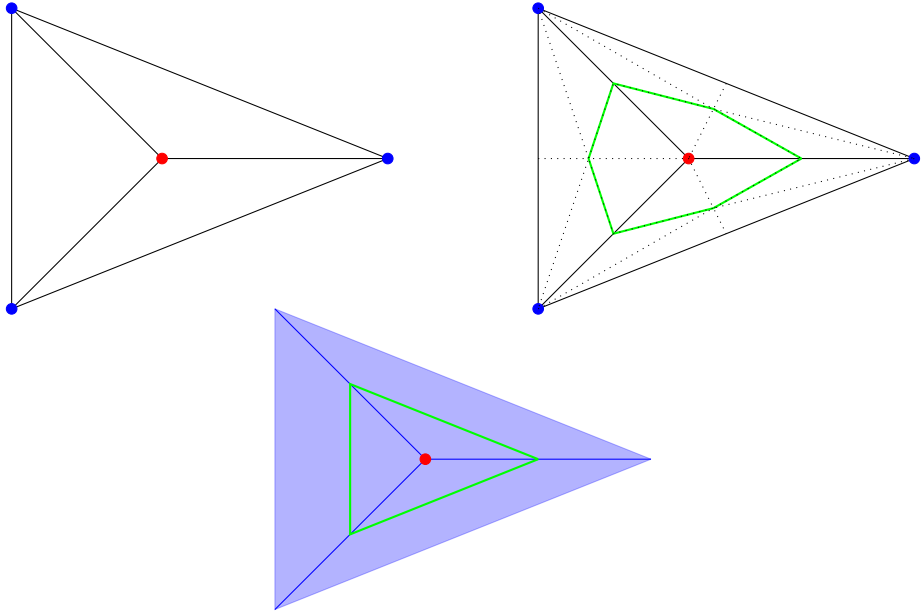

\begin{theorem}\label{theo:Intro_Holink}
Let $K\in \sS_P$ be a stratified simplicial set, $\I$ a flag and $b\in \Real{\Delta^{\I}}$ the barycenter. Then the following spaces are weakly equivalent:
\begin{itemize}
    \item $\Real{\Hol_{\I}(K)}$, a simplicial version of the space of exit-paths,
    \item $\Real{\Hol_{\I}(K^{\fib})}$,  same as the above, but computed on a fibrant replacement,
    \item $\Real{\Link{\I}(K)}$, the usual simplicial link, defined in terms of the subdivision,
    \item $\Real{K}_b$, a geometric notion of link, defined from the realization, 
    \item $\HolINP(\RealNP{K})$, a space of exit-paths with extra conditions,
    \item $\HolIP(\RealP{K})$, the space of exit-paths defined by Quinn.
\end{itemize}
\end{theorem}

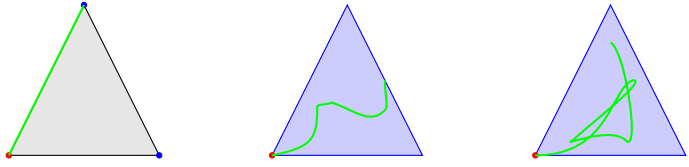
\begin{figure}[h]
\begin{tikzpicture}

\filldraw[black, opacity=0.1](0,0)--(1,2)--(2,0)--(0,0);
\draw (0,0)--(1,2)--(2,0)--(0,0);
\filldraw[red] (0,0) circle (1pt);
\filldraw[blue] (2,0) circle (1pt) (1,2) circle (1pt);
\draw[thick,green] (0,0)--(1,2);

\filldraw[shift={(3.5,0)},blue, opacity=0.2](0,0)--(1,2)--(2,0)--(0,0);
\draw[shift={(3.5,0)},blue] (0,0)--(1,2)--(2,0)--(0,0);
\filldraw[shift={(3.5,0)},red] (0,0) circle (1pt);
\draw[shift={(3.5,0)},thick,green] (0,0)..controls (0.5,0.1) and (0.6,0.2)..(0.6,0.6)..controls (0.6,0.7) and (0.6,0.65)..(0.8,0.7).. controls (1.1,0.6) and (1.3,0.4)..(1.5,0.6).. controls (1.55,0.65) and (1.5,0.8)..(1.5,1);

\filldraw[shift={(7,0)},blue, opacity=0.2](0,0)--(1,2)--(2,0)--(0,0);
\draw[shift={(7,0)},blue] (0,0)--(1,2)--(2,0)--(0,0);
\filldraw[shift={(7,0)},red] (0,0) circle (1pt);
\draw[shift={(7,0)},thick,green] (0,0).. controls (1,0) and (1,0.9)..(1.3,1).. controls (1.5,1) and (0.7,0.4)..(0.5,0.2).. controls (0.3,0.1) and (1,0.4)..(1.2,0.2).. controls (1.4,0) and (1.2,1.4)..(1,1.5);
\end{tikzpicture}
\caption{In green, from left to right, a $0$-simplex in $\Hol_{\I}(\Delta^{\J})$, an exit path in $\HolINP(\RealNP{\Delta^{\J}})$ and an exit-path in $\HolIP(\RealP{\Delta^{\J}})$, with $P=\{0<1\}$, $\I=[0<1]$, and $\J=[0\leq 1\leq 1]$}
\label{Fig:Holinks}
\end{figure}

While \cref{theo:Intro_triangle,theo:Intro_Holink} might look very different, they are closely related. Indeed, in the model categories we study, weak-equivalences are defined in terms of strata and homotopy links, in the sense that a map is a weak-equivalence if and only if it induces weak-equivalences between all strata and homotopy links. This means that a comparison between different notions of homotopy links immediatly induces a comparison between the corresponding class of weak-equivalences. The converse is also true, albeit in a less straightforward way. In fact, we prove \cref{theo:Intro_triangle,theo:Intro_Holink} in parallel throughout the paper, going back and forth between the topological and the homotopical point of view.

The article is organized as follows.

\cref{Section:Prelim} contains a recollection of all the necessary notions from \cite{douSimp,douteauEnTop,douteau2021stratified}, as well as definitions for all notions of link and homotopy link appearing in this article.

In \cref{Section:FSAE}, we prove that $\Sing_P$ characterizes all weak-equivalences. We do so by first proving in \cref{Section:Ex_P_FSAE} that the natural map $K\to \Ex_P(K)$ is a strong anodyne extension. In particular, this completes the proof that $\Exi_P$ is a fibrant replacement functor for stratified simplicial sets, see \cref{cor:Exi_P_fibrant_replacement}. The proof relies on the notion of strong anodyne extensions, introduced by Moss in \cite{MossSae} and already applied to stratified simplicial sets in \cite{douSimp}.

In \cref{Section:Real_Char_WE} we prove that the functor $\RealP{-}\colon \sS_P\to\Top_P$ characterizes weak-equivalences between arbitrary objects. We do so by showing that, for a stratified simplicial set $K$, $\RealP{\Link{\I}(K)}$, $\Real{K}_b$, $\HolINP(\RealNP{K})$ and $\HolIP(\RealP{K})$ are all weakly-equivalent. The key technical part, which consists mostly of point-set topology, is the proof that $\HolINP(\RealNP{K})$ and $\HolIP(\RealP{K})$ are weakly equivalent.

In \cref{Section:Hammock}, we use the results of \cref{Section:FSAE,Section:Real_Char_WE}  to show that the adjunctions $\RealStrat{-}\dashv \SingStrat$ and $\RealP{-}\dashv \Sing_P$ descend to equivalences between the homotopy categories of stratified simplicial sets and stratified spaces.
We then use this result to show that the homotopy theory of triangulable conically stratified objects embeds fully faithfully in the homotopy category of stratified spaces  (\cref{Cor:Conic_embed}). Finally, in  \cref{Section:Simp_Approx_Standalone}, we deduce a stratified simplicial approximation theorem from the comparison between the homotopy categories.

In \cref{Section:LastHolink}, we prove that for a stratified simplicial set $K$, the holinks $\Real{\Hol_{\I}(K)}$, $\Real{\Hol_{\I}(K^{\fib})}$ and $\HolIP(\RealP{K})$ are all weakly-equivalent. We do so by investigating a well-behaved classes of stratified spaces, vertically stratified CW-complexes and simplicial sets, in \cref{Section:Vertical} and proving a series of approximation theorems (\cref{prop:simplicial_vertical_approx,prop:vertical_approx}). Along the way, we show that the class of vertically stratified CW-complex, with their vertical maps and homotopies, model the homotopy category of stratified spaces (see \cref{Cor:CW_Model}). Finally, \cref{Section:Last_Approximation} gives the proof of \cref{prop:another_approximation} which is the key technical argument in comparing the homotopy links. 

In \cref{section:appendix}, we give an example of a particular stratum preserving map which obstructs the transport of the model structure from $\sS_P$ to $\Top_P$, and constrains the kind of model structures that one can hope to obtain on $\Top_P$. We also discuss how this example translates for other work on the subject such as \cite{nand2019simplicial} and \cite{haine2018homotopy}.

\cref{section:appendixB} expands on \cref{Section:Vertical}, making the relationship between vertical objects and diagrams more precise, and giving an alternate higher level proof of \cref{Cor:CW_Model}.

\section{Preliminaries}
\label{Section:Prelim}
We begin by recalling the framework of stratified homotopy theory used in \cite{douSimp,douteauEnTop, douteau2021stratified} as well as several of the central results of the theory.
These preliminaries are intended to be rather exhaustive, making the remainder of the paper generally accessible to the reader with general knowledge of model categories and abstract homotopy theory. While introducing the necessary language and notation, these preliminaries also serve to put the results discussed in \cref{sec:introduction} into rigorous context. The reader already familiar with the framework used in \cite{douSimp,douteauEnTop, douteau2021stratified} can safely skip most of those preliminaries, reading only \cref{subsec:gen_links} for the definitions of generalized links and homotopy links and \cref{subsec:failure_quillen,subsec:strat_sd_ex,subsec:recover_quil} for a precise recollection of the needed results.

First, we introduce the categories of stratified objects we are studying (\cref{subsec:poset,subsec:strat_spaces,subsec:stratified_simplicial_sets}). We equip these objects with notions of generalized homotopy links and introduce the basic relationships between the latter (\cref{Section:Classical_Links,subsec:gen_links,subsec:holink_and_diag}). These generalized homotopy links entail notions of weak equivalence on categories of stratified objects in question, which are integrated into model structures in \cref{subsec:mod_of_diag,subsec:mod_of_strat_space,subsec:mod_of_strat_ss}. Finally, the homotopical properties of the functors connecting the resulting model categories are discussed in \cref{subsec:failure_quillen,subsec:strat_sd_ex,subsec:recover_quil}.

\subsection{Partially ordered sets}\label{subsec:poset}
We begin by recalling a few general construction for partially ordered sets, which will serve as the indexing sets for stratifications.
\begin{definition}
Recall that a partially ordered set (or poset) is the data of a set $P$, equipped with an irreflexive, transitive and asymmetric relation, that is generically denoted $"{<}"$. We will also consider the weak relation, $"{\leq}"$, defined in the usual way via $p\leq q$ if $p<q$ or $p=q$. In this paper, a poset map is a map $\alpha\colon P\to Q$ such that if $p\leq p'$ in $P$, then $\alpha(p)\leq\alpha(p')$ in $Q$. The category of posets and poset maps is denoted $\Poset$.
\end{definition}

\begin{definition}
Given a poset, $P$, define the order topology on $P$ as follows. A subset $A\subset P$ is closed if and only if it is a downset, i.e. it fulfills the condition:
\begin{equation*}
    p\in A, q\leq p\Rightarrow q\in A.
\end{equation*}
\end{definition}
\begin{remark}
A basis of closed sets for this topology is given by the sets $A_p$, $p\in P$, defined as follows:
\begin{equation*}
    A_p=\{q\in P\mid q\leq p\}.
\end{equation*}
\end{remark}
\begin{remark}
Note that given a map of sets $\alpha\colon P\to Q$ between posets, the map $\alpha$ is a map of posets if and only if it is a continuous map between the posets equipped with their order topologies. This means that the order topology gives a fully-faithful functor $\Poset\hookrightarrow \Top$.
\end{remark}

\begin{definition}\label{def:Flags_DeltaP}
Let $P$ be a poset. 
\begin{itemize}
    \item A \define{flag} of $P$, $\J$, is a finite sequence in $P$, $\J=[p_0\leq\dots\leq p_n]$.
    \item A \define{regular flag} of $P$, $\I$, is a finite sequence in $P$ with no repeated entries $\I=[q_0<\dots<q_k]$. We will usually reserve the letter $\I$ for regular flags.
    \item Given a flag $\J=[p_0\leq \dots\leq p_n]$, we denote its underlying regular flag (obtained by deleting repeated entries) by $\{p_0\leq\dots\leq p_n\}$.
    \end{itemize}
A map of flags $f\colon [p_0\leq \dots\leq p_n]\to [q_0\leq\dots \leq q_m]$ is a non decreasing map $f\colon \{0,\dots,n\}\to\{0,\dots,m\}$ such that $q_{f(i)}=p_i$, for all $0\leq i\leq n$. Let $\Delta(P)$ be the category of flags and maps of flags and $R(P)\subset\Delta(P)$ be the full subcategory of regular flags.   
\end{definition}

\begin{definition}
\label{def:Nerve_Poset}
The \define{nerve of a poset} $P$ is the simplicial set $N(P)$ whose simplices are the flags of $P$. The $k$-th face operation is given by omitting the $k$-th entry in a flag while the $k$-th degeneracy is given by repeating the $k$-th entry. For $\J=[p_0\leq\dots\leq p_n]$, we will write $\Delta^{\J}$ for the associated $n$-simplex of $N(P)$, $\Delta^n\to N(P)$, seen as a simplicial set. 
\end{definition}

\begin{remark}
Note that, as was the case for spaces, a map of sets $\alpha\colon P\to Q$ between posets is a map of posets if and only if it extends to a simplicial map $N(\alpha)\colon N(P)\to N(Q)$. In other words, there is a fully-faithful embedding $N\colon\Poset\hookrightarrow\sS$.
\end{remark}

\begin{remark}
It is also classical to see a poset as a (small) category where there are no non-identity endomorphisms, and where there is at most one map between two objects. Starting from a set-theoretic poset $P$ one gets such a category by taking $P$ to be the set of objects, and having an arrow from $p\to q$ whenever $p\leq q$. Note that from this point of view, the nerve of the poset is nothing more than the nerve of the corresponding category. Similarly, the category of regular flags $R(P)$ is the subdivision of $P$, seen as a category.
\end{remark}

\begin{definition}\label{def:Varphi_P}
Let $P$ be a poset. Define $\varphi_P\colon \Real{N(P)}\to P$ as follows. For $\I=[p_0<\dots < p_n]$ a regular flag, using the identification \[\Real{\Delta^{\I}}\cong\{(t_0,\dots,t_n)\mid 0\leq t_i\leq 1,\ 0\leq i\leq n,\ \sum_it_i=1\}\subset \R^{n+1},\] we set
\begin{equation*}
    \varphi_P(t_0,\dots,t_n)=p_m, \ m=\max\{i\mid t_i\not=0\}.
\end{equation*}
\end{definition}

\subsection{Stratified spaces}\label{subsec:strat_spaces}
Before we define stratified spaces, a technical remark about the nature of the topological spaces considered is in order.
\begin{remark}
\label{rem:Delta_generated}
In this paper, $\Top$ stands for the category of $\Delta$-generated spaces (see \cite{duggerDelta} and \cite{Convenient}) and all continuous maps between them. This is to ensure that the category is locally presentable, which is needed to get the model structure described in \cref{theo:CMF_TopP} (see also \cref{rem:TopP_Transported}). It shouldn't be of much concern however, since all the usual examples of stratified spaces, as well as posets themselves, are $\Delta$-generated spaces. The implications of this choice of category for mapping spaces are discussed in \cref{rem:Compact_Open_Topology_Holinks}.
\end{remark}

In this paper, we consider the most general notion of stratified spaces among those commonly used. 

\begin{definition}\label{def:Stratified_Spaces}
A \define{stratified space} is the data of:
\begin{itemize}
    \item a topological space $X$,
    \item a poset $P$,
    \item a continuous map $\varphi_X\colon X\to P$, called the \define{stratification}.
\end{itemize}
By abuse of notation we will often refer to the above data just by the space $X$.
The \define{strata} of $X$ are the subspaces $\varphi_X^{-1}(p)=X^p\subset X$ for $p\in P$. A \define{stratified map} between two stratified spaces $X\to P$ and $Y\to Q$ is a commutative square of continuous maps
\begin{equation*}
    \begin{tikzcd}
    X
    \arrow{r}{f}
    \arrow[swap]{d}{\varphi_X}
    &Y
    \arrow{d}{\varphi_Y}
    \\
    P
    \arrow{r}{\bar{f}}
    &Q
    \end{tikzcd}
\end{equation*}
We will often refer to such a square just by the map $f$.
The category of stratified spaces and stratified maps is denoted $\Strat$.

Given a poset $P$, a \define{stratum preserving} map between two spaces stratified over $P$, $X$ and $Y$, is a continuous map $f\colon X\to Y$ such that the following triangle commutes:
\begin{equation*}
    \begin{tikzcd}
    X
    \arrow{rr}{f}
    \arrow[swap]{dr}{\varphi_X}
    && Y
    \arrow{dl}{\varphi_Y}
    \\
    &P \spaceperiod
    \end{tikzcd}
\end{equation*}
The category of spaces stratified over $P$ and stratum preserving maps is denoted $\Top_P$.
\end{definition}

\begin{remark}
\label{rem:Strat_Fibers_Over_Poset}
Consider the functor $\Strat\to\Poset$ sending $X\to P$ to $P$ and a stratified map $f$ to $\bar{f}$. This functor is a bifibration (see \cite[Section 2.1]{CagneMellies}), and the fiber over a poset $P$ is the subcategory $\Top_P$. This situation is pictured in the following diagram:
\begin{equation*}
    \begin{tikzcd}
    \Top_P
    \arrow[hookrightarrow]{r}
    \arrow{d}
    &\Strat
    \arrow{d}
    \\
    \{P,\Id_P\}
    \arrow[hookrightarrow]{r}
    &\Poset\spaceperiod
    \end{tikzcd}
\end{equation*}
\end{remark}

\begin{recollection}\label{rec:enriched_structures}
The category $\Top_P$ admits the structure of a simplicial category. The tensoring of $\Top_P$ for $X \in \Top_P$ and $S \in \sS$ is given by $X \times \Real{S}$ with the stratification induced by first projecting to the first component. By abuse of notation, this stratified space will be denoted $X \times S$. Consequently, for $X,Y \in \Top_P$ the simplicial hom is defined via
    \[ \Map(X,Y)_n = \Top_P(X \times \Delta^n, Y).\]
Furthermore, $\Top_P$ is also a tensored and cotensored category over $\Top$, again with the tensoring induced by taking the product and stratifying via projection to the stratified component. The mapping space
    \[ \mathcal C^0_P(X,Y) \subset \mathcal C^0(X,Y) \]
is obtained by considering the topology on $\Top_P(X,Y)$ induced by the inclusion into $\mathcal C^0(X,Y)$.
By construction, we obtain natural isomorphisms
    \[ \Sing \big ( \mathcal C^0_P( X,Y) \big ) \cong \Map(X,Y)\]
relating the simplicial and the topological enrichment of $\Top_P$.

Note that the category $\Strat$ also admits simplicial and topological enrichment, defined in a similar way, as well as internal hom-sets (see \cite[section 6.2]{nand2019simplicial}).
\end{recollection}

\begin{remark}\label{rem:Compact_Open_Topology_Holinks}

There is some ambiguity on the topology that should be assigned to the mapping space $\mathcal{C}^0_P(X,Y)$. In this work, the category $\Top$ stands for the category of $\Delta$-generated space, (see \cite{Convenient,duggerDelta}), and so $\mathcal{C}^0_P(X,Y)$ should be equipped with the $\Delta$-ification of the compact open topology. On the other hand, note that $\Sing(\mathcal{C}^0_P(X,Y))$ gives the same simplicial set whether one uses the compact open topology or its $\Delta$-ification, and in particular, the map $\colon\mathcal{C}^{0,\Delta}_P(X,Y)\to \mathcal{C}^{0,\text{c.o.}}_P(X,Y)$ given by the change in topology is a weak equivalence. Since we are only interested in the (weak) homotopy type of those objects, we will always consider $\mathcal{C}^0_P(X,Y)$ as equipped with the compact open topology. 

Furthermore, if $X$ is some compact stratified space and $K$ is a locally finite stratified simplicial set, then $\mathcal C^0_P(X,\RealP{K})$ will be metrizable. In particular, $\HolIP(\RealP{K})=\mathcal C^0_P(\RealP{\Delta^{\I}},\RealP{K})$, will be metrizable if $K$ is a locally finite stratified simplicial set (see \cref{def:Generalized_Holinks}). This observation will be usefull in the proof of \cref{theo:Strong_Holinks_are_Holinks}.
\end{remark}

The simplicial structure immediatly induces a notion of homotopy.

\begin{definition}\label{def:strat_homotopies_spaces}
Let $f,g\colon X\to Y$ be two maps in $\Strat$. A \define{stratified homotopy} between $f$ and $g$ is the data of a stratified map $H\colon X\times \Delta^1\to Y$ such that the restriction to $X\times \{0\}$ and $X\times \{1\}$ are equal to $f$ and $g$ respectively. We denote this relation by $f \simeq_{\Strat} g$.
A \define{stratified homotopy equivalence} is a stratified map $f\colon X\to Y$ such that there exists a stratified map $g\colon Y\to X$, such that $g\circ f \simeq_{\Strat} \Id_X$ and $f\circ g \simeq_{\Strat} \Id_Y$. For $X,Y\in \Top_P$, we will write $[X,Y]_P$ for the
set of stratified homotopy classes of stratum preserving maps between $X$ and $Y$.
\end{definition}

\begin{remark}
One can rephrase the above definition as follows: given two stratified spaces $X\to P$ and $Y\to Q$ and two maps between them, $f$ and $g$, a stratified homotopy between $f$ and $g$ corresponds to a commutative square in $\Top$:
\begin{equation*}
    \begin{tikzcd}
    X\times[0,1]
    \arrow[swap]{d}{\varphi_X\circ\pr_X}
    \arrow{r}{H}
    &Y
    \arrow{d}{\varphi_Y}
    \\
    P
    \arrow{r}{\bar{H}}
    &Q \spaceperiod
    \end{tikzcd}
\end{equation*}
There are several things to note here. 
\begin{itemize}
    \item If two maps are stratified homotopic, the underlying maps of spaces are homotopic, since $H$ is also a homotopy in the usual sense.
    \item Since the homotopy is constant at the level of posets, one must have $\bar{H}=\bar{f}=\bar{g}$. In particular, any stratified homotopy between stratum preserving maps is given by a stratum preserving homotopy, $H$. For this reason we will not make a distinction between stratum-preserving and stratified homotopies, and just use "stratified homotopy" as a generic term.
    \item Since the homotopy is constant at the poset level, any stratified homotopy equivalence must be over an isomorphism of posets.
\end{itemize}
\end{remark}

We will also make use of a stronger notion of stratified spaces, for which, in addition to a decomposition into strata, the stratification also encodes information about the neighborhoods of strata.

\begin{definition}
A \define{strongly stratified space} is the data of
\begin{itemize}
    \item a space $X$,
    \item a poset $P$,
    \item a continuous map $\varphi_X\colon X\to \Real{N(P)}$ called the \define{(strong) stratification}.
\end{itemize}
\define{Stratum-preserving maps}  between strongly stratified spaces are defined analogously to \cref{def:Stratified_Spaces} as commutative triangles
\begin{equation*}
    \begin{tikzcd}
    X
    \arrow{rr}{f}
    \arrow[swap]{dr}{\varphi_X}
    &&Y
    \arrow{dl}{\varphi_Y}
    \\
    & \Real{N(P)}
    &\spaceperiod
    \end{tikzcd}
\end{equation*}
This leads to the category $\Top_{N(P)}$, of spaces strongly stratified over $P$ and stratum preserving maps. By abuse of language, we will call the objects of $\Top_{N(P)}$ spaces stratified over $N(P)$.
\end{definition}

\begin{recollection}\label{rec:circ_varphiP}
The map $\varphi_P\colon \Real{N(P)}\to P$, of \cref{def:Varphi_P} induces a functor:
\begin{align*}
    \varphi_P\circ -\colon \Top_{N(P)}&\to \Top_P\\
    (X,\varphi_X\colon X\to \Real{N(P)})&\mapsto (X,\varphi_P\circ\varphi_X\colon X\to P).
\end{align*}
This functor admits a right-adjoint, $-\times_P\Real{N(P)}\colon \Top_P\to\Top_{N(P)}$, defined on objects as the following pull-back:
\begin{equation*}
    \begin{tikzcd}
    Y\times_P\Real{N(P)}
    \arrow{r}
    \arrow{d}
    &Y
    \arrow{d}{\varphi_Y}
    \\
    \Real{N(P)}
    \arrow{r}{\varphi_P}
    &P\spacecomma
    \end{tikzcd}
\end{equation*}
where the strong stratification on $Y\times_P\Real{N(P)}$ is given by the projection on the second factor.
Note that $\Top_{N(P)}$ also admits a simplicial and topological enrichment, as in \cref{rec:enriched_structures}, which analogously to the stratified setting, induces a notion of \define{strongly stratified homotopies}. These enrichment structures are compatible with the adjunction $\varphi_P\circ-\dashv -\times_P\Real{N(P)}$. In particular, this adjoint pair preserves the respective notions of stratified homotopies.
\end{recollection}

\subsection{Stratified simplicial sets}\label{subsec:stratified_simplicial_sets}

We will also consider a simplicial version of stratified spaces, the stratified simplicial sets.

\begin{definition}
A \define{stratified simplicial set} is the data of:
\begin{itemize}
    \item a simplicial set $K$,
    \item a poset $P$,
    \item a simplicial map $\varphi_K\colon K\to N(P)$, called the \define{stratification}.
\end{itemize}
By abuse of notation, we will often refer to the above data just by the simplicial set $K$. A \define{stratified map} between two stratified simplicial sets $K\to N(P)$, $L\to N(Q)$ is a commutative square of simplicial maps.
\begin{equation*}
    \begin{tikzcd}
    K
    \arrow{r}{f}
    \arrow[swap]{d}{\varphi_K}
    & L
    \arrow{d}{\varphi_L}
    \\
    N(P)
    \arrow{r}{N(\bar{f})}
    &N(Q)
    \end{tikzcd}
\end{equation*}
The category of stratified simplicial sets and stratified maps is denoted $s\Strat$.

Given a poset $P$, a \define{stratum preserving map} between two simplicial sets stratified over $P$, $K$ and $L$, is a simplicial map $f\colon K\to L$, such that the following triangle commutes:
\begin{equation*}
    \begin{tikzcd}
    K
    \arrow{rr}{f}
    \arrow[swap]{dr}{\varphi_K}
    && L
    \arrow{dl}{\varphi_L}
    \\
    &N(P)\spaceperiod
    \end{tikzcd}
\end{equation*}
The category of simplicial sets stratified over $P$ is denoted $\sS_P$.
\end{definition}

\begin{remark}\label{rem:sSetP_Presheaf_category}
The category $\sS_P$ can alternatively be defined as a presheaf category. Indeed, there is an equivalence of categories $\sS_P\cong \Fun(\Delta(P)^{\op},\Set)$ (see \cref{def:Flags_DeltaP}). Note that under this identification, a flag $\J=[p_0\leq\dots\leq p_n]$ is sent to the stratified simplicial set $\varphi_{\J}\colon\Delta^n\to N(P)$, satisfying $\varphi_{\J}(k)=p_k$, where $k$ stands for the $k$-th vertex of $\Delta^n$. We will denote this stratified simplicial set $\Delta^{\J}$, and call all such objects stratified simplices. See \cite[Proposition 1.3]{douSimp}. Note that this is consistent with the convention of \cref{def:Nerve_Poset}.
\end{remark}

\begin{recollection}
The category $\sS_P$ admits the structure of a simplicial category. The tensoring between $K\in \sS_P$ and $S\in \sS$ is given by the product $K\times S$ with the stratification given by the composition $K\times S\to K\to N(P)$. For $K,L\in \sS_P$ the simplicial hom is defined via
\begin{equation*}
    \Map(K,L)_n=\sS_P(K\times\Delta^n,L).
\end{equation*}
\end{recollection}

Just as for stratified spaces, the simplicial structure induces a notion of homotopy equivalences.

\begin{definition}\label{def:Strat_Homotopies_simplicial}
Let $f,g\colon K\to L$ be two maps in $s\Strat$. An \define{elementary stratified homotopy} between $f$ and $g$ is the data of a stratified map $H\colon K\times\Delta^1\to L$, whose restrictions to $K\times \{0\}$ and $K\times \{1\}$ give $f$ and $g$ respectively. The maps $f$ and $g$ are said to be \define{stratified homotopic} if there exists a finite sequence of maps $f_i$, $0\leq i\leq n$ such that $f_0=f$, $f_n=g$, and for all $0\leq i\leq n-1$, $f_i$ and $f_{i+1}$ are related by an elementary stratified homotopy. A \define{stratified homotopy equivalence} is a stratified map $f\colon K\to L$ such that there exist a stratified map $g\colon L\to K$ such that $g\circ f$ and $f\circ g$ are respectively stratified homotopic to $\Id_K$ and $\Id_L$. For $K,L\in \sS_P$, we will write $[K,L]_P$ for the set of stratum preserving maps between $K$ and $L$ up to stratified homotopy.
\end{definition}

\begin{recollection}
\label{rec:Stratified_Real_Strat}
The categories of stratified simplicial sets and stratified spaces are connected through a $\Real{-}\dashv \Sing$ style adjunction, just as in the unstratified context. It can be described explicitely as follows. Given some fixed poset $P$, we define a realization style functor
\begin{align*}
    \RealNP{-}\colon\sS_P&\to\Top_{N(P)}\\
    \left(\varphi_K\colon K\to N(P)\right)&\mapsto \left(\Real{\varphi_K}\colon \Real{K}\to \Real{N(P)}\right).
\end{align*}
Composing with the functor $\varphi_P\circ-\colon \TopNP\to\Top_P$ (see \cref{rec:circ_varphiP}) gives a functor $\RealP{-}\colon \sS_P\to\Top_P$. Both functors admit right adjoints, $\Sing_P\colon\Top_P\to\sS_P$ and $\Sing_{N(P)}\colon \TopNP\to\sS_P$, defined as follows. Given a stratified space $\varphi_X\colon X\to P$, its \define{stratified singular simplicial set}, $\Sing_P(X)$, is given by the pullback
\begin{equation*}
    \begin{tikzcd}
    \Sing_P(X)
    \arrow{r}
    \arrow{d}
    &\Sing(X)
    \arrow{d}{\Sing(\varphi_X)}
    \\
    N(P)
    \arrow[hookrightarrow]{r}
    &\Sing(P)\spacecomma
    \end{tikzcd}
\end{equation*}
where the bottom map is the adjoint map to $\varphi_{P}$ from \cref{def:Varphi_P} under $\Real{-} \dashv \Sing$.
Equivalently, as a presheaf on $\Delta(P)$, $\Sing_P(X)$ can be constructed via
\begin{equation*}
    \Sing_P(X)_{\J}=\Top_P(\RealP{\Delta^{\J}},X).
\end{equation*}
Note that the collection of adjunctions $\RealP{-}\dashv \Sing_P$ extend to an adjunction \[\RealStrat{-}\colon\sStrat\leftrightarrow\Strat\colon \SingStrat.\] 
The definition of $\Sing_{N(P)}$ is entirely analogous.
\end{recollection}

Note also that the adjunctions $\RealP{-}\dashv \Sing_P$ and $\RealStrat{-}\dashv\SingStrat$ are compatible with the simplicial structures (see \cite[Proposition 4.9]{douSimp}).

\begin{proposition}
The adjunctions $\RealP{-}\dashv \Sing_P$ and $\RealStrat{-}\dashv\SingStrat$ are simplicial, and preserve stratified homotopies (see \cref{def:strat_homotopies_spaces,def:Strat_Homotopies_simplicial}).
\end{proposition}

\subsection{Classical links and homotopy links}\label{Section:Classical_Links}
Recall that pseudo-manifolds are particular kinds of stratified spaces whose strata are manifolds satisfying gluing conditions. Those conditions can be expressed through the use of some smooth structures, as in \cite{Whitney} or \cite{Thom}. Or they can be expressed in a purely topological fashion by asking that all points should have neighborhood homeomorphic to cones. We recall the more recent and more general definition of \define{conically stratified spaces}, which will be sufficient for our purpose. Note that pseudo-manifolds are examples of conically stratified spaces.

\begin{definition}[{\cite[Definition A.5.5]{HigherAlgebra}}]
Let $P$ be a poset, and define $c(P)=\{*\}\coprod P$, with $*<p$ for all $p\in P$. Let $\varphi_L\colon L\to P$ be a stratified space. Its (stratified) \define{cone} is the stratified space
\begin{align*}
    c(L)=L\times [0,1)/L\times\{0\}&\to c(P)\\
    (l,t)&\mapsto \left\{\begin{array}{cc}
    \varphi_L(l) & \text{, if $t>0$}  \\
    *     & \text{, if $t=0$.} 
    \end{array}\right.
\end{align*}
equipped with the teardrop topology, (see \cref{rem:TeardropTop}).
A stratified space $X\to P$ is \define{conically stratified} if, for every $x\in X$, in some stratum $X^p$, there exists\begin{itemize}
    \item a stratified space $L\to P_{>p}=\{q\in P\mid q>p\}$, the \define{(local) link at }$x$,
    \item an open neighborhood $x\in U\subset X$,
    \item a space $Z$,
    \item and a stratified homeomorphism $Z\times c(L)\cong U$, over the poset identification $c(P_{>p})\cong P_{\geq p}\subset P$.
\end{itemize} 
\end{definition}
\begin{remark}\label{rem:TeardropTop}
In the above definition, one considers the teardrop topology on the cone $c(L)$ (see \cite[Definition A.5.3]{HigherAlgebra}). Note that for a compact (stratified) space, $L$, it coincides with the quotient topology. Since pseudo-manifolds are always assumed to have compact links, this subtlety does not come into play when studying those objects. However, when studying more general examples of conically stratified spaces, this distinction is crucial. In fact, we will use throughout a result of Lurie \cite[Theorem A.6.4]{HigherAlgebra} (see \cref{theo:Conical_Fibrant}) whose proof relies on the properties of the teardrop topology.
\end{remark}
An even more general notion, more suited for the study of stratified homotopies, was introduced by Quinn in \cite{quinn1988homotopically}, that of homotopically stratified sets. For those objects, instead of considering local links, one considers pairwise homotopy links, which provide a homotopy theoretic global approach to the former.
\begin{definition}[{\cite[Definition 2.1]{quinn1988homotopically}}]
\label{def:Quinn_homotopy_links}
Let $X\to P$ be a stratified space and $p<q\in P$. The \define{homotopy link} of the $p$-stratum in the $q$-stratum is the topological space
\begin{equation*}
    \hol_{p<q}(X)=\{\gamma\colon [0,1]\to X\mid \gamma(0)\in X^p,\ \gamma(t)\in X^q,\  \forall t>0\},
\end{equation*}
whose topology is induced by the inclusion $\hol_{p<q}(X)\subset C^0([0,1],X)$.
\end{definition}

Quinn then defined (see \cite[Definition 3.1]{quinn1988homotopically}) \define{homotopically stratified sets} as those (metric) stratified spaces satisfying:
\begin{itemize}
    \item evaluation at $0$, $\ev_0\colon\hol_{p<q}(X)\to X^p$ is a fibration in $\Top$, for all $p<q\in P$,
    \item inclusions $X^p\hookrightarrow X^p\cup X^q$ are tame, for all $p<q\in P$.
\end{itemize}

\begin{remark}
\label{rem:conical_pman_homotopically_stratified}
In this paper, we will not be overly concerned with the distinctions between those classes of spaces. As we will see in \cref{theo:Conical_Fibrant}, pseudo-manifolds, conically stratified spaces, and homotopically stratified sets share a common homotopical property. It is this property - and not their geometrical definitions directly - that is leveraged in the proof of \cref{theo:intro_Con_embeds_HoStrat}. We only state results about conically stratified spaces - since it is both a large class, and reasonably easy to define - but \cref{theo:intro_Con_embeds_HoStrat} also holds for pseudo-manifolds or homotopically stratified sets.
\end{remark}

Homotopy links are known to characterize the stratified homotopy equivalences between homotopically stratified sets, as was shown in \cite{miller2013}.

\begin{theorem}[{\cite[Theorem 6.3]{miller2013}}]
\label{theo:Miller_Theorem} A stratified map $f\colon X\to Y\in \Top_P$, between two homotopically stratified sets is a stratified homotopy equivalence if and only if, the induced maps of spaces
\begin{itemize}
    \item $X^p\to Y^p$,
    \item $\hol_{p<q}(X)\to \hol_{p<q}(Y)$,
\end{itemize}
are homotopy equivalences, for all $p \in P$ and $q \in P$ with $p <q$.
\end{theorem}

It is also possible to get a similar statement, while only asking the maps between strata and homotopy links to be \define{weak}-equivalences. In this case, one has to ask that $X$ and $Y$ are realizations of stratified simplicial sets, (see \cite[Theorem 5,4]{douSimp}). In any case, these results suggest that homotopy links are a good starting point to build a general homotopy theory of stratified spaces. On the other hand, consider the following example.

\begin{example}\label{ex:links_not_enough}
Consider the stratified simplex $\Delta^{\J}$ for the (regular) flag $\J = [ p_0 < p_1 < p_2]$ in $P$ as well as its boundary $\partial \Delta^{\J} \subset \Delta^{\J}$ with the induced stratification (see \cref{fig:Simplex_And_Boundary}). The inclusion
    \[ i\colon\RealP{\partial \Delta^{\J}} \hookrightarrow \RealP{\Delta^{\J}}\]
    induces homotopy equivalences on all strata and pairwise homotopy links, since for both spaces all of these are contractible. However, $i$ is clearly not a stratified homotopy equivalence since if we forget the stratification the underlying map of topological spaces is not even a weak equivalence. Note that this is not in contradiction to \cref{theo:Miller_Theorem} as $\RealP{\partial \Delta^{\J}}$ is not a homopically stratified set. Indeed, the natural map $\HolIP(\RealP{\partial \Delta^{\J}}) \to (\RealP{\partial \Delta^{\J}})_{p_1}$, for $\I = \{ p_1 < p_2 \}$, is not a fibration as it has nonempty source but is not surjective.
    If we want the stratified homotopy type to be finer than the homotopy type of the underlying space, $\partial\Delta^{\J}$ and $\Delta^{\J}$ need to have a distinct stratified homotopy type, which can not be defined through strata and homotopy links alone. \\
    On the other hand, note that the homotopy links of \cref{def:Quinn_homotopy_links} can be equivalently defined as $\HolIP(X)=\mathcal C^0_P(\RealP{\Delta^{\I}},X)$, for $\I=\{p<q\}$ a regular flag of $P$. Note that this definition readily extends to \define{arbitrary} regular flags $\I=\{p_0<\dots<p_n\}$. In particular, for the example at hand, one checks that the inclusion $\HolIP(\RealP{\partial\Delta^{\J}})\to \HolIP(\RealP{\Delta^{\J}})$ is not a weak equivalence for $\I=\J$. Indeed, the domain is empty, as there exists no stratified maps $\RealP{\Delta^{\J}}\to\RealP{\partial\Delta^{\J}}$, while the codomain contains the identity map. This illustrates the necessity to consider \define{generalized homotopy links}. Indeed, at least when restricting to stratified spaces that admit (stratified) triangulations, stratified maps inducing weak equivalences between all strata and generalized homotopy links are weak equivalence of the underlying spaces (see \cref{ex:counter_ex_forgetful,rem:forget_preserves_we_triangulable}). 
\end{example}
\begin{figure}[h]
\begin{tikzpicture}
\filldraw[black, opacity=0.1](0,0)--(1,2)--(2,0)--(0,0);
\draw (0,0)--(1,2)--(2,0)--(0,0);
\filldraw[red] (0,0) circle (1pt);
\filldraw[blue](1,2) circle (1pt);
\filldraw[green]  (2,0) circle (1pt) ;

\draw[shift={(5,0)}] (0,0)--(1,2)--(2,0)--(0,0);
\filldraw[shift={(5,0)},red] (0,0) circle (1pt);
\filldraw[shift={(5,0)},blue](1,2) circle (1pt);
\filldraw[shift={(5,0)},green]  (2,0) circle (1pt) ;

\end{tikzpicture}
\caption{The simplex $\Delta^{\J}$ and its boundary $\partial\Delta^{\J}$, with $\J=[p_0<p_1<p_2]$. The first has a non-empty $\J$-holink, while $\Hol_{\J}(\partial\Delta^{\J})=\emptyset$. For all other flags $\I\not =\J$, $\Delta^{\J}$ and $\partial\Delta^{\J}$ have equivalent $\I$-holinks.}
\label{fig:Simplex_And_Boundary}
\end{figure}
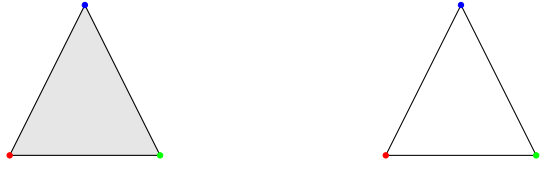

\subsection{Generalized links and homotopy links}\label{subsec:gen_links}

We now turn to the more general notion of homotopy links, that was suggested in \cref{ex:links_not_enough}. Note that we will simply call those \define{homotopy links}, instead of generalized homotopy links, and refer to the homotopy links of Quinn as \define{classical homotopy links}. Since the latter is just a particular case of the former, as pointed out in \cref{ex:links_not_enough}, this should not cause too much confusion.

\begin{definition}
\label{def:Generalized_Holinks}
Let $K\in \sS_P$, $X\in \Top_{N(P)}$, $Y\in \Top_P$ and $\I=[p_0<\dots<p_n]$ a regular flag. We define $\I$-th \define{homotopy links}, in each of the three categories, as follows:
\begin{itemize}
    \item $\Hol_{\I}(K)=\Map(\Delta^{\I},K)\in\sS$,
    \item $\HolINP(X)=\C^0_{N(P)}(\RealNP{\Delta^{\I}},X)\in\Top$,
    \item $\HolIP(Y)=\C^0_P(\RealP{\Delta^{\I}},Y)\in\Top$.
\end{itemize}
When there is a need to distinguish between them, the first two will be respectively called the \define{simplicial homotopy link} and the \define{strong homotopy link}, while the last one will always just be called the homotopy link.
\end{definition}

\begin{remark}
There is a simple mnemonic underlying the choice of font for $\Hol_{\I}$, $\HolINP$ and $\HolIP$. Vertices of $\Hol_{\I}(K)$ are given by simplicial maps from a stratified simplex into $K$. Meanwhile, elements of $\HolINP(X)$ are given by strongly stratum preserving maps from (the realization) of a stratified simplex into a space $X$ and elements of $\HolIP(Y)$ by stratum preserving ones. Hence, the degree of rigidity required from these maps decreases in the order of listing, see \cref{Fig:Holinks}. This is reflected by the font used for the "Ho" part of the respective notation.
\end{remark}

\begin{remark}\label{Rem:Comparing_Holink_Intro}
Let $X \in \Top_P$ be a $P$-stratified space and $\I$ a regular flag. Then, by \cref{rec:enriched_structures} and \cref{rec:circ_varphiP} we have natural isomorphisms:
    \[ \Hol_{\I}\big (\Sing_P(X) \big ) \cong \Sing \big ( \HolIP(X) \big) \cong \Sing \big( \HolINP(X \times_{P} \Real{N(P)}) \big).\]
\end{remark}

Before considering a converse statement with $K\in \sS_P$, we first introduce a generalization of the classical notion of links of sub-complex of a simplicial complex.

\begin{definition}
\label{def:Simp_Link}
Let $ K \in \sS_P$ be a stratified simplicial set, and $\I$ a regular flag. \define{The $\I$-th simplicial link of $K$}, $\Link{\I}{(K)}$, is defined via the following pullback diagram:
\begin{center}
    \begin{tikzcd}
        \Link{\I}{(K)} \arrow[r, hook] \arrow[d] & \sd{(K)} \arrow[d, "\sd(\varphi_K)"] \\
        \Delta^{0} \arrow[r, "i_{\I}"] & \sd(N(P)) \spacecomma 
    \end{tikzcd}
\end{center}
where $\sd$ is the usual barycentric subdivision, and $i_{\I}$ is the map sending the point to the unique vertex in $\sd(N(P))$ corresponding to the regular flag $\I$. This construction induces a functor
\begin{equation*}
    \Link{\I}{ }:\sS_{P} \to \sS.
\end{equation*}
\end{definition}
\begin{remark}
In the case when $K \in \sS_P$ comes from a simplicial complex and $P = \I = \{p_0 < p_1\}$, the simplicial set $\Link{\I}{(K)}$ is the simplicial link of the $p_0$ stratum in $K$. In other words, it is the complex obtained by taking the boundary of the simplicial complex given by simplices which contain a vertex in the $p_0$ stratum. As already noted in \cite{quinn1988homotopically}, from this it follows that $\Real{\Link{\I}{(K)}}$ and $\HolIP (\RealP{K})$ are weakly equivalent. One of the key result in this paper (\cref{theo:Intro_Holink}), is that this generalizes to arbitrary regular flags and stratified simplicial sets, and to all the notions of homotopy links considered here. 
\end{remark}

\begin{definition}
\label{def:Geom_Link}
Let $\varphi_K\colon K\to N(P)$ be a stratified simplicial set, $\I$ a regular flag and $b\in \Real{\Delta^{\I}}$ be the barycenter. Let $\Real{K}_b\subset \Real{K}$ be the subspace given by $\Real{\varphi_K}^{-1}(b)$. In other words, $\Real{K}_b$ is the space defined via the following pullback:
\begin{equation*}
    \begin{tikzcd}
             \Real{K}_b
             \arrow[hookrightarrow]{r}
             \arrow{d}
             & \Real{K}
             \arrow{d}{\Real{\varphi_K}}
             \\
             \{b\}
             \arrow[hookrightarrow]{r}
             &\Real{N(P)} \spaceperiod
    \end{tikzcd}
\end{equation*}
\end{definition}

\begin{remark}\label{rem:fibers_are_link}

 One can also think of $\Real{K}_b$ as a \define{geometric link} (see, \cite[Theorem 2.10]{MilnorSingPoint}). Consider the case where $P=\{0<1\}=\I$, with the identification $\Real{N(P)}\simeq [0,1]$, and $K\in \sS_P$ is a triangulation of some pseudo-manifold with an isolated singularity. Then, $\RealNP{\varphi_X}^{-1}([0,1[)$ is a neighborhood of the singular point, and must be isomorphic to some cone, $c(L)$, for some compact manifold $L$. The subspace $\Real{K}_b=\RealNP{\varphi_X}^{-1}(\frac{1}{2})$ is then homeomorphic to the subspace $L\times\{\frac{1}{2}\}$, or in other words, to the link.  See \cref{Fig:Link_And_Fiber} for a side by side comparison between $\Link{\I}(K)$ and $\Real{K}_b$.
 \end{remark}

Let $K\in\sS_P$ be a stratified simplicial set, $\I$ a regular flag and $b\in\Real{\Delta^\I}$ the barycenter. We will be interested in the following maps:
\begin{equation}\label{eq:Diagram_holinks_link_prelim}
    \begin{tikzcd}
             \Real{\Hol_{\I}(K)}
             \arrow{r}{1}
             &\HolINP(\RealNP{K})
             \arrow{r}{2}
             \arrow{d}{3}
             &\HolIP(\RealP{K})
             \\
             &\Real{K}_b
             \\
             &\Real{\Link{\I}(K)}
             \arrow[swap]{u}{*}
    \end{tikzcd}
\end{equation}
The map labeled $*$ is a non-natural homeomorphism and will be defined at the begining of \cref{Section:Real_Char_WE}, \cref{prop:Links_and_Links}. On the other hand the other maps are both easily defined, and seen to be natural. Using the structure described so far, the first map is equivalently given as a map $\Map(\Delta^{\I},X)=\Hol_{\I}(X)\to\Sing(\HolINP(\RealNP{X}))=\Map(\RealNP{\Delta^\I},\RealNP{X})$. It is the map sending a map $\Delta^{\I}\times\Delta^n\to X$ to its realization. The second map just corresponds to the inclusion $\mathcal C^0_{N(P)}(\RealNP{\Delta^\I},\RealNP{X})\hookrightarrow \mathcal C^0_{N(P)}(\RealP{\Delta^\I},\RealP{X})$. Finally, the third map is given by evaluating maps of the form $\RealNP{\Delta^\I}\to\RealNP{X}$ at $b\in \Real{\Delta^\I}$. \cref{Section:Real_Char_WE} is devoted to proving that the second and the third map in Diagram \eqref{eq:Diagram_holinks_link_prelim} are weak equivalences, while \cref{Section:LastHolink} handles the first map (or rather, equivalently, it deals with the composition of the first and the second).

\subsection{Homotopy links and diagrams}\label{subsec:holink_and_diag}

\begin{recollection}\label{rec:Diagram_Functors}
Let $K$ be a stratified simplicial set, and $\I\subset \I'$ two regular flags. The inclusion $\Delta^{\I}\to\Delta^{\I'}$ induces a map $\Hol_{\I'}(K)\to\Hol_{\I}(K)$. In fact, the data of the holinks can be synthesized as the following diagram
\begin{align*}
    R(P)^{\op}&\to\sS\\
    \I &\mapsto \Hol_\I(K),
\end{align*}
where $R(P)$ is the category of regular $P$-flags (see \cref{def:Flags_DeltaP}). We write $\Diag_P=\Fun(R(P)^{\op},\sS)$ for the category of such diagrams. Note that $\Diag_P$ can equivalently be described as the category of presheaves over $R(P)\times\Delta$, i.e. $\Diag_P\cong \Fun((R(P)\times\Delta)^{\op},\Set)$, though we will mostly consider objects of $\Diag_P$ as diagrams of simplicial sets.

We have a functor
\begin{align*}
    D_P\colon \sS_P&\to \Diag_P\\
    K&\mapsto \big(\I\mapsto \Hol_{\I}(K)\big)
\end{align*}
The functor admits a left adjoint, $C_P\colon \Diag_P\to \sS_P$. Since $\Diag_P$ is a presheaf category over $R(P)\times \Delta$, it is enough to define $C_P$ on pairs $(\I,n)\in R(P)\times \Delta$ and then take the left Kan extension along the Yoneda embedding $R(P)\times\Delta\hookrightarrow\Diag_P$. On $R(P)\times\Delta$, one defines $C_P(\I,n)=\Delta^\I\times\Delta^n$. A more explicit definition can be found in \cite[Definition 2.5]{douteauEnTop}. Similarily, there are functors $D^{\Top}_P\colon \Top_P\to\Diag_P$, as well as $D^{\Top_{N(P)}}_P\colon \Top_{N(P)}\to\Diag_P$ defined as follows
\begin{align*}
    D^{\Top}_P\colon \Top_P&\to\Diag_P\\
    Y&\mapsto \big(\I\mapsto\Sing(\HolIP(Y))\big)\\
    D^{\Top_{N(P)}}_P\colon \Top_{N(P)}&\to\Diag_P\\
    X&\mapsto \big(\I\mapsto\Sing(\HolINP(X))\big)
\end{align*}
$D_P^{\Top}$ and $D^{\Top_{N(P)}}_P$ both admit left adjoints, $C_P^{\Top}\colon\Diag_P\to\Top_P$, and $C_P^{\Top_{N(P)}}\colon\Diag_P\to\Top_{N(P)}$. Note that the six functors defined above satisfy several relations, which can be summed up by the following diagram of adjunctions:

\begin{equation}   \label{diag:factorized_triangle}
\begin{tikzcd}
         && \sS_P
         \arrow[lldd, "D_P"] 
         \arrow[rrdd, "\RealP{-}",  bend left = 60] 
         \arrow[rd, "\RealNP{-}",  shift left = 4pt]
         \\ 
         &&& \TopNP 
         \arrow[rd, shift left = 4pt, "\varphi_P \circ -"]
         \arrow[lu, "\Sing_{N(P)}"] 
         \arrow[llld] 
         \\
         \Diag_P 
         \arrow[rrru, shift left = 4 pt] 
         \arrow[rruu, "C_P", shift left = 4pt ] 
         \arrow[rrrr,"C_P^{\Top}"]
         &&&&
         \Top_P
         \arrow[lluu, "\Sing_P"{xshift = 10 pt, yshift = -5pt},  shift left = 4 pt, bend right =  60 ] \arrow[llll, shift left = 4pt,"D_P^{\Top}"] 
         \arrow[lu, "-\times_{P}\Real{N(P)}"]
    \end{tikzcd}
\end{equation}
In the above diagram, all pairs of parallel maps are adjoint pairs. Furthermore if one restrict to left adjoints, then one gets a commutative diagram (up to natural isomorphisms), and similarily for the right adjoints. Given the compatibility between all those functors, we will mostly omit the superscript $\Top$ and $\Top_{N(P)}$.
\end{recollection}

The values of the functors $D_P$ on stratified objects compile all the information about their (generalized) homotopy links. This suggests that diagrams form a reasonable basis to build stratified homotopy types from. 

\subsection{The model category of diagrams}\label{subsec:mod_of_diag}

Recall that any category of functors into a model category admits a projective model structure. Note that throughout the paper, we consider the category of simplicial sets with the Kan-Quillen model structure.

\begin{theorem}[{\cite[Theorem 11.6.1]{hirschhornModel}}]
\label{theo:CMF_DiagP}
There exists a projective model structure on $\Diag_P=\Fun(R(P)^{\op},\sS)$, such that:
\begin{itemize}
    \item fibrations are maps $f\colon F\to G$, such that $f(\I)\colon F(\I)\to G(\I)$ is a Kan fibration for all $\I\in R(P)$,
    \item weak equivalences are maps $f\colon F\to G$, such that $f(\I)\colon F(\I)\to G(\I)$ is a weak equivalence of simplicial sets, for all $\I\in R(P)$.
\end{itemize}
\end{theorem}
This model structure is cofibrantly generated, and we need the following definition to make explicit the sets of generating (trivial) cofibrations.

\begin{definition}
Let $S$ be a simplicial set and $\I$ a regular flag. Let $S^{\I}\in \Diag_P$ be the diagram defined as follows:
\begin{equation*}
    S^{\I}(\I')=\left\{\begin{array}{cl}
         S& \text{ if $\I'\subset \I$,} \\
         \emptyset&  \text{ else.}
    \end{array}\right.
\end{equation*}
\end{definition}

\begin{proposition}
The set of generating cofibrations for the model structure of  \cref{theo:CMF_DiagP} is given by:
\begin{equation*}
    \{(\partial\Delta^n)^{\I}\to(\Delta^n)^{\I}\mid n\geq 0,\ \I\in R(P)\}.
\end{equation*}
The set of generating trivial cofibrations is given by:
\begin{equation*}
    \{(\Lambda^n_k)^{\I}\to(\Delta^n)^{\I}\mid n\geq 1, \ 0\leq k\leq n,\ \I\in R(P)\}.
\end{equation*}
Where $\Lambda^n_k$ is the $k$-th horn on $\Delta^n$, i.e. $\Lambda^n_k=\cup_{i\not=k}d_i(\Delta^n)\subset \Delta^n$.
\end{proposition}

It is difficult to give an explicit description of the cofibrant objects in a projective model structure in general, but in the case of $\Diag_P$ there exists a simple characterization. Since cofibrant diagrams are related to the vertical objects we will use in \cref{Section:Vertical} (see \cref{rem:Vertical_Are_Diagrams}), we give such a description here.

\begin{proposition}
\label{prop:Cofibrant_Diagrams}
Assume that $P$ is of finite length, and let $F\in \Diag_P$ be a diagram. Then $F$ is cofibrant if and only if the following two conditions are satisfied:
\begin{enumerate}
    \item \label{item:Cofibrant_Diagrams_Mono}  For all regular flags $\I\subset \I'$, the map $F(\I')\to F(\I)$ is a monomorphism,
    \item \label{item:Cofibrant_Diagrams_square} If $\I_1,\I_2$ are two regular flags such that $\I_0=\I_1\cap\I_2\not=\emptyset$, then either
    \begin{itemize}
        \item $\I_3=\I_1\cup\I_2$ is a flag, and $F(\I_1)\cap F(\I_2)=F(\I_3)\subset F(\I_0)$,
        \item or $\I_1\cup\I_2$ is not a flag, and $F(\I_1)\cap F(\I_2)=\emptyset$.
    \end{itemize}
\end{enumerate}
\end{proposition}

\begin{proof}
We give a brief sketch of the proof. For the direct implication, note that a projectively cofibrant diagram must satisfy Condition \eqref{item:Cofibrant_Diagrams_Mono} whenever the indexing category is a poset. For Condition \eqref{item:Cofibrant_Diagrams_square}, one easily checks that it is satisfied by the domains and codomains of generating cofibrations. Then, it is a matter of checking that this property is preserved by taking disjoint unions, pushouts and retracts.

For the reverse implication, consider the following characterization of projectively cofibrant diagrams of simplicial sets \cite[Corollary 9.4]{UniversalHomotopyTheories}. A diagram $F\colon R(P)^{\op}\to \sS$ is projectively cofibrant if and only if the following two conditions are satisfied:
\begin{enumerate}[label=(\alph*)]
    \item \label{item:Dugger_b} For all $n\geq 0$, the functor of $n$-simplicies $F_n\colon R(P)^{\op}\to\Set$ splits into two sub-functors $F_n=F_n^{\text{degen}}\coprod F_n^{\text{non-degen}}$ where the former reaches exactly the degenerate $n$-simplices of $F(\I)$, for all $\I$.
    \item \label{item:Dugger_a} For all $n\geq 0$, $F_n$ can be further decomposed into a coproduct of representables.
\end{enumerate}
We will first prove that $\eqref{item:Cofibrant_Diagrams_Mono}\Rightarrow\ref{item:Dugger_b}$. Let $\sigma\in F(\I)_n$ for some $\I\in R(P)$. Define the set $F_{\sigma}\subset\coprod_{\I} F(\I)_n$ as the smallest subset such that
\begin{itemize}
    \item $\sigma\in F_{\sigma}$,
    \item $\forall\ \I,\I'\in R(P)$, if $\tau\in F(\I)\cap F_{\sigma}$ and $\I'\subset \I$, then the image of $\tau$ in $F(\I')$ is in $F_{\sigma}$
    \item $\forall\ \I,\I'\in R(P)$, if $\tau\in F(\I)$ and $\I'\subset \I$, and the image of $\tau$ in $F(\I')$ is in $F_{\sigma}$, then $\tau$ is in $F_{\sigma}$.
\end{itemize}
First note that by construction, if $\sigma\in F(\I)$ and $\tau\in F(\I')$, then either $F_{\sigma}\cap F_{\tau}=\emptyset$ or $F_{\sigma}=F_{\tau}$. Next, by \eqref{item:Cofibrant_Diagrams_Mono}, if $\sigma$ is non-degenerate, then so are all the simplices in $F_{\sigma}$, since cofibrations send non-degenerate simplices to non-degenerate simplices, and similarly, if $\sigma$ is degenerate all simplices in $F_{\sigma}$ are degenerate. Now, note that $F_{\sigma}$ defines a subfunctor of $F_n$ by setting $F_\sigma(\I)=F_\sigma\cap F(\I)$.
To prove that $\eqref{item:Cofibrant_Diagrams_Mono} + \eqref{item:Cofibrant_Diagrams_square}\Rightarrow \ref{item:Dugger_a}$, it is enough to show that the $F_\sigma$ defined above are representable. Let $Q_{\sigma}=\{(\tau,\I)\mid \tau\in F_{\sigma}(\I)\}$, and equip it with the order relation $(\tau_1,\I_1)<(\tau_2,\I_2)$ if $\I_1\subset \I_2$ and the image of $\tau_2$ in $F(\I_1)$ is $\tau_1$. Then, by construction of $F_{\sigma}$, the poset $Q_{\sigma}$ is connected. 
In particular, if $(\tau,\I)$ and $(\tau',\I)$ are two elements of $Q_{\sigma}$, then they must be related by a zigzag. Note that by \eqref{item:Cofibrant_Diagrams_square}, any zigzag of the form $(\tau_1,\I_1)>(\tau_0,\I_0)<(\tau_2,\I_2)$ can be replaced by a zigzag $(\tau_1,\I_1)<(\tau_3,\I_3)>(\tau_2,\I_2)$. Then, by repeated application of \eqref{item:Cofibrant_Diagrams_square}, one can assume that there exists $(\tau'',\I')$, such that, $(\tau,\I)<(\tau'',\I')>(\tau',\I)$, which implies that $\tau=\tau'$. This has two consequences. First, $Q_{\sigma}$ is a subposet of $R(P)^{\op}$, and second, $F_\sigma(\I)$ is either empty or is a point. To show that $F_{\sigma}$ is reprensentable, it remains to be showed that $Q_{\sigma}$ admits a maximal element. Once again, by repeated application of \eqref{item:Cofibrant_Diagrams_square}, any two elements $(\tau_1,\I_1),(\tau_2,\I_2)$ must admit an upper bound $(\tau_3,\I_3)\geq(\tau_1,\I_1),(\tau_2,\I_2)$, but now, $\I_3$ is a regular flag of length greater or equal than the length of $\I_1$ and $\I_2$, with possible equality only if $\I_1=\I_2$. Since $P$ is assumed to be of finite length, the length of regular flags is bounded, and thus $Q_\sigma$ admits a maximal element, $\I_\sigma$. The functor $F_\sigma\colon R(P)^{\op}\to \Set$ is then represented by $\I_\sigma$.
\end{proof}

\subsection{Model categories for stratified spaces}\label{subsec:mod_of_strat_space}
We can now describe the model category of stratified spaces as introduced in \cite{douteauEnTop}.

\begin{theorem}[{\cite[Theorem 2.15]{douteauEnTop}}]
\label{theo:CMF_TopP}
There exists a model structure on $\Top_P$, where a map $f\colon X\to Y$ is
\begin{itemize}
    \item a fibration if the induced maps $\HolIP(X)\to\HolIP(Y)$ are fibrations for all $\I\in R(P)$,
    \item a weak equivalence if the induced maps $\HolIP(X)\to\HolIP(Y)$ are weak equivalences for all $\I\in R(P)$.
\end{itemize}
This model structure is cofibrantly generated and the set of generating cofibrations is given by
\begin{equation*}
    \{\RealP{\Delta^{\I}}\times S^{n-1}\hookrightarrow \RealP{\Delta^{\I}}\times B^{n}\mid n\geq 0,\ \I\in R(P)\},
\end{equation*}
while the set of generating trivial cofibrations is given by
    \begin{equation*}
    \{\RealP{\Delta^{\I}}\times B^{n}\times\{0\}\hookrightarrow \RealP{\Delta^{\I}}\times B^{n}\times [0,1]\mid n\geq 0,\ \I\in R(P)\},        
    \end{equation*}
where $B^n$ and $S^{n-1}$ are respectively the $n$-dimensional ball and the $(n-1)$-dimensional sphere bounding it.
\end{theorem}

\begin{remark}\label{rem:StratSpacesAreFibrant}
The model category $\Top_P$ has the notable property that all of its objects are fibrant, which is a property also enjoyed by the classical model structure on $\Top$. To see this, note that for all $\I\in R(P)$, $\HolIP(P)=*$, in particular, for all $X\in \Top_P$, the map $\HolIP(X)\to \HolIP(P)=*$ is a fibration in $\Top$, which implies that the map $X\to P$ is a fibration in $\Top_P$.
\end{remark}
\begin{example}\label{ex:counter_ex_forgetful}
For arbitrary stratified spaces in $\Top_P$ it is not the case that the forgetful functor $\TopP \to \Top$ preserves weak equivalences. Indeed, let $E$ be the Hawaiian Earring (see for ex. \cite{Hawaii}) and $f: E'=\bigvee_{n \in \mathbb N}S^1 \to E$ the bijection obtained by decomposing $H$ into its circles, and gluing these along the basepoint. $f$ is far from being a weak-equivalence of topological spaces. Indeed, \cite[Thm. 3.1]{Hawaii} shows that it does not even induce an isomorphism on singular homology in degree $1$. Nevertheless, if we consider $E$ and $E'$ as stratified over $P = \{0 <1 \}$, by sending the singular point to $0$, then $f$ induces a weak equivalence of stratified spaces. Indeed, it even induces a homeomorphism on strata and homotopy links. 
This should not be excessively surprising. The forgetful functor $\Top_P\to \Top$ is left Quillen, which means it is only expected to preserve weak equivalences between \textbf{cofibrant} objects. However, not every object in $\Top_P$, in particular not the Hawaiian Earring, is cofibrant.
\end{example}
\begin{remark}\label{rem:forget_preserves_we_triangulable}
Despite \cref{ex:counter_ex_forgetful}, it holds true that the forgetful functor $\Top_P \to \Top$ preserves weak equivalences for much larger class than that of cofibrant objects. In fact, it preserves weak equivalences between realizations of stratified simplicial sets. This is a consequence of \cref{theo:Equivalence_Simplicial_Homotopy_Category}, together with the fact that $\sS_P \to \sS$ preserves weak equivalences, by a direct application of Ken Browns Lemma (\cite[Cor. 7.7.2]{hirschhornModel}.
\end{remark}
\begin{remark}\label{rem:TopP_Transported}
The model structure defined above was originally by transporting the model structure from $\Diag_P$ along the adjunction $C_P\colon \Diag_P\leftrightarrow\Top_P\colon D_P$ (see \cref{rec:Diagram_Functors}). This uses the transport theorem \cite[Corollary 3.3.4]{NecessaryAndSufficientConditionInduced}, and this is were the necessity to restrict to $\Delta$-generated spaces arose, since the theorem requires $\Top_P$ to be locally presentable (see \cref{rem:Delta_generated}). Using the alternative transfer theorem \cite[A.1]{stephan2013equivariant} instead, one can prove the existence of an identically defined model structure on the category of $P$-stratified spaces defined using \define{all} topological spaces.
Denote the larger model category constructed in this fashion by $\mathcal T_P$. Then $\mathcal T_P$ has identically defined acyclic cofibrations and fibrations as $\Top_P$. In particular, the adjunction 
\begin{align*}
    \mathcal T_P \leftrightarrow \Top_P
\end{align*}
given by inclusion and $\Delta$-generation is a Quillen-adjunction. It follows immediately from the universal property of $\Delta$-generation and the definition of homotopy links that both unit and counit of this adjunction are weak equivalences. In particular, the two model categories are Quillen-equivalent. Hence, from the perspective of homotopy theory, the restriction to $\Delta$-generated spaces can be neglected. At the same time, $\TopP$ enjoys the additional benefit of being a combinatorial model category enriched over $\Top$, (and over $\sS$, through the functor $\Sing$) making it generally preferable to work with. Even more, one easily checks that with its simplicial structure described in \cref{rec:enriched_structures}, $\Top_P$ is a simplicial model category. 
\end{remark}
There is a similar model structure for $\TopNP$, also obtained by transporting along the adjunction with $\Diag_P$.
\begin{theorem}[{\cite[Theorem 2.8]{douteauEnTop}}]
\label{theo:CMF_TopNP}
There exists a model structure on $\TopNP$, where a map $f\colon X\to Y$ is
\begin{itemize}
    \item a fibration if the induced maps $\HolINP(X)\to\HolINP(Y)$ are fibrations for all $\I\in R(P)$,
    \item a weak equivalence if the induced maps $\HolINP(X)\to\HolINP(Y)$ are weak equivalences for all $\I\in R(P)$.
\end{itemize}
This model structure is cofibrantly generated and the set of generating cofibrations is given by
\begin{equation*}
    \{\RealNP{\Delta^{\I}}\times S^{n-1}\hookrightarrow \RealNP{\Delta^{\I}}\times B^{n}\mid n\geq 0,\ \I\in R(P)\},
\end{equation*}
while the set of generating trivial cofibrations is given by
    \begin{equation*}
    \{\RealNP{\Delta^{\I}}\times B^{n}\times\{0\}\hookrightarrow \RealNP{\Delta^{\I}}\times B^{n}\times [0,1]\mid n\geq 0,\ \I\in R(P)\},        
    \end{equation*}
where $B^n$ and $S^{n-1}$ are respectively the $n$-dimensional ball and the $(n-1)$-dimensional sphere bounding it.
\end{theorem}

The model categories from \cref{theo:CMF_TopP,theo:CMF_TopNP} are Quillen-equivalent, through the adjunction $\varphi_P\circ-\dashv -\times_P\Real{N(P)}$ described in \cref{rec:circ_varphiP}, see \cite[Theorem 2.15]{douteauEnTop} together with the correction in \cite[Theorem 3.15]{douteau2021stratified} for a proof.
\begin{theorem}
The adjunction $
    \varphi_P\circ-\colon \TopNP\leftrightarrow\Top_P\colon -\times_P\Real{N(P)} $
is a Quillen-equivalence.
\end{theorem}

The model categories $\TopNP$ and $\Top_P$ are also Quillen-equivalent to $\Diag_P$.

\begin{theorem}[{\cite[Theorem 2.12]{douteauEnTop}}]
The adjunctions $C_P\colon \Diag_P\leftrightarrow\TopNP\colon D_P$ and $C_P\colon \Diag_P\leftrightarrow\Top_P\colon D_P$ are Quillen-equivalences.
\end{theorem}

Leveraging the structure described in \cref{rem:Strat_Fibers_Over_Poset}, one can construct a model structure on $\Strat$ through the theory of Quillen-bifibrations, see \cite{CagneMellies}. Note that any isomorphism of posets $\alpha\colon P\to Q$, induces an equivalence of categories $\Diag_Q\to\Diag_P$. In particular, given a map in $\Strat$ 
\begin{equation}
    \begin{tikzcd}
             X
             \arrow{r}{f}
             \arrow[swap]{d}{\varphi_X}
             &Y
             \arrow{d}{\varphi_Y}
             \\
             P
             \arrow{r}{\bar{f}}
             &Q\spacecomma
    \end{tikzcd}
\end{equation}
where $\bar{f}$ is an isomorphism, one can consider the diagram associated to $Y$ as a diagram in $\Diag_P$, and consider $f$ as inducing a map in $\Diag_P$. From this, it follows that a weak equivalence induces a map of homotopy links.
\begin{theorem}[{\cite[Theorem 3.6]{douteauEnTop}}]
There exist a model structure on $\Strat$ where a map $f\colon (X\to P)\to (Y\to Q)$ is a weak equivalence if and only if $\bar{f}\colon P\to Q$ is an isomorphism, and the maps induced by $f$
\begin{equation*}
    \HolIP(X)\to\mathcal{H}\!\textit{o}\mathrm{Link}_{\bar{f}(\I)}(Y),
\end{equation*}
are weak equivalences for all $\I\in R(P)$.\\
This model category is cofibrantly generated with set of generating cofibrations is given by
\begin{equation*}
    \{\Real{\Delta^{\I}}_{\N}\times S^{n-1}\hookrightarrow \Real{\Delta^{\I}}_{\N}\times B^{n}\mid n\geq 0,\ \I\in R(\N)\},
\end{equation*}
and the set of generating trivial cofibrations is given by
    \begin{equation*}
    \{\Real{\Delta^{\I}}_{\N}\times B^{n}\times\{0\}\hookrightarrow \Real{\Delta^{\I}}_{\N}\times B^{n}\times [0,1]\mid n\geq 0,\ \I\in R(\N)\}.       
    \end{equation*}

\end{theorem}

\subsection{The model category of stratified simplicial sets}\label{subsec:mod_of_strat_ss}

Before defining a model structure for stratified simplicial sets, we need to define a particular set of stratified horn inclusions. 

\begin{definition}
A  \define{stratified horn inclusion} is an inclusion $\Lambda^{\J}_k\to \Delta^{\J}\in \sS_P$, where $\J$ is a flag and $\Lambda^{\J}_k=\cup_{i\not= k}d_i(\Delta^{\J})$, i.e. it is the union of all proper faces of $\Delta^{\J}$ other than its $k$-th face.
An \define{admissible horn inclusion} is a stratified horn inclusion $\Lambda^{\J}_k\to \Delta^{\J}$ such that any of the following equivalent condition is satisfied:
\begin{itemize}
    \item The map $\Lambda^{\J}_k\to \Delta^{\J}$ is a stratified homotopy equivalence (see \cref{def:Strat_Homotopies_simplicial}).
    \item There exists a flag $\J'$ such that $\Delta^{\J}$ is either the $k$-th or $(k+1)$-th degeneracy of $\Delta^{\J'}$.
    \item $\J=[p_0\leq\dots\leq p_n]$ and $p_k=p_{k+1}$ or $p_k=p_{k-1}$.
\end{itemize}
\end{definition}
For a proof that the above conditions are equivalent, see \cite[Proposition 1.13]{douSimp}.

\begin{theorem}[{\cite[Theorem 2.14]{douSimp}}]
\label{theo:CMF_sSetP}
There exists a cofibrantly generated model structure on $\sS_P$ where the generating set of cofibrations is given by
\begin{equation*}
    \{\partial\Delta^{\J}\to\Delta^{\J}\mid \Delta^{\J}\in \Delta(P)\},
\end{equation*}
and the generating set of trivial cofibrations is given by
\begin{equation*}
    \{\iota^{\J}_k\colon\Lambda^{\J}_k\to\Delta^{\J}\mid \text{$\iota^{\J}_k$ is an admissible horn inclusion} \}.
\end{equation*}
weak equivalences in this category are the maps $f\colon K\to L$ such that for all $\I\in R(P)$, the induced maps
\begin{equation}
    \Hol_{\I}(K^{\fib})\to\Hol_{\I}(L^{\fib})
\end{equation}
are weak equivalences, where $(-)^{\fib}$ means passing to a fibrant replacement.
\end{theorem}

\begin{remark}
As is usual for model category of presheaves, all objects of $\sS_P$ are cofibrant. In fact, the class of cofibrations is exactly the class of monomorphisms in $\sS_P$. Indeed, the set of generating cofibrations is the set of boundary inclusions of (stratified) simplices, which can easily be seen to generate all monomorphisms.
\end{remark}

\begin{remark}\label{rem:Haine_Structure_is_localization}
Note that stratified homotopy equivalences in the sense of \cref{def:Strat_Homotopies_simplicial} are weak equivalences in the model category described above. Furthermore, this model structure is minimal among model structures satisfying this property and having monomorphisms as cofibrations, i.e. it has the smallest possible class of weak equivalences. This comes from the fact that the model structure of \cref{theo:CMF_sSetP} is built using Cisinski's theory \cite{CisinskiPrefaisceaux}. By \cite[Théorème 1.3.22]{CisinskiPrefaisceaux}, such a model category is specified by the data of a cylinder (\cref{def:Strat_Homotopies_simplicial}) and a class of anodyne extension. In the case of $\sS_P$, since admissible horn inclusions are stratified homotopy equivalences, the class of anodyne extensions is the one generated by the empty set, see \cite[1.3.12, Proposition 1.3.13]{CisinskiPrefaisceaux}. This implies in particular that if a model structure on $\sS_P$ is such that
\begin{itemize}
    \item the cofibrations are monomorphisms,
    \item stratified homotopy equivalences are weak equivalences,
\end{itemize}
then it is a right Bousfield localization of the structure of \cref{theo:CMF_sSetP}. One particular example of interest is the model structure $\sS_P^{\text{Joyal-Kan}}$ defined by Haine in \cite{haine2018homotopy}. Haine starts from the Joyal model structure, and then localizes the structure along the inclusions of stratified simplicial sets into their cylinders. In particular, $\sS_P^{\text{Joyal-Kan}}$ is a right Bousfield localization of $\sS_P$.
\end{remark}

\begin{remark}
Note that, contrary to the case of stratified spaces (\cref{theo:CMF_TopP}), the weak equivalences for stratified simplicial sets are not \define{defined} as the maps inducing weak equivalences between homotopy links, but only as those maps satisfying this property \define{after a suitable fibrant replacement}. In particular, the model structure on $\sS_P$ is not transported from the category of diagrams $\Diag_P$. On the other hand, consider the model structure described by Henriques in \cite{Henriques}, in which weak equivalences are defined as those maps inducing weak equivalences between all simplicial homotopy links. The two model structures are known to be the same for very general reasons (for example, consider the previous remark), but it is not immediately clear why the classes of weak equivalences coincide. We investigate this in \cref{Section:LastHolink}, see also \cref{rem:Henriques}.
\end{remark}
Similarly to the case of stratified spaces, there is a model structure on all stratified simplicial sets.
\begin{theorem}
There exists a model structure on $s\Strat$ where a map $f\colon (K\to N(P))\to (L\to N(Q))$ is a weak equivalence if and only if $\bar{f}\colon P\to Q$ is an isomrophism and the maps induced by $f$
\begin{equation*}
    \Hol_{\I}(K^{\fib})\to\Hol_{\bar{f}(\I)}(L^{\fib})
\end{equation*}
are weak equivalences for all $\I\in R(P)$, where $(-)^{\fib}$ is a fibrant replacement in $\sS_P$.

The model category is cofibrantly generated, and the set of generating cofibrations is given by
\begin{equation*}
    \{\partial\Delta^{\I}\hookrightarrow \Delta^{\I}\mid \I\in R(\N)\},
\end{equation*}
and the set of generating trivial cofibrations is given by
\begin{equation*}
    \{\iota_k^{\I}\colon \Lambda_k^{\I}\hookrightarrow \Delta^{\I}\mid \text{$\iota^{\I}_k$ is admissible in $\sS_{\N}$}\}.
\end{equation*}
\end{theorem}

\subsection{Failure to be a Quillen-equivalence}\label{subsec:failure_quillen}

Now that we have two homotopy theories, one for stratified simplicial sets, and one for stratified spaces, we want to show that they coincide through the adjunction $\RealStrat{-}\colon s\Strat\leftrightarrow\Strat\colon \SingStrat$. In the language of model categories, this would mean showing that the adjunction is a Quillen-equivalence. However, this adjunction already fails to be a Quillen-adjunction, since, the right adjoint $\SingStrat$ does not preserve fibrancy, as we show in the following recollection.
\begin{recollection}
\label{rec:SingP_Not_Quillen}
Let $X\to P$ be a stratified space. For $X$ to be fibrant, one has to check that for all regular flags $\I\in R(P)$, $D_P(X)(\I)$ is a Kan-complex. But note that by definition, those are equal to $\Sing(\HolIP(X))$, and so, are Kan-complexes, which means that any stratified space is fibrant as already observed in \cref{rem:StratSpacesAreFibrant}. On the other hand, for $\Sing_P(X)$ to be fibrant means that in any diagram of the form
\begin{equation*}
\begin{tikzcd}
    \Lambda^{\J}_k
    \arrow{r}
    \arrow{d}
    &\Sing_P(X)
    \arrow{d}
    \\
    \Delta^{\J}
    \arrow{r}
    \arrow[dashed]{ur}
    &N(P)\spacecomma
\end{tikzcd}
\end{equation*}
where $\Lambda^{\J}_k$ is admissible, there exists a lift. Using the adjunction $\RealP{-}\dashv\Sing_P$, and the fact that $\Sing_P(P)=N(P)$, this is equivalent to asking for a lift in diagrams of the form
\begin{equation*}
\begin{tikzcd}
    \RealP{\Lambda^{\J}_k}
    \arrow{r}
    \arrow{d}
    &X
    \arrow{d}
    \\
    \RealP{\Delta^{\J}}
    \arrow{r}
    \arrow[dashed]{ur}
    &P\spaceperiod
\end{tikzcd}
\end{equation*}
Now consider the poset $P=\{0<1\}$, and the stratified space $X=\RealP{\Lambda^{\J}_1}$ with $\J=[0\leq 0\leq 1]$. Then, $\Lambda^{\J}_1$ is admissible, and so for $\Sing_P(X)$ to be fibrant, there must be a lift in the following diagram
\begin{equation*}
\begin{tikzcd}
    \RealP{\Lambda^{\J}_1}
    \arrow{r}{\Id}
    \arrow{d}
    &\RealP{\Lambda^{\J}_1}
    \arrow{d}
    \\
    \RealP{\Delta^{\J}}
    \arrow{r}
    \arrow[dashed]{ur}
    &P\spacecomma
\end{tikzcd}
\end{equation*}
but there exist no stratified section $\RealP{\Delta^{\J}}\to\RealP{\Lambda^{\J}_1}$, as illustrated in \cref{Fig:Non_Fibrancy}, which means that $\Sing_P(\RealP{\Lambda^{\J}_1})$ is not fibrant. In particular, the functor $\Sing_P$ does not preserve fibrations.
\end{recollection}

\begin{figure}[h]
\begin{tikzpicture}
\draw[blue, thick,opacity=0.3](0,2)--(2,2);
\draw[red, thick](0,0)--(0,2.028);

\draw[->] (2,1)--(3.6,1);
\filldraw[blue,blue,opacity=0.3](4,2)--(6,2)--(4,0)--(4,2);
\draw[blue,thick](4,0)--(6,2);
\draw[red,thick](4,0)--(4,2.01);

\draw[->,gray, dashed] (4.9,1.1)--(4.1,1.9);
\draw[->,gray, dashed] (4.4,0.6)--(4.1,0.9);
\draw[->,gray, dashed] (5.4,1.6)--(5.1,1.9);
\end{tikzpicture}
\caption{The inclusion $\RealP{\Lambda^{\J}_1}\hookrightarrow\RealP{\Delta^{\J}}$, with $\J=[0\leq 0\leq 1]$, admits no stratified section, since any section would send some part ot the interior of $\Real{\Delta^{2}}$ to the singular stratum. }
\label{Fig:Non_Fibrancy}
\end{figure}
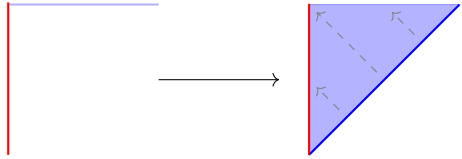

On the other hand, stratified spaces that are in one of the three classes mentioned in \cref{Section:Classical_Links}, pseudo-manifolds, conically stratified spaces and homotopically stratified sets, have better behaved $\Sing_P$.

\begin{theorem}[{\cite[Theorem A.6.4]{HigherAlgebra}\cite[Proposition 8.1.2.6]{nand2019simplicial}}]
\label{theo:Conical_Fibrant}
Let $X\to P$ be a stratified space. If it is either a conically stratified space or if $P$ is finite and $X$ is a homotopically stratified set, then the simplicial set underlying $\Sing_P(X)$ is a quasi-category.
\end{theorem}

In those cases, this implies that $\Sing_P(X)$ is a fibrant object in $\sS_P$, see \cite[Proposition 4.12]{douSimp}. Note that while \cref{rec:SingP_Not_Quillen} implies that the adjunction $\RealP{-}\dashv\Sing_P$ is not a Quillen-adjunction, we can still use it to deduce results about objects for which $\Sing_P$ \define{is} fibrant, such as those mentionned above. In fact it is possible to recover an independent proof of Miller's Theorem  (\cref{theo:Miller_Theorem}), this way (see \cite[Section 5]{douSimp}).
Relatedly, \cref{theo:intro_Con_embeds_HoStrat} (\cref{Cor:Conic_embed}) states that those objects do behave like fibrant objects in the homotopy category $\Ho\Strat$.

\subsection{The stratified subdivision and its adjoint}\label{subsec:strat_sd_ex}

As we have seen in the previous section, the model structure on $\Top_P$ and $\sS_P$ cannot be compared directly through the adjunction $\RealP{-}\dashv\Sing_P$. To achieve a comparison, one needs to work with a suitably defined stratified subdivision. Note that this stratified subdivision was already used to characterize the model structure of \cref{theo:CMF_sSetP} in \cite{douSimp}, and that, through its adjoint, it allows for the definition of a fibrant replacement functor in $\sS_P$, see \cref{Section:Ex_P_FSAE,cor:Exi_P_fibrant_replacement}.
\begin{remark}
In this subsection and through the remainder of this paper, we will often abuse notation by identifying an $n$-dimensional stratified simplex $\Delta^{\J}$ with its underlying simplex $\Delta^{n}$. This allows for notation such as $\Delta^{\J}_0$ to refer to the set of vertices of $\Delta^n$.
\end{remark}
\begin{recollection}
Denote by $\sd \colon \sS \to \sS$ the classical barycentric subdivision \cite[Section 2]{KanSubdivision}.
Given a stratified simplex $\Delta^{\J}$, $\J= [p_0 \leq \dots \leq p_n]$ its stratified subdivision is the subcomplex $\sd_P(\Delta^{\J}) \subset \sd(\Delta^{n}) \times \Delta^{\J}$, given by such simplices $[( \sigma_0, q_0), \dots ,( \sigma_k, q_k)] $
fulfilling $\{q_0,\dots, q_k\} \subset \{ p_i \mid i \in \sigma_0 \}$. It is stratified via projection to the second component. This construction left Kan extends to a functor
  \[ \sd_P\colon \sS_P \to \sS_P .\]
  For a stratified simplicial set $K\in\sS_P$, we can also describe the stratified simplicial set $\sd_P(K)$ more explicitely. Note that we have an inclusion $\sd_P(N(P))\subset\sd(N(P))\times N(P)$, composing it with the projection on the first factor gives a map $\sd_P(N(P))\to \sd(N(P))$. The stratified simplicial set $\sd_P(K)$ can then be seen as the following pullback:
  \begin{equation*}
      \begin{tikzcd}
               \sd_P(K)
               \arrow{r}
               \arrow{d}
               &\sd(K)
               \arrow{d}{\sd(\varphi_K)}
               \\
               \sd_P(N(P))
               \arrow{r}
               &\sd(N(P))\spacecomma
      \end{tikzcd}
  \end{equation*}
 and the stratification is given by the composition $\sd_P(K)\to\sd_P(N(P))\to N(P)$.

The functor $\sd_P$ naturally comes with a natural transformation,
$\lv_P\colon \sd_P \to 1_{\sS_P}$, the \define{stratified last vertex map}.
We define it on stratified simplices by giving its value on vertices: 
\begin{align*}
  \lv_P\colon\sd_P(\Delta^{\J})_0 &\to \Delta^{\J}_0 \\
    ([q_0 \leq \dots \leq q_k], q) &\mapsto \max \{ i \mid q_i = q \}.
\end{align*}
Then, since the above is a map between simplicial complexes, its value on vertices uniquely defines a simplicial map $\lv_P\colon \sd_P(\Delta^{\J})\to\Delta^{\J}$.

The functor $\sd_P$ admits a right adjoint, defined as 
    \begin{align*}
    \Ex_P: \sS_P &\to \sS_P \\
    K &\mapsto \sS_P(\sd_P(\Delta^{-}), K)
    \end{align*}
    where again, by \cref{rem:sSetP_Presheaf_category}, we consider $\sS_P$ as a category of presheaves over $\Delta(P)$. The adjoints of the maps $\lv_P\colon \sd_P(K)\to K$ give natural maps $ \iota_K\colon K \hookrightarrow \Ex_P(K)$, which assemble into a natural transformation $\iota \colon \Id_{\sS_P}\to\Ex_P$. 
    Finally, we denote by $\sd^n_P$, $\Ex^n_P$, $\lv^n_P, \iota^n$ the obvious functors obtained by iteration and by $\Ex^{\infty}_P$ the colimit of the diagram \[1_{\sS_P} \hookrightarrow \Ex_P \hookrightarrow Ex^2_P \hookrightarrow \dots .\]
    Given $K \in \sS_P$ the stratified simplicial set $\Ex^{\infty}_P(K)$ is fibrant \cite[Lemma 2.10]{douSimp}. It is the content of \cref{cor:Exi_P_fibrant_replacement} that $X \hookrightarrow \Ex^{\infty}_P(X)$ in fact defines a fibrant replacement of $K$.
\end{recollection}

\begin{figure}[h]
\begin{tikzpicture}
\draw[black](-1,-1) -- (1,-1);
\draw[black](-1,-1) -- (0,1);
\draw[black](0,1) -- (1,-1);
\filldraw[blue](0,1) circle (2pt);
\filldraw[red] (1,-1) circle (2pt);
\filldraw[red] (-1,-1) circle (2pt);

\draw[black,cm ={1,0,0,1,(4cm,0cm)}](-1,-1) -- (1,-1);
\draw[black,cm ={1,0,0,1,(4cm,0cm)}](-1,-1) -- (0,1);
\draw[black,cm ={1,0,0,1,(4cm,0cm)}](0,1) -- (1,-1);
\draw[black,cm ={1,0,0,1,(4cm,0cm)}](-0.5,0) -- (0,-0.5)--(0.5,0);
\draw[black,cm ={1,0,0,1,(4cm,0cm)}](0,-1)--(0,-0.5)--(0,1);
\draw[black,cm ={1,0,0,1,(4cm,0cm)}](1,-1)--(0,-0.5)--(-1,-1);
\filldraw[blue,cm ={1,0,0,1,(4cm,0cm)}](0,1) circle (2pt);
\filldraw[red,cm ={1,0,0,1,(4cm,0cm)}] (1,-1) circle (2pt);
\filldraw[red,cm ={1,0,0,1,(4cm,0cm)}] (-1,-1) circle (2pt);
\filldraw[blue,cm ={1,0,0,1,(4cm,0cm)}] (-0.5,0) circle (2pt);
\filldraw[blue,cm ={1,0,0,1,(4cm,0cm)}] (0.5,0) circle (2pt);
\filldraw[blue,cm ={1,0,0,1,(4cm,0cm)}] (0,-0.5) circle (2pt);
\filldraw[red,cm ={1,0,0,1,(4cm,0cm)}] (0,-1) circle (2pt);

\draw[black,cm ={1,0,0,1,(8cm,0cm)}](-1,-1) -- (1,-1);
\draw[black,cm ={1,0,0,1,(8cm,0cm)}](-1,-1) -- (0,1);
\draw[black,cm ={1,0,0,1,(8cm,0cm)}](0,1) -- (1,-1);
\draw[black,cm ={1,0,0,1,(8cm,0cm)}](0,-1)--(0,-0.5)--(0,0.5)--(0,1);
\draw[black,cm ={1,0,0,1,(8cm,0cm)}](0.3,0.4)--(0,0.3)--(-0.3,0.4);
\draw[black,cm ={1,0,0,1,(8cm,0cm)}](0.7,-0.4)--(0,-0.3)--(-0.7,-0.4);
\draw[black,cm ={1,0,0,1,(8cm,0cm)}](-1,-1)--(0,-0.3)--(1,-1);
\draw[black,cm ={1,0,0,1,(8cm,0cm)}](0.7,-0.4)--(0,0.3)--(-0.7,-0.4);
\filldraw[blue,cm ={1,0,0,1,(8cm,0cm)}](0,1) circle (2pt);
\filldraw[red,cm ={1,0,0,1,(8cm,0cm)}] (1,-1) circle (2pt);
\filldraw[red,cm ={1,0,0,1,(8cm,0cm)}] (-1,-1) circle (2pt);
\filldraw[red,cm ={1,0,0,1,(8cm,0cm)}] (0,-1) circle (2pt);
\filldraw[red,cm ={1,0,0,1,(8cm,0cm)}](0,-0.3) circle (2pt);
\filldraw[red,cm ={1,0,0,1,(8cm,0cm)}](0.7,-0.4) circle (2pt);
\filldraw[red,cm ={1,0,0,1,(8cm,0cm)}](-0.7,-0.4) circle (2pt);
\filldraw[blue,cm ={1,0,0,1,(8cm,0cm)}](0.3,0.4) circle (2pt);
\filldraw[blue,cm ={1,0,0,1,(8cm,0cm)}](-0.3,0.4) circle (2pt);
\filldraw[blue,cm ={1,0,0,1,(8cm,0cm)}](0,0.3) circle (2pt);

\draw[black,cm ={1,0,0,1,(12cm,0cm)}](-1,-1) -- (1,-1);
\draw[black,cm ={1,0,0,1,(12cm,0cm)}](-1,-1) -- (0,1);
\draw[black,cm ={1,0,0,1,(12cm,0cm)}](0,1) -- (1,-1);
\draw[black,cm ={1,0,0,1,(12cm,0cm)}](-0.66,-0.33)--(-0.33,-0.5) -- (0.33,-0.5)--(0.66,-0.33);
\draw[black,cm ={1,0,0,1,(12cm,0cm)}](-0.33,0.33)--(0,0.17) -- (0.33,0.33);
\draw[black,cm ={1,0,0,1,(12cm,0cm)}](0.33,-1)--(0.33,-0.5) -- (0,0.17);
\draw[black,cm ={1,0,0,1,(12cm,0cm)}](-0.33,-1)--(-0.33,-0.5) -- (0,0.17);
\draw[black,cm ={1,0,0,1,(12cm,0cm)}](-1,-1)--(-0.33,-0.5);
\draw[black,cm ={1,0,0,1,(12cm,0cm)}](1,-1)--(0.33,-0.5);
\draw[black,cm ={1,0,0,1,(12cm,0cm)}](0.66,-0.33)--(0,0.17);
\draw[black,cm ={1,0,0,1,(12cm,0cm)}](-0.66,-0.33)--(0,0.17);
\draw[black,cm ={1,0,0,1,(12cm,0cm)}](-0.33,-1)--(0.33,-0.5);
\draw[black,cm ={1,0,0,1,(12cm,0cm)}](0,1)--(0,0.17);
\filldraw[red,cm ={1,0,0,1,(12cm,0cm)}] (-1,-1) circle (2pt);
\filldraw[red,cm ={1,0,0,1,(12cm,0cm)}] (-0.33,-1) circle (2pt);
\filldraw[red,cm ={1,0,0,1,(12cm,0cm)}] (-0.33,-0.5) circle (2pt);
\filldraw[red,cm ={1,0,0,1,(12cm,0cm)}] (-0.66,-0.33) circle (2pt);
\filldraw[blue,cm ={1,0,0,1,(12cm,0cm)}](1,-1) circle (2pt);
\filldraw[blue,cm ={1,0,0,1,(12cm,0cm)}](0.33,-1) circle (2pt);
\filldraw[blue,cm ={1,0,0,1,(12cm,0cm)}](0.66,-0.33) circle (2pt);
\filldraw[blue,cm ={1,0,0,1,(12cm,0cm)}](0.33,-0.5) circle (2pt);
\filldraw[green,cm ={1,0,0,1,(12cm,0cm)}](0,1) circle (2pt);
\filldraw[green,cm ={1,0,0,1,(12cm,0cm)}](-0.33,0.33) circle (2pt);
\filldraw[green,cm ={1,0,0,1,(12cm,0cm)}](0.33,0.33) circle (2pt);
\filldraw[green,cm ={1,0,0,1,(12cm,0cm)}](0,0.17) circle (2pt);
\end{tikzpicture}
\caption{From left to right, the simplex $\Delta^{[0\leq 0\leq 1]}$, its "naive" subdivision, its stratified subdivision $\sd_P(\Delta^{[0\leq 0\leq 1]})$, and the stratified subdivision of $\Delta^{[0\leq 1\leq 2]}$}
\label{FigureSubdivisions}
\end{figure}
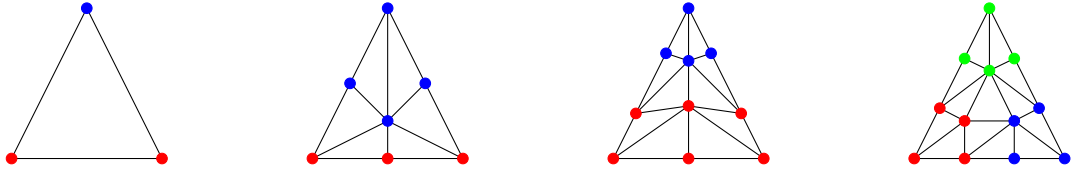

We will also use some properties of the subdivision functor.

\begin{proposition}
\label{prop:LvP_WeakEquivalence}
For all stratified simplicial sets $K\in \sS_P$, the last vertex map $\lv_P\colon \sd_P(K)\to K$ is a weak equivalence. In particular,
if $f\colon K\to L$ is a map in $\sS_P$, then $f$ is a weak equivalence if and only if $\sd_P(f)$ is a weak equivalence.
\end{proposition}

\begin{proof}
Consider the following commutative diagram
\begin{equation*}
    \begin{tikzcd}
             \sd_P(K)
             \arrow{r}{\sd_P(f)}
             \arrow[swap]{d}{\lv_P}
             &\sd_P(L)
             \arrow{d}{\lv_P}
             \\
             K
             \arrow{r}{f}
             &L\spaceperiod
    \end{tikzcd}
\end{equation*}
The first part of the proposition asserts that the vertical morphisms are weak equivalences. Thus, by the two out of three law, the first assertion of the proposition implies the second assertion. For the first assertion, consider \cite[Lemma A.3]{douSimp}, which applies to $\sd_P$ and $\lv_P$ as shown in the proof of \cite[Theorem 2.14]{douSimp}. One can also prove this directly by checking that it holds on representables, then using the cube lemma in an inductive argument.
\end{proof}

\subsection{Recovering a Quillen-equivalence}\label{subsec:recover_quil}
Composing the adjoint pairs $\sd_P\dashv\Ex_P$ and $\RealP{-}\dashv\Sing_P$, one gets a Quillen-equivalence.

\begin{theorem}[{\cite[Theorem 1.15]{douteau2021stratified}}]
\label{theo:QE_Through_sdP}
The adjoint pair
\begin{equation}\label{QuillenEquivalenceThroughSdP}
    \RealP{\sd_P(-)}\colon \sS_P\leftrightarrow\Top_P\colon\Ex_P\Sing_P
\end{equation}
is a Quillen-equivalence.
\end{theorem}

In particular, this means that the homotopy theories of spaces and simplicial sets stratified over $P$ coincide. 

\begin{remark}
The Quillen-equivalence obtained above is only partially satisfactory for several reasons.
\begin{itemize}
    \item First, the adjunction \eqref{QuillenEquivalenceThroughSdP} is not very well suited for the study of conically stratified objects, since even reasonable PL-pseudomanifolds are not cofibrant objects of $\Top_P$. In particular, they are not in the image of $\RealP{\sd_P(-)}$. Thus, it makes it difficult to relate the homotopy theory of those classical objects with that of stratified spaces. On the other hand, by working with the $\RealP{-}\dashv\Sing_P$ adjunction, we are able to show that the classical homotopy theory of conically stratified spaces embeds fully faithfully in that of stratified spaces (\cref{Cor:Conic_embed}). The proof relies on the fact that triangulated conically stratified spaces are in the image of $\RealP{-}$ and have fibrant $\Sing_P$. This illustrates the usefulness of working with the more natural adjunction $\RealP{-}\dashv\Sing_P$.
    \item The functors $\sd_P$ and $\Ex_P$ are not at all compatible with the Quillen-adjunctions  $\sS_P\leftrightarrow\sS_Q$, induced by maps of posets $\alpha\colon P\to Q$. In particular, one cannot recover a global adjunction $s\Strat\leftrightarrow\Strat$ by gluing the adjunctions \eqref{QuillenEquivalenceThroughSdP} for all $P$. This means that in order to compare the homotopy theory of all stratified spaces and stratified simplicial sets, then one needs to work directly with the adjunction $\RealStrat{-}\dashv\SingStrat$.
    \item Finally, the stratified setting is very similar to the classical setting. In the latter, the $\Real{-}\dashv\Sing$ adjunction is already a Quillen-equivalence, with no need for subdivision. Given the fact that the homotopy theory associated to $s\Strat$ and $\Strat$ can actually be related in a meaningful way (\cref{theo:Equivalence_Simplicial_Homotopy_Category}), through the $\RealStrat{-}\dashv\SingStrat$ adjunction, the appearance of $\sd_P$ and $\Ex_P$ in \eqref{QuillenEquivalenceThroughSdP} can appear as artificial.
\end{itemize}
\end{remark}

With that said, \cref{theo:QE_Through_sdP} can be seen as the first step to obtain a comparison of the topological and simplicial setting through the $\RealP{-}\dashv\Sing_P$ adjunction. Since the latter adjunction is not a Quillen-adjunction, we will need to work with categories with weak equivalences instead of the full model structures, see \cref{Section:Hammock}. In particular, we will need to prove that the functors $\RealP{-}$ and $\Sing_P$ characterize all weak equivalences (\cref{Cor:Realization_Preserve_Weak_Equivalences,theo:SingP_Characterize_WeakEquivalences}). By that, we mean that those functor preserve and reflect weak equivalences. By \cref{theo:QE_Through_sdP}, this is already known for the functors $\RealP{\sd_P(-)}$ and $\Ex_P\Sing_P$, since they are part of a Quillen-equivalence and all objects of $\sS_P$ are cofibrant and all objects of $\Top_P$ are fibrant. Thus, it suffices to show that for any $K\in \sS_P$, $\RealP{\sd_P(K)}\to\RealP{K}$ is a weak equivalence, (see \cref{Section:Real_Char_WE}, and the proof of \cref{Cor:Realization_Preserve_Weak_Equivalences}), and that for any $X\in \Top_P$, the map $\Sing_P(X)\to \Ex_P(\Sing_P(X))$ is a weak equivalence (\cref{PropositionExPFibrantReplacement}).

\section{$\Ex_P$ and $\Sing_P$ characterize weak equivalences}	
In the category of simplicial sets, an explicit and particularly convenient fibrant replacement functor is given by the Kan fibrant replacement $S \hookrightarrow \Exi(S)$ (see \cite{KanSubdivision,goerss2012simplicial}). The inclusion map is a trivial cofibration $S \hookrightarrow \Exi(S)$ which can be obtained (through transfinite composition) by glueing on simplices along horns (see \cite{MossSae}). In this section, we show that the same can be said for $\Exi_P$, in particular that it gives a fibrant replacement functor in the category $\sS_P$. From this, we then also obtain that the functor $\Sing_P: \Top_P \to \sS_P$ characterizes weak equivalences.
\label{Section:FSAE}
    \subsection{Fibrant replacement a la Kan, in $\sS_P$}
    \label{Section:Ex_P_FSAE}
    The main content of this subsection is proving the following result.
	\begin{proposition}\label{PropositionExPFibrantReplacement}
		Let ${K}\in \sS_P$ be a stratified simplicial set, then the map ${K}\to \Ex_P(K)$ is a trivial cofibration.
	\end{proposition}
	Together with the results from \cite{douSimp}, we will then obtain as an immediate consequence:
	\begin{corollary}\label{cor:Exi_P_fibrant_replacement}
	The functor $\Ex_P^{\infty}\colon \sS_P\to\sS_P$ is a fibrant replacement functor.
	\end{corollary}
	
	We prove \cref{PropositionExPFibrantReplacement} by adapting a proof of the equivalent statement in the non-stratified case from \cite{MossSae}. 
	This relies on the notion of a stratified strong anodyne extension, already used in \cite{douSimp}.
	Some intermediary results of \cite{MossSae} immediately extend to the stratified case and their proofs will be omitted. We refer the reader to the proof of \cite[Theorem 22]{MossSae} for more details.
	Nevertheless, a significant amount of technical preparations is required to replicate the necessary arguments in the stratified framework. 
	\begin{definition}
		A \define{(stratified) strong anodyne extension, (S)SAE} for short, is a morphism ${A} \hookrightarrow {B}$ in $\sS_P$, that is given by a transfinite composition of cobase changes of admissible horn inclusions.
	\end{definition}
	For the sake of brevity, we are just going to omit the \textit{stratified} and just refer to SAEs from here on out.
	\begin{remark}
		By definition, all SAEs are trivial cofibrations. In particular, an SAE ${A} \to {B}$ is a monomorphism, and can be safely identified with the inclusion of a (stratified) sub-simplicial set.
	\end{remark}
	Strong anodyne extensions enjoy an entirely combinatorial characterization, given as follows. Given a stratified simplicial set ${K}$, let $K_{\nd}$ be its set of non-degenerate simplices.
	\begin{definition}
		Let $m\colon {A} \hookrightarrow {B}$ be a cofibration in $\sS_P$. 
		\begin{enumerate}
			\item A \define{pairing} on $m$ is a partition of $B_{n.d.}\setminus A_{n.d.}$ into two sets $B_I$ and $B_{II}$ together with a bijection $T:B_{II} \to B_{I}$. The elements of $B_I$ and $B_{II}$ are referred to as \define{type I} and \define{type II} simplices respectively.
		\end{enumerate}
		Now let $T\colon B_{II} \to B_{I}$ be a pairing on $m$.
		\begin{enumerate}
			\setcounter{enumi}{1}
			\item $T$ is called \define{proper} if for each $\sigma  \in B_{II}$, $\sigma$ is a codimension one face of $T(\sigma) $ in a unique way.
			\item $T$ is called \define{admissible}, if in addition, for any type II simplex, $\sigma$, such that $T(\sigma)\colon \Delta^{\J}\to B$, and $\sigma=d_k(T(\sigma))$,  $\Lambda^{\J}_k\to\Delta^{\J}$ is an admissible horn inclusion.
			\item Given a pairing $T$ on $m$, the \define{ancestral relation} is the transitive binary relation on $B_{II}$ generated by $\sigma\prec_T \tau$ if $\sigma\not=\tau$ and $\sigma$ is a face of $T(\tau)$
			\item $T$ is called \define{regular} if the ancestral relation $\prec_T$ is well founded. 
		\end{enumerate}
	\end{definition}
	\begin{proposition}\label{PropFSAEChar}
		A cofibration ${A} \hookrightarrow {B}$ in $\sS_P$ is a SAE if and only if it admits an admissible proper regular pairing.
	\end{proposition}
	\begin{proof}
	    The proof of \cite[Proposition 12]{MossSae} directly generalizes. The extra admissibility hypothesis guarantees that only admissible horns are being filled.
	\end{proof}
	Now, since SAEs are trivial cofibrations, we can prove \cref{PropositionExPFibrantReplacement} by exhibiting a suitable pairing. We will achieve this by decomposing the map ${K}\to\Ex_P(K)$ into two maps, and exhibiting a presentation for each of the factors.
	
	\begin{definition}
	Let ${K}\in \sS_P$ be a stratified simplicial set. Its naive subdivision is $\sdn(K)=(\sd(K),\varphi_K\circ\lv)$. This defines a functor
	\begin{equation*}
	    \sdn\colon\sS_P\to\sS_P
	\end{equation*}
	which admits a right adjoint, $\Exn\colon \sS_P\to\sS_P$. The later can also be described as a simplicial subset $\Exn(K)\subset\Ex(K)$
	\begin{equation*}
	    \Exn(K)=\{\sigma\colon \sd(\Delta^{\J})\to K\ |\ \text{$\sigma$ is stratum preserving}\}
	\end{equation*}
	Now, let $\J$ be some flag $[p_0\leq\dots\leq p_n]$, define the following map on vertices:
	\begin{align*}
	    t_{\J}\colon \sd_P(\Delta^{\J})&\to \sdn(\Delta^{\J})\\
	    (q_0\leq\dots\leq q_m,r)&\mapsto (q_0\leq\dots\leq q_l)
	\end{align*}
	    where $q_l=r$ and $q_{l+1}>r$. This extends to a map of stratified simplicial sets, $t_{\J}\colon \sd_P(\Delta^{\J})\to \sdn(\Delta^{\J})$, and further into a natural transformation $t\colon \sd_P\to\sdn$. 
	\end{definition}
	
Note that the natural transformation $t\colon \sd_P\to\sdn$ fits into a commutative diagram
	    \begin{equation}\label{diagSdSplit}
	        \begin{tikzcd}
	        \sd_P\arrow{rr}{\lv_P}
	        \arrow[swap]{dr}{t}
	        && 1_{\sS_P}
	        \\
	        &\sdn
	        \arrow[swap]{ur}{\lv}
	        \end{tikzcd}
	    \end{equation}
	which, by adjunction, gives the following diagram
	\begin{equation}\label{diagExSplit}
			\begin{tikzcd}
				1_{\sS_P}	\arrow[rd ]\arrow[rr] & & \Ex_P  \\
				& \Exn \arrow[ru]& , 
			\end{tikzcd}
		\end{equation}
	where the horizontal arrow is the map of interest. We will show that the maps to, and from $\Exn$ are SAEs. In particular, this will imply that their composition is an SAE, which in turn implies \cref{PropositionExPFibrantReplacement}. We start by showing that they are cofibrations.
	\begin{lemma}\label{lem:Candidate_FSAE_Are_Cofibrations}
		All natural transformations in Diagram \eqref{diagExSplit} are cofibrations.
	\end{lemma}
	\begin{proof}
		It suffices to show, that all of the natural transformations in \eqref{diagSdSplit} are epimorphisms of simplicial sets when evaluated at $\Delta^{\mathcal J}$. For $\lv$ this is well known. Thus, by closedness of epimorphisms under composition, it remains to show that $t_{\mathcal J}$ is an epimorphism, for each flag $\mathcal J = [p_0 \leq\dots\leq p_n]$. Label $e_0,\dots,e_n$ the vertices of $\Delta^{\J}$, and let $\sigma=(\sigma_0,\dots,\sigma_k)$ be a simplex in $\sdn(\Delta^{\J})$. For $0\leq i\leq k$, let 
		$q_i := \max \{ p_j \mid e_j \in \sigma_i \}$.
		Note that $q_i$ specifies the stratum to which the vertex $\sigma_i$ of $\sdn (\Delta^{\J})$ belongs. For $q \in P$, denote $j_q=\min\{i\ | \ q_i=q\}$. Next, define  $\widetilde{\sigma}\subset \Delta^{\J}$ as $\widetilde{\sigma}=\{e_i\ | \ e_i\in \sigma_{j_{p_i}}\}$. In other words, $\widetilde{\sigma}$ is the simplex given by such vertices $e$, which lie in the smallest $\sigma_i$ which, considered as a simplex of $\sdn(\Delta^{\J})$, has the same stratum as $e$. Note, that $\widetilde{\sigma}$ is built such that for each $q_i$ it contain some vertex of stratum $q_i$. In particular, 
		\[\sigma'=((\sigma_0\cup \widetilde{\sigma},q_0),\dots,(\sigma_k\cup\widetilde{\sigma},q_k))
		\]
		defines a simplex of $\sd_P(\Delta^{\J})$. An elementary computation shows that $t(\sigma')=\sigma$. 
		\end{proof}

		\begin{proposition}\label{prop:KtoExNSAE}
		Let ${K}\in \sS_P$ be a stratified simplicial set, the map ${K}\hookrightarrow\Exn(K)$ is an SAE.
		\end{proposition}
		
		\begin{proof}
		    By \cref{lem:Candidate_FSAE_Are_Cofibrations} we already know that it is a cofibration, and so we need to exhibit a proper regular admissible pairing. Consider the (non-stratified) composition 
		    \begin{equation*}
		        {K}\to\Exn(K)\to\Ex(K).
		    \end{equation*}
		    We know of a proper regular pairing for the composition from \cite[Theorem 22]{MossSae}. By \cref{LemmaRestrictFSAE}, it is enough to show that this pairing correctly restricts to $\Exn(K)$, and that the restricted pairing is admissible. Note that in \cite{MossSae}, the pairing $T$ is defined by pre-composing with certain maps $r^k_n\colon \sd(\Delta^{n+1})\to\sd(\Delta^{n})$, for $0\leq k\leq n$ defined on vertices as follows (see \cite[Definition 25]{MossSae}, the pairing is defined immediatly afterwards). First, for vertices of $\sd(\Delta^{n+1})$ of the form $\{i\}$, one has :
		    \begin{equation*}
		        r^k_n(\{i\})=\left\{\begin{array}{cl}
		        \{i\}&\text{ if $0\leq i\leq k$}\\
		        \{0,\dots,k\} &\text{ if $i=k+1$}\\
		        \{i-1\} &\text{ if $i>k+1$}
		        \end{array}
		        \right.
		    \end{equation*}
		    And, then for an arbitrary vertex $\sigma\in\sd(\Delta^{n+1})$, one has
		    \begin{equation*}
		        r^k_n(\sigma)=\bigcup\limits_{i\in\sigma}r^k_n(\{i\})
		    \end{equation*}
		    Now let $\J=[p_0\leq\dots\leq p_n]$ be some flag, and $\J^k$ the flag obtained by repeating $p_k$. Then, note that $r^k_n$ gives a stratum preserving map
		    \begin{equation*}
		        r^k_n\colon\sdn(\Delta^{\J^k})\to\sdn(\Delta^{\J})
		    \end{equation*}
		    Now, if $f\colon \sd(\Delta^n)\to K$ is some simplex of type II which happens to be in $\Exn(K)$, then there is some flag $\J$ such that $f\colon \sdn(\Delta^{\J})\to K$ is a stratum preserving map. Furthermore, one has $T(f)=f\circ r^k_n$ for some $0\leq k\leq n$, but then, $f\circ r^k_n\colon \sdn(\Delta^{\J^k})\to {K}$ is also a stratum preserving map. Which implies that $T$ restricts to a proper and regular pairing to the inclusion ${K}\to\Exn(K)$. Finally, the pairing is admissible since the horn inclusion $\Lambda^{\J^k}_k\to\Delta^{\J^k}$ is admissible, by definition of $\J^k$.
		\end{proof}
	\begin{lemma}\label{LemmaRestrictFSAE}
		Suppose one is given two cofibrations of simplicial sets $m_1: B_{0} \hookrightarrow B_1$, $m_2: B_1 \hookrightarrow B_2$. Then, a proper  regular pairing $T$ on $m_2 \circ m_1$  restricts to proper  regular pairings on $m_1$ and on $m_2$ if and only if we have $T(B_{II} \cap B_{1,n.d.}) \subset B_{1,n.d.}$. This also holds for stratified simplicial sets and admissible pairings.
	\end{lemma}
		\begin{proof}
		Note first that
		properness and admissibility are clearly sustained under restriction. As any subset of a well founded set is well founded, the same holds for regularity. Hence, the only thing to verify is that $T$ restricts to a bijection. The condition on $T$ guarantees that the restriction of $T$ to $B_{II} \cap B_{1,n.d.} \to B_{I} \cap B_{1,n.d.}$  is well defined. Injectivity is automatic, and surjectivity follows from the fact that for $\sigma\in B_{I}$, $T^{-1}\sigma$ is always a face of $\sigma$. In particular, if $\sigma\in B_{1,I}$, then $T^{-1}(\sigma)\subset\sigma\subset B_1$. Finally, the restriction of $T$ on $B_{1,\nd}$ gives a proper regular pairing. This also implies that the restriction of $T$ on $B_2\setminus B_1$ is a proper regular pairing.
	\end{proof}
	We are left with showing, that $\Exn(K) \hookrightarrow \Ex_P (K)$ is an SAE for $ K \in \sS_P$. We will do so in two steps. We will first construct some sub-object of $\Exn(K)$, $\widehat{J}$, such that $\widehat{J}\hookrightarrow \Ex_P(K)$ is an SAE, following Moss's proof in \cite{MossSae}, and then use \cref{LemmaRestrictFSAE} to show that it restricts to the desired SAE. To describe $\widehat{J}$ and the pairing, we need to introduce some maps on subdivisions.
	
	\begin{definition}
	Let $n\geq 0$, and $0\leq k\leq n$. Define maps $\widetilde{j}^k_n\colon \sd(\Delta^n)\to\sd(\Delta^n)$ and $\widetilde{r}^k_n\colon \sd(\Delta^{n+1})\to\sd(\Delta^n)$ as follows. On vertices of the form $\{i\}$, they are defined as 
	\begin{equation*}
	    \widetilde{j}^k_n(\{i\})=\left\{\begin{array}{cl}
	    \{i,\dots,n\}&\text{ if $i<k$}\\
	    \{i\} &\text{ if $i\geq k$}
	    \end{array}\right.\text{ and } \widetilde{r}^{k}_n(\{i\})=\left\{\begin{array}{cl}
	    \{i\} &\text{ if $i<k$}\\
	    \{k,\dots,n\} &\text{ if $i=k$}\\
	    \{i-1\} &\text{ if $i>k$}
	    \end{array}\right.
	\end{equation*}
	On vertices $\sigma$, they are defined as 
	\begin{equation*}
	    \widetilde{j}^k_n(\sigma)=\bigcup\limits_{i\in\sigma}\widetilde{j}^k_n(\{i\})\text{ and }
	    \widetilde{r}^k_n(\sigma)=\bigcup\limits_{i\in\sigma}\widetilde{r}^k_n(\{i\})
	\end{equation*}
	And then, they are extended linearly to maps of simplicial sets. Then, given a flag $\J= [p_0\leq\dots\leq p_n]$, the product maps
	\begin{align*}
	    \widetilde{j}^k_n\times\Id\colon\sd(\Delta^n)\times N(P)&\to\sd(\Delta^n)\times N(P)\\
	    \widetilde{r}^k_n\times\Id\colon\sd(\Delta^{n+1})\times N(P)&\to\sd(\Delta^n)\times N(P)
	\end{align*}
	Restrict to stratum preserving maps
		\begin{align*}
	    j^k_n\colon\sd_P(\Delta^{\J})&\to\sd_P(\Delta^{\J});\\
	    r^k_n\colon\sd_P(\Delta^{\J^k})&\to\sd_P(\Delta^{\J}),
	\end{align*}
	where $\J^k$ is the flag obtained from $\J$ by repeating $p_k$.
	
	For ${K}$ some fixed stratified simplicial set, $n\geq 0$, and $0\leq k\leq n$, let $J^k_n$ be the subset of $\Ex_P(K)_n$ defined as follows
	\begin{equation*}
	    J^k_n=\{\sigma\colon \sd_P(\Delta^{\J})\to{K}\ |\ \sigma\circ j^k_n=\sigma\}
	\end{equation*}
	\end{definition}
	
	The maps $j^k_n$ and $r^k_n$ satisfy a lot of relations, somewhat similar to the simplicial relations.
	
	\begin{proposition}\label{PropTonsOfEq}
		The morphisms $ j^k,  r^k$ fulfill the following equations (for any fixed flag $\mathcal J = [p_0\leq\dots\leq p_n]$ as the target of all the compositions)
			\begin{align}
				\label{PropTonsOfEq4}\tag{1'} r^k \circ \sd_{P}(d^k) &= 1_{\Delta^{\mathcal J}} &{0 \leq k \leq n}\\ \label{PropTonsOfEq5}\tag{2'}
				 j^k \circ  r^k \circ \sd_{P}(d^i) \circ  j^{k+1} &=  j^k \circ  r^k \circ \sd_{P}(d^i) &{0 \leq k < i \leq n} \\\label{PropTonsOfEq6}\tag{3'}
				 r^k \circ \sd_P(d^i) &= \sd_P(d^i) \circ  r^{k-1} &{0 \leq i < k \leq n}
			\end{align}
			\begin{align}\label{PropTonsOfEq7}\tag{4'}
				 r^k \circ  j^{h+1} &=  j ^{h} \circ  r^k &{0 \leq h \leq k \leq n} \\
				\label{PropTonsOfEq8}\tag{5'} j^{k} \circ \sd_P(d^i) \circ  j^{k-1} &=  j^{k} \circ \sd_{P}(d^i)	&{1\leq k\leq n,\text{ } 0\leq i \leq n} 
				\\
				\label{PropTonsOfEq9}\tag{6'}
				 j^h \circ  r^k &=  j^h \circ \sd_P(s^k) &{1\leq k < h \leq n}
			\end{align}
			\begin{align}
				\label{PropTonsOfEq10}\tag{7'}
				 j^k \circ  r^k \circ  r^{k+1} &=  j^k \circ  r^k \circ \sd_P(s^k) &{0\leq k \leq n}
				\\
				\label{PropTonsOfEq11}\tag{8'}
				\sd_P(s^h) \circ  j^k \circ r^k &=  j^k \circ  r^k \circ \sd_P(s^{h+1}) &{0\leq k \leq h \leq n} \\ 
				\sd_P(s^h) \circ  j^{k+1} \circ r^{k+1} &=  j^k \circ r^k \circ \sd_P(s^{h}) &{0\leq h \leq k \leq n}\tag{9'}
			\end{align}
	\end{proposition}
	\begin{proof}
	The maps $j^k_n$ and $r^k_n$ are obtained from the corresponding maps in \cite{MossSae} by conjugating them with the automorphisms $D_n\colon\sd(\Delta^n)\to\sd(\Delta^n)$ sending $\{i\}$ to $\{n-i\}$. 
	Equation $(1')$ through $(9')$ are then obtained by conjugating Moss' equations $(1)$ through $(9)$ \cite[Lemma 26]{MossSae}.
	\end{proof}
	
	\begin{lemma}\label{lem:JHat_SubSimplicialSet}
	The subset $\widehat{J}\subset\Ex_P(K)$, defined as
	\begin{equation*}
	  \widehat{J}_n=J^n_n 
	\end{equation*}
	is a simplicial subset. Furthermore, $\widehat{J}\subset\Exn(K)$.
	\end{lemma}
	
	\begin{proof}
	    For the first part, it suffices to show that if $\sigma\colon\sd_P(\Delta^{\J})\to{K}$ is in $J^n_n$, then for all $0\leq k\leq n$, we have both
	    \begin{equation*}
	        \sigma\circ\sd(d^k)\in J^{n-1}_{n-1} \text{ and } \sigma\circ \sd (s^k)\in J^{n+1}_{n+1}.
	    \end{equation*}
	    The former comes from equality \eqref{PropTonsOfEq8}, while the later follows from the following equality:
	    \begin{equation}\label{eq:Extra_Equation_j_k}
	        \tag{10'} j^{n}_n\circ \sd_P(s^k)\circ j^{n+1}_{n+1}=j^n_n\circ \sd_P(s^k)
	    \end{equation}
	    Given the definitions of the $j^k$, it is enough to check that \cref{eq:Extra_Equation_j_k} holds on vertices of $\sd(\Delta^{n+1})$ of the form $\{i\}$, for the $\widetilde{j}$ and $\sd(s^k)$. But evaluating the transformed equation on $\{i\}$ gives on both sides either $\{i,\dots,n\}$, if $k\geq i$, or $\{i-1,\dots,n\}$ if $k<i$, which concludes the first part of the proof.
	    
	    For the second part, Let $\J=[p_0\leq\dots\leq p_n]$, and consider the factorization problem
	    \begin{equation*}
	        \begin{tikzcd}
	        \sd_P(\Delta^{\J})\arrow{rr}{j^n}
	        \arrow{dr}{t_{\J}}
	        &&\sd_P(\Delta^{\J})
	        \\
	        &\sdn(\Delta^{\J})
	        \arrow[dashed]{ur}{f}
	        \end{tikzcd}
	    \end{equation*}
	    One checks that the map $f$, defined on vertices as $f(\sigma)=(\widetilde{j}^n_n(\sigma),\max\{p_i\ |\ i\in \sigma\})$, makes the diagram commute. But then, any simplex $\sigma\in \widehat{J}$ must satisfy $\sigma=\sigma\circ j^n=\sigma\circ f\circ t_{\J}$. In particular, such a simplex is in $\Exn(K)$, which concludes the proof.
	\end{proof}
	
	\begin{lemma}\label{lem:JHat_FSAE}
	The inclusion $\widehat{J}\to\Ex_P(K)$ is an SAE.
	\end{lemma}
	
	\begin{proof}
	    Moss' proof of \cite[Theorem 22]{MossSae} directly translates into a proof that the inclusion of \cref{lem:JHat_FSAE} is an SAE. Note that due to the conjugation with $D_n$, the inclusion relations between the $J^k_n$ are reversed from those in \cite{MossSae}. Additionally, the key difference is that in our setting, $\widehat{J}$ does not coincide with ${K}$. Nevertheless, \cite[Lemmas 27, 28 and 29]{MossSae}  generalize in our setting since they are direct consequences of the equalities of \cite[Lemma 26]{MossSae}, which also hold in our context in the form of \cref{PropTonsOfEq}.
	    
	    To see that the pairing from \cite{MossSae} is admissible, let $\J=[p_0\leq\dots\leq p_n]$ be some flag, and $\sigma\colon \sd_P(\Delta^{\J})\to{K}$ a type $II$ simplex. Then, there is some $k\geq 0$ such that $T(\sigma)=\sigma\circ r^k$, and $\sigma=d_k T(\sigma)$. In particular, $T(\sigma)$ is of the form $\sd_P(\Delta^{\J^k})\to {K}$, with $\J^k$ obtained from $\J$ by repeating $p_k$. This means that $\Lambda^{\J^k}_k\to\Delta^{\J^k}$ is an admissible horn inclusion, and so the pairing is admissible.
	\end{proof}
	
	\begin{proposition}\label{prop:ExNtoExSAE}
	Let ${K}\in \sS_P$ be a stratified simplicial set, the map $\Exn(K)\to\Ex_P(K)$ is an SAE.
	\end{proposition}
	
	\begin{proof}
	    By \cref{LemmaRestrictFSAE}, it is enough to show that the pairing given in the proof of \cref{lem:JHat_FSAE} correctly restricts to $\Exn(K)$. Let $\J= [p_0\leq\dots\leq p_n]$ be a flag, and $\sigma\colon \sdn(\Delta^{\J})\to{K}$ a type II simplex in $\Exn(K)$. Its image under the pairing $T$ is the simplex of $\Ex_P(K)$, $\sigma\circ t_{\J}\circ r^k$, for some $0\leq k\leq n$, and we need to find some $\tau\in \Exn(K)$ such that $\sigma\circ t_{\J}\circ r^k=\tau\circ t_{\J^k}$, where $\J^k$ is obtained from $\J$ by repeating $p_k$. Consider the following diagram:
	    \begin{equation*}
	        \begin{tikzcd}
	        \sd_P(\Delta^{\J^k})
	        \arrow{r}{r^k}
	        \arrow[swap]{d}{t_{\J^k}}
	        &\sd_P(\Delta^{\J})
	        \arrow{d}{t_{\J}}
	        \\
	        \sdn(\Delta^{\J^k})
	        \arrow[dashed]{r}{g}
	        &\sdn(\Delta^{\J})
	        \end{tikzcd}
	    \end{equation*}
	    It suffices to find a stratum preserving map $g\colon \sdn(\Delta^{\J^k})\to\sdn(\Delta^{\J})$ making the diagram commute, since then $\tau=\sigma\circ g$ would satisfy $\sigma\circ t_{\J}\circ r^k=\tau\circ t_{\J^k}$. We define $g$ on vertices as follows. 
	    \begin{equation*}
	        g(\mu)=\{i\in\widetilde{r^k_n}(\mu)\ |\ p_i\leq \max\{p_j\ |\ j\in \mu\}\}
	    \end{equation*}
	    for $\mu$ a vertex of $\sd(\Delta^{n+1})$. An elementary computation now gives commutativity.
	\end{proof}
	\cref{prop:KtoExNSAE} and \cref{prop:ExNtoExSAE} together complete the proof of \cref{PropositionExPFibrantReplacement}. Indeed, we have proven that the map $K\to \Exn(K)\to\Ex_P(X)$ is the composition of two SAE, hence it is a SAE.

	\subsection{$\Sing_P$ characterizes weak equivalences}
	
	\cref{PropositionExPFibrantReplacement} has the following immediate consequence.

	\begin{theorem}\label{theo:SingP_Characterize_WeakEquivalences}
		Let $f\colon {X}\to{Y}$ be a map in $\Top_P$. It is a weak equivalence if and only if $\Sing_P(f)\colon \Sing_P(X)\to\Sing_P(Y)$ is a weak equivalence in $\sS_P$. The analgous result holds for $\Sing_{N(P)}$ and maps in $\Top_{N(P)}$
	\end{theorem}
	
	\begin{proof}
	By \cite{douteau2021stratified} we know that $\Ex_P\Sing_P$ is the right functor of a Quillen-equivalence. In particular, if $f\colon {X}\to{Y}$ is a map in $\Top_P$, since all objects of $\Top_P$ are fibrant, it is a weak equivalence if and only if $\Ex_P\Sing_P(f)\colon \Ex_P\Sing_P(X)\to\Ex_P\Sing_P(Y)$ is a weak equivalence in $\sS_P$. Now consider the following commutative diagram:
		\begin{equation*}
			\begin{tikzcd}
				\Sing_P(X)
				\arrow{r}
				\arrow{d}
				&\Sing_P(Y)
				\arrow{d}
				\\
				\Ex_P\Sing_P(X)
				\arrow{r}
				&\Ex_P\Sing_P(Y)
			\end{tikzcd}
		\end{equation*}
		By \cref{PropositionExPFibrantReplacement}, the vertical maps are trivial cofibrations and in particular weak equivalences. By two out of three, this implies that $\Sing_P(f)$ is a weak equivalence if and only if $\Ex_P\Sing_P(f)$ is a weak equivalence if and only if $f$ is a weak equivalence. The proof for $\Sing_{N(P)}$ is identical.
	\end{proof}

\section{Realizations characterize weak-equivalences}
\label{Section:Real_Char_WE}	
In this section, we prove that a stratum preserving simplicial map $f:K \to L$ in $\sS_P$ is a weak equivalence if and only if $\RealP{f}$ is a weak equivalence in $\Top_P$. In light of the results in \cite{douteau2021stratified} which establishes that $\RealP{\sd_P{(-)}}$ is the left functor in a Quillen equivalence, and those in \cite{douSimp} which imply that $\sd_P(K)\to K$ is a weak-equivalence in $\sS_P$, the main part is to show that $\RealP{-}$ does preserve weak equivalences.
We show this result by obtaining all but the first of the weak equivalences in Diagram \eqref{eq:Diagram_holinks_link_prelim}.
\subsection{Comparing the simplicial links and the geometrical links}
We start with a comparison between the simplicial link, $\Real{\Link{\I}(-)}$, and the geometrical link, $\Real{-}_b$, see \cref{def:Simp_Link,def:Geom_Link}.

\begin{proposition}
\label{prop:Links_and_Links}
    Let $K\in \sS_P$ be a stratified simplicial set and $A \subset K$ a stratified simplicial subset. 
    Let $\I$ be a regular flag in $P$, and $b \in \Real{\Delta^{\mathcal I}}$ the barycenter. Then, there is a commutative diagram in $\Top$
    \begin{center}
        \begin{tikzcd}
        \Real{\Link{\I}(A)} \arrow[r, hook] \arrow[d, "\sim"]& \Real{\Link{\I}(K)} \arrow[d, "\sim"] \\
        \Real{A}_{b} \arrow[r, hook] & \Real{K}_b 
        \end{tikzcd}
    \end{center}
    where the vertical maps are homeomorphisms.
\end{proposition}

The proof of the proposition relies on the following lemma, which is a direct consequence of \cite[Theorem 2.3.2]{waldhausen2000spaces}.

\begin{lemma}\label{lem:FPmagic}
Let $h_P\colon \Real{\sd(N(P))}\to \Real{N(P)}$ be the usual homeomorphism between a simplicial complex and its barycentric subdivision.
Let $\fil K \in \sS_P$ be a stratified simplicial set. There exists a homeomorphism $h_K\colon \Real{\sd(K)}\to \Real{K} $, which induces an isomorphism in $\Top_{N(P)}$ :
\begin{equation*}
    h_K\colon (\Real{\sd(K)},h_P\circ\Real{\sd(\varphi_K)})\to\RealNP{\fil{K}} 
\end{equation*}
Furthermore, the homeomorphism $h_K$ restricts to homeomorphisms on all non-degenerate simplices. In particular, if $\fil{A}\subset\fil{K}$ is a stratified simplicial subset, then the restriction of $h_K$ induces an isomorphism in $\Top_{N(P)}$
\begin{equation*}
(h_K)_{|A}\colon(\Real{\sd(A)},h_P\circ\Real{\sd(\varphi_A)}) \to \RealNP{\fil A} 
\end{equation*}
\end{lemma}
\begin{proof}
Consider the simplicial map $\varphi_K\colon K\to N(P)$. Its codomain is a non-singular simplicial set, and so by \cite[Theorem 2.3.2]{waldhausen2000spaces}, there exist a homeomorphism $h_K$ such that the following diagram commutes:
\begin{equation*}
    \begin{tikzcd}
    \Real{\sd(K)}
    \arrow{r}{h_K}
    \arrow[swap]{d}{\Real{\sd(\varphi_K)}}
    &\Real{K}
    \arrow{d}{\Real{\varphi_K}}
    \\
    \Real{\sd(N(P))}
    \arrow{r}{h_P}
    &\Real{N(P)}
    \end{tikzcd}
\end{equation*}
The commutativity of the diagram implies that $h_K$ is an isomorphism in $\Top_{N(P)}$, which is the first part of the lemma. The second part of the lemma is the content of \cite[Proposition 2.3.23]{waldhausen2000spaces}.
\end{proof}

\begin{proof}[Proof of \cref{prop:Links_and_Links}]
Let $A\subset K$ be an inclusion of stratified simplicial sets. Since the realization functor $\Real{-}\colon \sS\to \Top $ preserve pullbacks, $\Real{\Link{\I}(K)}$ is given by the following pullback:
\begin{equation*}
    \begin{tikzcd}
    \Real{\Link{\I}(K)}
    \arrow[hookrightarrow]{r}
    \arrow{d}
    &\Real{\sd(K)}
    \arrow{d}{\Real{\sd(\varphi_K)}}
    \\
    \{*\}
    \arrow[hookrightarrow]{r}{\Real{i_{\I}}}
    &\Real{\sd(N(P))}
    \end{tikzcd}
\end{equation*}
On the other hand, $\Real{K}_b$ is defined as the following pullback:
\begin{equation*}
    \begin{tikzcd}
    \Real{K}_b
    \arrow[hookrightarrow]{r}
    \arrow{d}
    &\Real{K}
    \arrow{d}{\Real{\varphi_K}}
    \\
    \{b\}
    \arrow[hookrightarrow]{r}
    &\Real{N(P)}
    \end{tikzcd}
\end{equation*}
By \cref{lem:FPmagic}, $h_K$ and $h_P$ give an isomorphism between the arrows on the right hand side of both squares. On the other hand, $h_P(\Real{i_{\I}}(*))=b$, since $b$ is the barycenter of $\Real{\Delta^{\I}}$, which means that there is an isomorphism between the pullback squares, producing the isomorphism $\Real{\Link{\I}(K)}\to\Real{K}_b$. Since $h_K$ correctly restricts to simplicial subsets, by \cref{lem:FPmagic}, we obtain the left side of the square in \cref{prop:Links_and_Links} as well as its commutativity.
\end{proof}

\subsection{From links to homotopy links}\label{subsec:links_to_homotopy_links}
We begin by showing that, up to homotopy, the geometric link given by $\Real{K}_b$ agrees with $\HolINP(\RealNP{K})$.
\begin{lemma}\label{lem:Holinks_are_Links}
Let $K$ be a stratified simplicial set, $\I$ a regular flag, and $y=(y_0,\dots,y_n)$ a point in the interior of $\Real{\Delta^{\I}}$. The inclusion $\{y\}\hookrightarrow \RealNP{\Delta^{\I}}$ induces a weak-equivalence
\begin{equation}\label{Eq:Holink_to_Link}
r_y\colon\HolINP(\RealNP{K})\to \Real{K}_y.
\end{equation}
\end{lemma}	
The proof will take the remainder of the subsection. However, we begin with an immediate corollary.
\begin{corollary}\label{prop:Links_are_Holinks}
Let $K$ and $L$ be two stratified simplicial sets, and $f\colon \RealNP{K}\to\RealNP{L}$ any map in $\Top_{N(P)}$. Then $f$ is a weak-equivalence if and only if $f$ induces weak-equivalences $\Real{K}_b\to\Real{L}_b$
for all barycenters of simplices in $\Real{N(P)}$. 
\end{corollary}
\begin{proof}
By definition of the model structure on $\Top_{N(P)}$, $f:X \to Y$ is a weak-equivalence if and only if $f$ induces weak-equivalences $\HolINP(X)\to \HolINP(Y)$, for all regular flags, $\I$. Thus, it is enough to show that the evaluation at $b$ map
\begin{equation*}
\HolINP(\RealNP{K})\to\Real{K}_b
\end{equation*} 
is a weak equivalence.
This is the content of \cref{lem:Holinks_are_Links}.
\end{proof}

We will actually prove that the map \eqref{Eq:Holink_to_Link} is a  homotopy equivalence. To define an inverse, we need a way to define a map $\RealNP{\Delta^{\I}}\to \RealNP{K}$ from the data of a point in $\Real{K}_b$. This is the purpose of the next lemma.
\begin{lemma}
Let $K$ be a stratified simplicial set, and $\I$ a regular flag. Let $\Real{K}_{\Int(\Real{\Delta^{\I}})}$ be the  pre-image of $\Int(\Real{\Delta^{\I}})$ under $\RealNP{\varphi_K}$ and consider $ \Real{K}_{\Int(\Real{\Delta^{\I}})}\times \RealNP{\Delta^{\I}}$ as a strongly stratified space via the projection on the second factor. 
Then there exists a map in $\Top_{N(P)}$ 
\begin{equation*}
\rho\colon\Real{K}_{\Int(\Real{\Delta^{\I}})}\times \RealNP{\Delta^{\I}}\to \RealNP{K},
\end{equation*}
such that the restriction of $\rho$ to the fiber-product $\Real{K}_{\Int(\Real{\Delta^{\I}})}\times_{\Real{\varphi_K}} \RealNP{\Delta^{\I}}$ coincides with the projection on the first factor, and extends in this way to $\Real{K}\times_{\Real{\varphi_K}} \RealNP{\Delta^{\I}}$.
\end{lemma}
\begin{proof}
Any point in $\Real{K}_{\Int(\Real{\Delta^{\I}})}$ can be described uniquely as a pair $(\sigma,(t_0,\dots,t_m))$, where $\sigma\colon \Delta^{m}\to K$ is a non-degenerate simplex such that $\varphi_K(\Delta^m)$ is some degeneracy of $\Delta^{\I}$, and $(t_0,\dots,t_m)$ is a point in the standard $m$-simplex (that is, a $(m+1)$-tuple satisfying $0\leq t_i\leq 1$, for all $i$ and $\sum t_i=1$). We can relabel the $(m+1)$-tuple $(t_0,\dots,t_m)$ as $(t^0_0,\dots,t^0_{k_0},t^1_0,\dots,t^n_{k_n})$ so that the coordinates of $\Real{\varphi_K}(\sigma,(t^0_0,\dots,t^0_{k_0},t^1_0,\dots,t^n_{k_n}))$ are $(\sum t^0_i, \sum t^1_i,\dots,\sum t^n_i)$. The condition that $\Real{\varphi_K}(\sigma,(t^0_0,\dots,t^0_{k_0},t^1_0,\dots,t^n_{k_n}))\in \Int\Real{\Delta^{\I}}$ thus guarantees that $\sum t^k_i>0$, for all $0\leq k\leq n$. Now, given a point in $\RealNP{\Delta^{\I}}$, $(q_0,\dots,q_n)$, let $\rho$ be the continuous map:
\begin{align*}
\rho\colon\Real{K}_{\Int(\Real{\Delta^{\I}})}\times \RealNP{\Delta^{\I}}&\to \RealNP{K}\\
((\sigma,(t_0,\dots,t_m)),(q_0,\dots,q_n))&=(\sigma,(\frac{q_0}{\sum_i t^0_i}t^0_0,\dots,\frac{q_0}{\sum_i t^0_i}t^0_{k_0},\frac{q_1}{\sum_i t^1_i}t^1_0,\dots,\frac{q_n}{\sum_i t^n_i}t^n_{k_n}))
\end{align*}
For a point in the fibered product, one must have $\sum_it^l_i=q_l$ for all $0\leq l\leq n$, and $\rho$ extends on it as:
\begin{align*}
\Real{K}\times_{\Real{\varphi_K}}\Real{\Delta^{\I}}&\to \Real{K}\\
((\sigma,(t^0_0,\dots,t^0_{k_0},t^1_0,\dots,t^n_{k_n})),(q_0,\dots,q_n))&\mapsto(\sigma,(t^0_0,\dots,t^0_{k_0},t^1_0,\dots,t^n_{k_n}))
\end{align*}
\end{proof}	
\begin{proof}[proof of \cref{lem:Holinks_are_Links}]
Let $y=(y_0,\dots,y_n)$ be a point in $\Int(\Real{\Delta^{\I}})$. We show that the evaluation at $y$ \[r_y\colon\HolINP(\RealNP{K})=\mathcal C^0_{N(P)}(\RealNP{\Delta^{\I}},\RealNP{K})\to \Real{K}_y\]
 is a homotopy equivalence. We define a section $\Real{K}_y\to \HolINP(K)$ as follows
\begin{align*}
c_y\colon \Real{K}_y &\to \HolINP(K)\\
x&\mapsto  \left\{\begin{array}{ccl}
\RealNP{\Delta^{\I}}&\to &\RealNP{K}\\
(q_0,\dots,q_n)&\mapsto &\rho(x,(q_0,\dots,q_n))
\end{array}\right.
\end{align*}
By the previous lemma, one has $r_y\circ c_y=\Id_{\Real{K}_y}$. To construct the homotopy in the other direction, we will use the following (non-stratified) identification:
\begin{align*}
\Real{\partial\Delta^{\I}}\times [0,1] /\Real{\partial\Delta^{\I}}\times \{0\}&\simeq \Real{\Delta^{\I}}\\
((a_0,\dots,a_n),s)&\mapsto (sa_0+(1-s)y_0,\dots,sa_n+(1-s)y_n)
\end{align*}
We now define the homotopy from $c_y\circ r_y$ to the identity.
\begin{align*}
H\colon \HolINP(\RealNP{K})\times [0,1]&\to \HolINP(\RealNP{K})\\
(f,t)&\mapsto \left\{\begin{array}{ccl}
\RealNP{\Delta^{\I}}&\to &\RealNP{K}\\
(\bar{a},s)&\mapsto &\rho (f(\bar{a},st),(\bar{a},s))
\end{array}\right.
\end{align*}
Note that $H$ is well defined, since when $t\not =1$, $f(\bar{a},st)$ is in $\Real{K}_{\Int(\Real{\Delta^{\I}})}$, for all $(\bar{a},s)\in \Real{\Delta^{\I}}$, and for $t=1$, either $s\not =1$, and $f(\bar{a},s)\in \Real{K}_{\Int(\Real{\Delta^{\I})}}$, or $s=1$ and $(f(\bar{a},1),(\bar{a},1))\in \Real{K}\times_{\Real{\varphi_K}}\Real{\Delta^{\I}}$, and $\rho$ is defined as the projection on the first factor.
\end{proof}
 \subsection{ $\RealNP{-}$ preserves weak equivalences}
Next, we use \cref{prop:Links_are_Holinks} to derive the following result.
\begin{proposition}\label{prop:RealNP_preserves_we}
The functor
  \begin{equation*}
      \RealNP{-}: \sS_P \to \TopNP
  \end{equation*}
  preserves weak equivalences.
  \end{proposition}
  \begin{proof}
 Note first, that by Ken Brown's Lemma, it suffices to show, that $\RealNP{-}$ sends trivial cofibrations into weak equivalences. Now assume that $f\colon A\hookrightarrow K$ is a trivial cofibration. By  \cref{prop:Links_are_Holinks} it is enough to show that for any $b\in \Real{N(P)}$, the barycenter of some simplex $\Delta^{\I}$, the inclusion $\Real{A}_b\to\Real{K}_b$ is a weak-equivalence. By \cref{prop:Links_and_Links} this is equivalent to showing that the inclusion $\Real{\Link{\I}(A)}\hookrightarrow\Real{\Link{\I}(K)}$ is a weak-equivalence. But since the functor $\Real{-}$ preserves weak-equivalences, it is enough to show that $\Link{\I}(A)\hookrightarrow\Link{\I}(K)$ is a weak-equivalence, which follows from \cref{lem:Link_preserve_trivial_cofibrations}.
\end{proof}
\begin{lemma}\label{lem:Link_preserve_trivial_cofibrations}
 The functor $\Link{\I}{}: \sS_P \to \sS$ preserve trivial cofibrations.
\end{lemma}
\begin{proof}
It is sufficient to prove that $\Link{\I}$ sends strong anodyne extensions to strong anodyne extensions (see \cref{Section:FSAE}) because in both model structures, trivial cofibrations are given by retracts of strong anodyne extensions.
    The functor
    $\Link{\I}{}$ is constructed by first applying a right adjoint (subdivision) and then pulling back to a vertex of $\sd(N(P))$. In particular, it is given by the composition of two functors preserving colimits and therefore
    preserves all colimits itself. Thus, it suffices to show that $\Link{\I}{}$ sends admissible horn inclusions into (strong) anodyne extensions in $\sS$.
    Further, if we show the statement for admissible horn inclusions up to a certain dimension $n$, then it automatically follows that $\Link{\I}$ sends all stratified strong anodyne extensions $ A  \to B$ with $B$ of dimension lesser or equal to $n$ into strong anodyne extensions.
    
    We proceed via induction over $n$. 
    The case $n=0$ is trivial. So let $\J$ be a flag in $P$ of length $n+1$ and $k \leq n+1$ such that $\Lambda_k^{\J} \hookrightarrow \Delta^{\J}$ is an admissible horn inclusion. 
    We show that $\Link{\I}{(\Lambda^\J_k)} \hookrightarrow \Link{\I}{(\Delta^\J)}$ is a strong anodyne extension. 
    Let $\tau\colon \Delta^{\J}\to\Delta^{\J}$ be the maximal non-degenerate simplex of $\Delta^{\J}$ and $\tau_k\colon\Delta^{\J'} \hookrightarrow \Delta^{\J}$ be its $k$-th face.
    If $\J$ does not degenerate from $\I$, then by construction neither does $\J'$, since $\Lambda^{\J}_k$ is admissible. In particular, $\Link{\I}{(\Delta^{\J})}$ does not contain the vertices corresponding to $\tau$ and $\tau_k$.
    These are precisely the vertices of $\sd(\Delta^\J)$ that are missing in $\sd(\Lambda^{\mathcal J}_k)$. Hence, $\Link{\I}{(\Lambda^{\mathcal J}_k \hookrightarrow \Delta^\J)}$ is a bijection and there is nothing to show.
    
    Now, if $\J$ degenerates from $\I$, then so does $\J'$ and hence both $\tau_k$ and $\tau$ define vertices in $\Link{\I}{(\Delta^\J)}$.
    Let $D$ be the full subcomplex of $\sd{(\Delta^\J)}$ spanned by all vertices, but the one corresponding to $\tau_k$, and let $D_{\I}$ be its intersection with $\Link{\I}{(\Delta^{\J})}$.
    We obtain a pairing on the inclusion $A:=D_{\I} \hookrightarrow \Link{\I}{(\Delta^{\J})}=:B$ by taking $B_I$ to be the set of non-degenerate simplices in $B_{n.d.}\setminus A_{n.d.}$, of the form $[\sigma_0,..., \sigma_{m-1},\tau_k,\tau]$, and $B_{II}$ the set of simplices of the form $[\sigma_0,\dots,\sigma_{m-1},\tau_k]$, for $m\geq 0$. 
    Then, the map sending $[\sigma_0,..., \sigma_{m-1},\tau_k]$ to $[\sigma_0,..., \sigma_{m-1},\tau_k,\tau]$ is a proper regular pairing $B_{II} \to B_I$. Thus, $D_{\I} \hookrightarrow \Link{\I}{(\Delta^\J)}$ is a SAE. 
    Now, $D_{\I}$ is a cone on $\Link{\I}{(\Lambda^{\J}_k)}$ with conepoint $\tau$.
    Hence by a simplicial set version of \cite[Corollary p. 249]{whitehead1939simplicial} 
    it is enough to find a strong anodyne extension $\Delta^{0}\to \Link{\I}{(\Lambda^{\J}_k)}$
    to conclude that $\Link{\I}{(\Lambda_{k}^{\J})} \hookrightarrow D_{\I}$
    is also a strong anodyne extension. Note that since $\Lambda^{\J}_k$ is an admissible horn and $\J$ degenerates from $\I$, any stratum-preserving inclusion $\Delta^{\I}\hookrightarrow \Lambda^{\J}_k$
    is a strong anodyne extension. Passing to the links, we get a map $\Delta^{0}=\Link{\I}{(\Delta^{\I})}\hookrightarrow\Link{\I}{(\Lambda^{\J}_k)}$, which is a strong anodyne extension by the inductive hypothesis. 
    Finally, we have found two strong anodyne extensions $\Link{\I}(\Lambda^{\J}_k)\to D_{\I}$ and $D_{\I}\to \Link{\I}(\Delta^{\J})$, which compose to give the desired strong anodyne extension.
\end{proof}
\subsection{Comparing homotopy links in $\Top_{N(P)}$ and $\Top_P$}

In this section, we prove the following theorem.
\begin{theorem}\label{theo:Strong_Holinks_are_Holinks}
    Let $K$ be a stratified simplicial set, and $\I$ be some regular flag. The natural map
    \begin{equation*}
        \HolINP(\RealNP{K})\to \HolIP(\RealP{K})
    \end{equation*}
    is a weak-equivalence.
\end{theorem}

Given that the proof of the above theorem is fairly long and somewhat involved, we give a brief summary here. First, we show \cref{lem:Holinks_On_Finite_Subsets,lem:I_equals_P}, allowing us to make two simplifying assumptions: that $\I=P$, and that $K$ is locally finite. Then, we show (in \cref{lem:Holink_comparison_factors}) that up to homotopy, the map $\HolINP(\RealNP{K})\to \HolIP(\RealP{K})$ factors through $\HolIP(\RealP{K}^{\redu})$, where $\RealP{K}^{\redu}\subset \RealP{K}$ is a subspace obtained by taking away some of the least singular simplices of $K$ (see \cref{def:K_redu,fig:DeltaI_Red}). We then prove that the map $\HolINP(\RealNP{K})\to \HolIP(\RealP{K}^{\redu})$ is an homotopy equivalence (\cref{lem:Holink_comparison_factors}). This simplifies the problem because now we need to compare two homotopy links in $\Top_P$,  $\HolIP(\RealP{K}^{\redu})$ and $\HolIP(\RealP{K})$. To compare these two, we will decompose the inclusion $\RealP{K}^{\redu}\subset \RealP{K}$ into $\RealP{K}^{\redu,k}\subset\dots\subset\RealP{K}^{\redu,0}= \RealP{K}$, see \cref{def:K_redu_l}, and prove that each of those induces a homotopy equivalence between homotopy links. Intuitively, homotopy inverses are obtained by sending a map $f\colon \RealP{\Delta^{\I}}\to \RealP{K}^{\redu,l}$ to some homotopic map $g\colon \RealP{\Delta^{\I}}\to \RealP{K}^{\redu,l+1}\subset\RealP{K}^{\redu,l}$, where the homotopy between $f$ and $g$ comes from precomposing $f$ by some homotopy $H\colon\RealP{\Delta^{\I}}\times [0,1]\to \RealP{\Delta^{\I}}$. This homotopy should stratifiedly homotope the identity on $\RealP{\Delta^{\I}}$ into a smaller space, which $f$ maps into $\RealP{K}^{\redu,l+1}$. As one might expect there exists no single homotopy $H$ that works for arbitrary $f\in \HolIP(\RealP{K}^{\redu,l})$. Instead, we first define some homotopy $S^l\colon \RealP{\Delta^{\I}}\times [0,1]\to \RealP{\Delta^{\I}}$ (\cref{def:Map_Sl}), as well as a parameter $\alpha\colon \HolIP(\RealP{K}^{\redu,l})\times [0,1]\to [0,1]$ (\cref{lem:Constructing_alpha_partition_unity}). For any given $f$, the homotopy we are looking for is then obtained by combining the homotopy $S^l$ and the parameter $\alpha$. Note that a partition of unity in $\HolIP(\RealP{K}^{\redu,l})$ appears in the construction of $\alpha$. Such a partition of unity exists thanks to the assumption that $K$ is locally finite, see \cref{rem:Compact_Open_Topology_Holinks}.

Before moving on to the proof, we will need the following simplifying assumptions:
\begin{enumerate}[label= As.(\arabic*)]
    \item \label{Item:Assumption1_Section4}$K$ is a locally finite simplicial set. This ensures that \[\HolIP(\RealP{K})=\C^0_P(\RealP{\Delta^{\I}},\RealP{K}) \]
    is metrizable (when equipped with the compact open topology, see \cref{rem:Compact_Open_Topology_Holinks}). \cref{lem:Holinks_On_Finite_Subsets} implies that we can assume that $K$ is locally finite without loss of generality.
    \item \label{Item:Assumption2_Section4}$P=\I$. \cref{lem:I_equals_P} implies that we can assume this is true without loss of generality.
\end{enumerate}

\begin{lemma}\label{lem:Holinks_On_Finite_Subsets}
Let $K$ be a stratified simplicial set, and let $A$ be the category of finite stratified simplicial subsets, $K^{\alpha}\subset K$, and inclusions. Then, for all $n\geq 0$ and all pointings, the following natural maps are isomorphisms:
\begin{align*}
\lim_{\substack{\to\\\alpha\in A}}\pi_n(\HolINP(\RealNP{K^{\alpha}}),*)&\to\pi_n(\HolINP(\RealNP{K}),*),\\
\lim_{\substack{\to\\\alpha\in A}}\pi_n(\HolIP(\RealP{K^{\alpha}}),*)&\to\pi_n(\HolIP(\RealP{K}),*).\\
\end{align*}
\end{lemma}

\begin{proof}
    Through adjunction, one sees that an element in the homotopy group $\pi_n(\HolINP(\RealNP{K}),x)$ is nothing more than a map  $S^n\times\RealNP{\Delta^{\I}}\to \RealNP{K}$, in $\Top_{N(P)}$, sending $*\times\RealNP{\Delta^{\I}}$ to $x$. In particular, such a map only reaches a finite simplicial subset of $K$, which implies the surjectivity of the map under study. The same observation for homotopies implies the injectivity.
\end{proof}

\begin{lemma}\label{lem:I_equals_P}
Let $K$ be a stratified simplicial set. Let $\widetilde{K}$ be the simplicial subset defined as the following pullback:
\begin{equation*}
    \begin{tikzcd}
    \widetilde{K}
    \arrow{r}
    \arrow{d}
    &K
    \arrow{d}{\varphi_K}
    \\
    \Delta^{\I}
    \arrow{r}
    &N(P) \spaceperiod
    \end{tikzcd}
\end{equation*}
The inclusion $\widetilde{K}\to K$ induces weak-equivalences
\begin{align*}
    \HolINP(\RealNP{\widetilde{K}})&\to \HolINP(\RealNP{K}),\\
        \HolIP(\RealP{\widetilde{K}})&\to \HolIP(\RealP{K}).
\end{align*}
\end{lemma}

\begin{proof}
    The first map is even an isomorphism. Indeed, Let $f\colon \RealNP{\Delta^{\I}}\to\RealNP{K}$ be a map in $\Top_{N(P)}$. Since geometric realization preserves pullbacks, such a map must factor through $\RealNP{\widetilde{K}}$. 
    
    For the second map, consider the following subspace of $\RealP{K}$:
    \begin{equation*}
        Z=\{x\in \RealP{K} \ |\ \varphi_P \circ \Real{\varphi_K}(x)\in \I\}
    \end{equation*}
    By construction, one has $ \HolIP(Z)\cong \HolIP(\RealP{K})$, so it is enough to show that the inclusion $\iota\colon\RealP{\widetilde{K}}\to Z$ is a stratified homotopy equivalence. Now, recall that the points of $\RealP{K}$ can be described by pairs $(\sigma\colon \Delta^{\J}\to K,\xi=(\xi_0,\dots,\xi_n))$, where $\J=[p_0\leq \dots\leq p_n]$ is some flag. With those notations, one can give an alternative description of $Z$ as
    \begin{equation*}
        Z=\{(\sigma,\xi)\in \RealP{K}\ |\ p_{m_\xi}\in \I\ \},
    \end{equation*}
    where $ m_\xi=\max\{i \ |\ \xi_i\not=0\}$.
    Given a point $(\sigma,\xi)$ in $Z$, let $\xi_{\I}=\sum_{p_i\in \I}\xi_i$, and let $\epsilon_i$ be $0$ if $p_i\not\in \I$ and $1$ if $p_i\in \I$. With this notation, we can define a stratified retract, $r\colon Z\to \widetilde{K}$, at the level of simplices :
    \begin{align*}
        r_{\sigma}\colon Z\cap \sigma(\Delta^{\J})&\to \Real{\widetilde{K}}\cap \sigma(\Delta^{\J})\\
        (\xi_0,\dots,\xi_n)&\mapsto \frac{1}{\xi_{\I}}(\epsilon_0\xi_0,\dots,\epsilon_n\xi_n)
    \end{align*}
    One checks easily that those maps are compatible with faces and degeneracies, and produce a global retract $r\colon Z\to \widetilde{K}$, satisfying $r\circ i=\Id_{\widetilde{K}}$. Furthermore, if $(\xi_0,\dots,\xi_n)\in Z\cap \sigma(\Delta^{\J})$, then $p_{m_{\xi}}\in \I$ and $\epsilon_{m_{\xi}}=1$. This implies that \[p_{m_{\xi}}=\varphi_P\circ\Real{\varphi_K}(\xi_0,\dots,\xi_n)=\varphi_P\circ\Real{\varphi_K}(r(\xi_0,\dots,\xi_n))=p_{m_{\xi}},\] 
    and so $r$ is a stratified map. Finally, the straight-line homotopies between $\iota_{\sigma}\circ r_{\sigma}$ and $\Id_{Z\cap\sigma(\Delta^{\J})}$ assemble to produce a stratified homotopy between $\iota\circ r$ and $\Id_{Z}$
\end{proof}

In the remainder of the subsection, we use the following notations: $\J=[p_0\leq\dots\leq p_n]$ is a flag, and $\I=[q_0<\dots<q_k]$ is the regular flag such that $\I=\{p_0\leq\dots\leq p_n\}$. In other words, $\Delta^{\J}$ is degenerated from $\Delta^{\I}$. Furthermore, if $p\in P$, let $\J_p=\{i\ |\ 0\leq i\leq n,\ p_i=p\}$. Then, for a point $(\xi_0,\dots,\xi_n)\in \Real{\Delta^{\J}}$, let $\xi_p=\sum_{i\in \J_p}\xi_i$. Note that by construction, the stratification $\Real{\Delta^{\J}}\to \Real{N(P)}$ is given by the map $(\xi_0,\dots,\xi_n)\mapsto (\xi_{q_0},\dots,\xi_{q_k})$.
\begin{definition}
\label{def:K_redu}
Let $K\in \sS_P$ be a stratified simplicial set. We define the subspace $\RealP{K}^{\redu}\subset \RealP{K}$ as follows:
\begin{equation*}
    \RealP{K}^{\redu}=\{(\sigma,\xi)\ |\ \xi_p=0\Rightarrow \xi_{p'}=0\ \forall p'\geq p\}
\end{equation*}
Equivalently, $\RealP{K}^{\redu}$ is defined as the following pullback
\begin{equation*}
    \begin{tikzcd}
    \RealP{K}^{\redu}
    \arrow[hookrightarrow]{r}
    \arrow{d}
    &\RealP{K}
    \arrow{d}{\Real{\varphi_K}}
    \\
    \RealP{\Delta^{\I}}^{\redu}
    \arrow[hookrightarrow]{r}
    &\RealP{N(P)}\spaceperiod
    \end{tikzcd}
\end{equation*}
\end{definition}

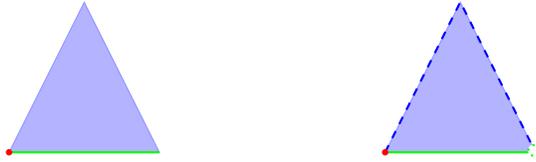
\begin{figure}[h]
\begin{tikzpicture}
\filldraw[blue, opacity=0.3](0,0)--(1,2)--(2,0)--(0,0);
\draw[green,thick] (0,0)--(2,0);
\filldraw[red] (0,0) circle (1pt);

\filldraw[shift={(5,0)},blue, opacity=0.3](0,0)--(1,2)--(2,0)--(0,0);
\draw[shift={(5,0)},blue,dashed,thick] (0,0)--(1,2)--(2,0);
\draw[shift={(5,0)},green,thick] (0,0)--(2,0);
\filldraw[shift={(5,0)},red] (0,0) circle (1pt);
\filldraw[shift={(5,0)},white] (1.975,0.025) circle (2pt);
\draw[shift={(5,0)},green, thick, dotted] (1.975,0.025) circle (2pt);

\end{tikzpicture}
\caption{The stratified space $\RealP{\Delta^{\I}}$ and the subspace $\RealP{\Delta^{\I}}^{\redu}$, with $\I=[p_0<p_1<p_2]$. The dashed blue lines and the green circle indicate the missing faces.}
\label{fig:DeltaI_Red}
\end{figure}

\begin{lemma}
\label{lem:Holink_comparison_factors}
Let $K$ be a stratified simplicial set. There exists a homotopy equivalence, $\lambda^{*}\colon \HolINP(\RealNP{K})\to \HolIP(\RealP{K}^{\redu})$ making the following diagram commute up to homotopy:
\begin{equation}\label{Eq:Holinks_commute_up_to_homotopy}
\begin{tikzcd}
\HolINP(\RealNP{K})
\arrow[dashed]{dr}{\lambda^{*}}
\arrow{rr}
&&\HolIP(\RealP{K})
\\
&\HolIP(\RealP{K}^{\redu})
\arrow[hookrightarrow]{ur}
\end{tikzcd}    
\end{equation}
\end{lemma}

\begin{proof}
    Let $\I=[q_0<\dots<q_k]$ be a regular flag, and consider the linear map $\lambda\colon \Real{\Delta^{\I}}\to\Real{\Delta^{\I}}$ sending the vertex $q_i$ to the barycenter of $\Delta^{[q_0<\dots<q_i]}\subset \Delta^{\I}$. We will show that the map 
    \begin{align*}
       \lambda^* \colon \HolINP(\RealNP{K})&\to \HolIP(\RealP{K}^{\redu})\\
       f & \mapsto \Big \{
       x \mapsto (f \circ \lambda)(x)\Big\} 
    \end{align*}
    has the desired properties. First note that the map $\lambda\colon \RealP{\Delta^{\I}}\to\RealP{\Delta^{\I}}$ is a stratum preserving map, which means that for any $f\in \HolINP(\RealNP{K})$, $f\circ\lambda\in \HolIP(\RealP{K})$. Furthermore, by construction, $\lambda(\Real{\Delta^{\I}})\subset \Real{\Delta^{\I}}^{\redu}$. In particular, $\lambda^*$ takes values in $\HolIP(\RealP{K}^{\redu})$, and it is well-defined. 
    
    Now, let $H\colon \RealP{\Delta^{\I}}\times[0,1]\to \RealP{\Delta^{\I}}$ be the straight-line homotopy between $\lambda$ and $\Id_{\Real{\Delta^{\I}}}$. Then the map
    \begin{align*}
        H^*\colon\HolINP(\RealNP{K})\times [0,1]&\to \HolIP(\RealP{K})\\
        (f,s)&\mapsto \Big\{\xi\mapsto f(H(\xi,s))\Big\}
    \end{align*}
    gives the homotopy between the two paths in Diagram \eqref{Eq:Holinks_commute_up_to_homotopy}.
    
    It remains to be shown that $\lambda^*$ is a homotopy equivalence. Consider the map $\rho\colon \RealP{K}^{\redu}\times_P\Real{N(P)}\to\RealNP{K}$ from \cref{lem:Straightening_V2}. Using the natural isomorphism $\HolIP(\RealP{K}^{\redu})\simeq \HolINP(\RealP{K}^{\redu}\times_P\Real{N(P)})$, $\rho$ induces a map
    \begin{equation*}
        \rho_*\colon \HolIP(\RealP{K}^{\redu})\cong\HolINP(\RealP{K}^{\redu}\times_P\Real{N(P)})\to \HolINP(\RealNP{K})
    \end{equation*}
We will show that $\rho_*$ is a homotopy inverse to $\lambda^*$. We first consider the composition $\rho_*\circ\lambda^*$. Consider the homotopy $H^*$ defined above. For $f\in \HolINP(\RealNP{K})$, $H^*(f,s)$ lands in $\HolIP(\RealP{K}^{\redu})\cong \HolINP(\RealP{K}^{\redu}\times_P\Real{N(P)})$ if $s<1$ and lands in $\HolINP(\RealNP{K})$ for $s=1$. Let $\widehat{\Real{K}}$ be the strongly stratified space defined as the following union:
\begin{equation*}
    \widehat{\Real{K}}=\RealP{K}^{\redu}\times_P\Real{N(P)}\cup \RealNP{K}\times_{\Real{N(P)}}\Real{N(P)}\subset \Real{K}\times_P\Real{N(P)}\subset\Real{K}\times\Real{N(P)}
\end{equation*}
where the stratification is given by projecting on the second factor. The homotopy $H^*$ factors through a map
\begin{equation*}
    H^*\colon\HolINP(\RealNP{K})\times [0,1]\to \HolINP(\widehat{\Real{K}}).
\end{equation*}
By \cref{lem:Straightening_V2}, $\rho$ extends to $\widehat{\Real{K}}$ as the identity on $\RealNP{K}$. In particular, the composition $\rho_*\circ H^*$ gives a homotopy between $\rho_*\circ \lambda^*$ and $\Id$.

For the other composition, $\lambda^*\circ\rho_*$, consider the homotopy in $\Top_P$, from \cref{lem:Straightening_V2}
\begin{equation*}
    R\colon \left(\RealP{K}^{\redu}\times_P\Real{N(P)}\right)\times [0,1]\to \RealP{K}.
\end{equation*}
Given a map in $\Top_{N(P)}$, $f\colon \RealNP{\Delta^{\I}}\to\RealP{K}^{\redu}\times_P\Real{N(P)}$, and $s\in [0,1]$, consider the composition
\begin{equation*}
    \begin{tikzcd}
    \RealP{\Delta^{\I}}
    \arrow{r}{H_s}
    &\RealP{\Delta^{\I}}
    \arrow{r}{f}
    &\RealP{K}^{\redu}\times_P\Real{N(P)}
    \arrow{r}{R_s}
    &\RealP{K}
    \end{tikzcd}
\end{equation*}
If $s>0$, $R_s$ takes value in $\RealP{K}^{\redu}$. On the other hand, if $s=0$, $R_0=\rho$, $H_0=\lambda$, and $f\circ\lambda$ takes value in $\RealP{K}^{\redu}\times_P\RealP{\Delta^{\I}}^{\redu}$. By \cref{Rem:Rho_Sends_red_to_red}, this implies that $R_0\circ f\circ H_0$ takes value in $\RealP{K}^{\redu}$. Now consider the following homotopy:
\begin{align*}
    G\colon\HolINP(\RealP{K}^{\redu}\times_P\Real{N(P)})\times[0,1]&\to \HolIP(\RealP{K}^{\redu})\\
    (f,s)&\mapsto R_s\circ f\circ H_s
\end{align*}
At $s=1$, $H_1=\Id$, and $R_1$ is the projection to $\RealP{K}^{\redu}$. In particular, $G_1$ is the natural homeomorphism $\HolINP(\RealP{K}^{\redu}\times_P\Real{N(P)})\simeq \HolIP(\RealP{K}^{\redu})$. At $s=0$, $G_0$ is $\lambda^*\circ \rho_*$, precomposed with the aforementioned natural homeomorphism. In particular, $\lambda^*$ is a homotopy equivalence.
\end{proof}

\begin{lemma}\label{lem:Straightening_V2}
There exists a map in $\Top_{N(P)}$ 
\begin{equation*}
    \rho\colon \RealP{K}^{\redu}\times_P\Real{N(P)}\to \RealNP{K},
\end{equation*}
which extends continuously to $\RealNP{K}\times_{\Real{N(P)}}\Real{N(P)}$ as $\Id_{\Real{K}}$, where the stratification on the domain is given by the projection on the second factor. 
Furthermore, there exists a homotopy in $\Top_P$
\begin{equation*}
    R\colon \left(\RealP{K}^{\redu}\times_P\Real{N(P)}\right)\times [0,1]\to \RealP{K}
\end{equation*}
between $\rho$ and the projection to $\RealP{K}^{\redu}$, such that for all $s>0$, $R_s$ takes value in $\RealP{K}^{\redu}$.
\end{lemma}
\begin{remark}\label{Rem:Rho_Sends_red_to_red}
Note that, since the map $\rho\colon \RealP{K}^{\redu}\times_P\Real{N(P)}\to \RealNP{K}$ from \cref{lem:Straightening_V2} is a strongly stratified map, its restriction to $\RealP{K}^{\redu}\times_P\RealP{\Delta^{\I}}^{\redu}$ must take value in $\RealP{K}^{\redu}=\Real{\varphi_K}^{-1}(\RealP{\Delta^{\I}}^{\redu})$.
\end{remark}

\begin{proof}[Proof of \cref{lem:Straightening_V2}]
    Recall that if $(\xi_0,\dots,\xi_n)\in \Real{\Delta^{\J}}$, with $\J=[p_0\leq\dots\leq p_n]$, and $p\in P$, then $\xi_p=\sum_{p_i=p}\xi_i$. Now, write $\I=[q_0<\dots<q_k]$. By \labelcref{Item:Assumption2_Section4}, $P=\I$, which means that for all $0\leq i\leq n$ there exists $0\leq j\leq k$ such that $p_i=q_j$. Furthermore, we will label the points in $\Real{N(P)}=\Real{\Delta^{\I}}$ as $(t_{q_0},\dots,t_{q_k})$. We will construct $\rho$ on simplices, let $\sigma\colon \Delta^{\J}\to K$ be a stratified simplex. Define $\rho_{\J}$ as :
    \begin{align*}
        \rho_{\J}\colon \RealP{\Delta^{\J}}^{\redu}\times_P\Real{N(P)}&\to \RealNP{\Delta^{\J}}\\
        \big ((\xi_0,\dots,\xi_n),(t_{q_0},\dots,t_{q_k}) \big )&\mapsto (\alpha_0\xi_0,\dots,\alpha_n\xi_n) 
    \end{align*}
    where $\alpha_i=\frac{t_{p_i}}{\xi_{p_i}}$, if $\xi_{p_i}\not =0$, and $\alpha_i=0$ if $\xi_{p_i}=0$. Now, if $\xi_q=0$ for some $q\in P$, since $(\xi_0,\dots,\xi_n)$ is in $\RealP{\sigma(\Delta^{\J})}^{\redu}$, this implies that $\xi_{q'}=0$, for $q'\geq q$. In turn, $\varphi_P(t_{q_0},\dots,t_{q_k})=\varphi_P\circ\Real{\varphi_K}(\xi_0,\dots,\xi_n)<q$. In particular, in this case, one must have $t_{q'}=0$ for $q'\geq q$, and $\sum \alpha_i\xi_i=\sum t_{q_j}=1$, which implies that $\rho_{\J}$ is well defined. If one sets $(\xi'_0,\dots,\xi'_n)=(\alpha_0\xi_0,\dots,\alpha_n\xi_n)$, one has $\xi'_q=t_q$ for all $q\in P$. In particular, $\rho_{\J}$ is a map in $\Top_{N(P)}$. Furthermore, if we restrict $\rho_{\J}$ to pairs $((\xi_0,\dots,\xi_n),(t_{q_0},\dots,t_{q_k}))$ such that $\Real{\varphi_{K}}(\xi_0,\dots,\xi_n)=(t_{q_0},\dots,t_{q_k})$, then, for all $0\leq i \leq n$, either $\xi_{p_i}\not = 0$ and $\alpha_i=1$, or $\xi_{p_i}=0$, and in both cases, $\alpha_i\xi_i=\xi_i$. Using this $\rho_{\J}$ can be extended to $\RealNP{\Delta^{\J}}\times_{\Real{N(P)}}\Real{N(P)}$ by the identity (or, more precisely, the projection on the first factor). Assembling all the $\rho_{\J}$ together gives the desired map.
    
    Now, to produce the homotopy $R$, consider the following straight-line homotopy defined on simplices, where the $\alpha_i$ are defined as above:
    \begin{align*}
        R_{\J}\colon \left(\RealP{\Delta^{\J}}^{\redu}\times_P\Real{N(P)}\right)\times [0,1] &\to \RealP{\Delta^{\J}}\\
          (  (\xi_0,\dots,\xi_n  ),  (t_{q_0},\dots,t_{q_k}  ),s  )&\mapsto 
        (  (s+  (1-s  )\alpha_0  )\xi_0,\dots,  (s+  (1-s  )\alpha_n  )\xi_n )
    \end{align*}
    To see that $R_{\J}$ is a stratum preserving map, notice that if $s=0$, $R_{\J}$ is just $\rho_{\J}$ which is a map in $\Top_{N(P)}$. 
    If $s>0$, $(s+(1-s)\alpha_i)\not=0$ for all $0\leq i\leq n$ and so $(s+(1-s)\alpha_i)\xi_i=0\Leftrightarrow\xi_i=0$ for all $0\leq i\leq n$. This implies two things. First, $\varphi_P\circ\Real{\varphi_K}(\xi_0,\dots,\xi_n)=\varphi_P\circ\Real{\varphi_K}((s+(1-s)\alpha_0)\xi_0,\dots,(s+(1-s)\alpha_n)\xi_n)$, and so $R_{\J}$ is a stratified homotopy, and second, $R_{\J}$ takes value in $\RealP{\Delta^{\J}}^{\redu}$ for all $s>0$.
    Finally, notice that for $s=1$, $R_{\J}$ is the projection on $\RealP{\Delta^{\J}}$, and so $R_{\J}$ is a homotopy between $\rho_{\J}$ and this projection. Assembling the $R_{\J}$ gives the desired homotopy.
\end{proof}

In order to prove \cref{theo:Strong_Holinks_are_Holinks}, it remains to be shown that the map induced by the inclusion $\HolIP(\RealP{K}^{\redu})\to\HolIP(\RealP{K})$ is a weak-equivalence. We will do so by further decomposition.

\begin{definition}
\label{def:K_redu_l}
Write $\I=[q_0<\dots<q_k]$, and let $0\leq l\leq k$. Define the following subspace of $\RealP{K}$:
\begin{equation*}
    \RealP{K}^{\redu,l}=\{(\sigma,\xi)\ |\ \xi_p=0\Rightarrow \xi_{p'}=0,\ \forall p'\geq p,\ \forall p<q_l\}.
\end{equation*}
We then have a sequence of inclusions
\begin{equation*}
    \RealP{K}^{\redu}=\RealP{K}^{\redu,k}\subset\dots\subset\RealP{K}^{\redu,0}=\RealP{K}.
\end{equation*}
\end{definition}

\cref{theo:Strong_Holinks_are_Holinks} then follows from the following lemma.

\begin{lemma}\label{lem:Last_Lemma_For_Strong_Holinks}
Let $0\leq l\leq k-1$, the inclusion $\RealP{K}^{\redu,l+1}\to \RealP{K}^{\redu,l}$ induces a homotopy equivalence
\begin{equation*}
  \HolIP(\RealP{K}^{\redu,l+1})\to  \HolIP(\RealP{K}^{\redu,l})
\end{equation*}
\end{lemma}

In order to prove \cref{lem:Last_Lemma_For_Strong_Holinks}, we will need a few more technical results.
\begin{definition}
\label{def:Map_Sl}
Let $0\leq l\leq k-1$, define the map $S^l$ as
\begin{align*}
    S^l\colon \Real{\Delta^{\I}}\times [0,1]&\to \Real{\Delta^{\I}}\\
    ((t_{q_0},\dots,t_{q_k}),s)&\mapsto (t_{q_0},\dots,t_{q_{l-1}},t_{q_l}+(1-s)\sum_{j=l+1}^k t_{q_j},st_{q_{l+1}}
,\dots,st_{q_k}).\end{align*}
\end{definition}

\begin{lemma}\label{lem:Properties_of_S}
The map $S^l$ is stratum preserving outside of $s=0$, and it restricts to the projection on the first factor on the subspace 
\begin{equation*}
\Real{\Delta^{[q_0<\dots<q_l]}}=\{(t_{q_0},\dots,t_{q_k})\ |\ t_{q}=0\ \forall q>q_l\}.
\end{equation*}
\end{lemma}

\begin{proof}
    Let $t=(t_{q_0},\dots,t_{q_k})$, and $s\in [0,1]$, and write $t'=S^l(t,s)$. If $\varphi_P(t)=q\leq q_l$, then $t_{q'}=t'_{q'}=0$ for all $q'>q$, and $S^l(t,s)=t$ for all $s$, this addresses the second part of the lemma. If $\varphi_P(t)=q>q_l$, then for $s>0$, one has $t'_{q}=st_{q}>0$, and $t'_{q'}=0$ for all $q'>q$. In particular, $\varphi_P(t')=q$, and $S^l(-,s)$ is a stratum preserving map.
\end{proof}

\begin{lemma}\label{lem:Xredl_neighborhood}
For all $0\leq l\leq k-1$, $\RealP{K}^{\redu,l+1}$ is a neighborhood of the $q_l$-stratum of $\RealP{K}^{\redu,l}$
\end{lemma}

\begin{proof}
Note that the subspace $\RealP{K}^{\redu,l}$ can be defined as the following pullback square:
\begin{equation*}
    \begin{tikzcd}
    \RealP{K}^{\redu,l}
    \arrow[hookrightarrow]{r}
    \arrow{d}
    &\RealP{K}
    \arrow{d}{\RealP{\varphi_K}}
    \\
    \RealP{\Delta^{\I}}^{\redu,l}
    \arrow{r}
    &\RealP{\Delta^{\I}}
    \end{tikzcd}
\end{equation*}
Now since $\RealP{\varphi_K}$ is a continuous, stratum preserving map, it is enough to show that $\RealP{\Delta^{\I}}^{\redu,l+1}$ is a neighborhood of the $q_l$-stratum of $\RealP{\Delta^{\I}}^{\redu,l}$. Now, a point in the $q_l$-stratum of $\RealP{\Delta^{\I}}^{\redu,l}$ is of the form $(t_{q_0},\dots,t_{q_k})$, with $t_{q_l}\not=0$ and $t_q=0$ for all $q>q_l$. Furthermore, the defining condition of $\RealP{\Delta^{\I}}^{\redu,l}$ implies that $t_q\not =0$ for all $q<q_l$. This means that there exist a neighborhood of $t$, $U\subset \Real{\Delta^{\I}}$, such that for any $t'\in U$, $q\leq q_l\Rightarrow t'_{q}\not =0$. But this means that all points in $U$ satisfy the defining condition of $\RealP{\Delta^{\I}}^{\redu,l+1}$. In particular, $U\subset \RealP{\Delta^{\I}}^{\redu,l+1}$, and the latter is a neighborhood of the $q_l$-stratum of $\RealP{\Delta^{\I}}^{\redu,l}$.
\end{proof}

\begin{lemma}\label{lem:Constructing_alpha_partition_unity}
There exists a continuous map $\alpha\colon \HolIP(\RealP{K}^{\redu,l})\times [0,1]\to [0,1]$, such that
\begin{enumerate}
    \item $\alpha(f,s)=0\Rightarrow s=0$,
    \item $f(S^l(t,\alpha(f,\sum_{j=l+1}^kt_{q_j})))\in \RealP{K}^{\redu,l+1}$, for all $t=(t_{q_0},\dots,t_{q_k})\in \Real{\Delta^{\I}}$ and for all $f\in \HolIP(\RealP{K}^{\redu,l})$. 
\end{enumerate}
\end{lemma}

\begin{proof}
    Let $\epsilon >0$, and define $\Real{\Delta^{\I}}^{\epsilon,l}$ as the following subset of $\Real{\Delta^{\I}}$: 
    \begin{equation*}
        \Real{\Delta^{\I}}^{\epsilon,l}=\{t\in \Real{\Delta^{\I}}\mid \sum_{i=l}^{k}t_{q_i}\geq \epsilon\}
    \end{equation*}
    Now, pick some stratified map $f\colon \RealP{\Delta^{\I}}\to\RealP{K}^{\redu,l}$, and consider the intersection of the following nested family of compact subsets:
    \begin{equation*}
        \bigcap_{\alpha>0}f(S^l( \Real{\Delta^{\I}}^{\epsilon,l}\times [0,\alpha])) =f(S^l(\Real{\Delta^{\I}}^{\epsilon,l}\times\{0\}))
    \end{equation*}
   Note that if $t\in \Real{\Delta^{\I}}^{\epsilon,l}$, $S^l(t,0)$ is in the $q_l$-stratum of $\RealP{\Delta^{\I}}$. In particular, $f(S^l(\Real{\Delta^{\I}}^{\epsilon,l}\times\{0\}))$ is a compact subset of the $q_l$-stratum of $\RealP{K}^{\redu,l}$. Since, by \cref{lem:Xredl_neighborhood}, $\RealP{K}^{\redu,l+1}$ is a neighborhood of the $q_l$-stratum of $\RealP{K}^{\redu,l}$, there must exist $\alpha^{\epsilon}>0$ such that $f(S^l( \Real{\Delta^{\I}}^{\epsilon,l}\times [0,\alpha^{\epsilon}]))\subset \RealP{K}^{\redu,l+1}$. By the definition of the compact-open topology, this also holds for any $g$ in a neighborhood $U$ of $f$ in $\HolIP(\RealP{K}^{\redu,l})$. 
   Hence, we can cover $\HolIP(\RealP{K}^{\redu,l})$ by a family of opens $U_i$ such that there exists $\alpha^{\epsilon}_i>0$ satisfying $g\in U_i\Rightarrow g(S^l( \Real{\Delta^{\I}}^{\epsilon,l}\times [0,\alpha^{\epsilon}_i]))\subset \RealP{K}^{\redu,l+1}$. By \cref{lem:Holinks_On_Finite_Subsets}, we may assume that $K$ is locally finite. But this implies that $\HolIP(\RealP{K}^{\redu,l})$ is metrizable, and hence, paracompact (see also \cref{rem:Compact_Open_Topology_Holinks}). In particular, one can find a partition of unity $(\phi_i)$ subordinated to the open cover $(U_i)$. We can now define the following continuous map:
   \begin{align*}
       \alpha^{\epsilon}\colon \HolIP(\RealP{K}^{\redu,l})&\to [0,1]\\
       f&\mapsto \sum_i\phi_i(f)\alpha_i^{\epsilon}
       \end{align*}
       Note that by construction, for all $f\in \HolIP(\RealP{K}^{\redu,l})$, we have $\alpha^{\epsilon}(f)\leq \alpha_i^{\epsilon}$, for some $i$ such that $f\in U_i$. In particular,  we have 
       \begin{equation}\label{eq:f_S_alpha_Subset_Xlplusone}
       f(S^l( \Real{\Delta^{\I}}^{\epsilon,l}\times [0,\alpha^{\epsilon}(f)]))\subset \RealP{K}^{\redu,l+1}    
       \end{equation}
       
       Now, to construct $\alpha$, we will use the family of functions $\alpha^{\epsilon}$, for $\epsilon=\frac{1}{2^n}$, $n\geq 1$. Consider the covering of $(0,1]$ given by the family $I_n=(\frac{1}{2^{n+1}},\frac{1}{2^{n-2}})$, $n\geq 1$, and the family of closed intervals $J_n=\left[\frac{1}{2^{n}},\frac{1}{2^{n-1}}\right]$, $n\geq 1$. Pick a family of bump functions on $[0,1]$, $\psi_n$, $n\geq 1$ satisfying :
       \begin{itemize}
           \item $\psi_n(s)\in [0,1]$, for all $s\in (0,1]$,
           \item $\psi_n(s)=1$ if $s\in J_n$,
           \item $\psi_n(s)=0$ if $s\not\in I_n$,
       \end{itemize}
       and define the continuous map:
       \begin{align*}
           \alpha\colon \HolIP(\RealP{K}^{\redu,l})\times [0,1]&\to [0,1]\\
           (f,s)&\mapsto s\prod_{n\geq 1}\left(1-\psi_n(s)(1-\alpha^{\frac{1}{2^n}}(f))\right)
       \end{align*}
       Note that for $s\in [0,1]$, there is only a finite number of $n\geq 1$ such that $\psi_n(s)\not=0$, which means that the above product only has a finite amount of non-trivial terms. We need to check that it satisfies both conditions of \cref{lem:Constructing_alpha_partition_unity}. Let $f\in \HolIP(\RealP{K}^{\redu,l})$. The first part is clear, since $\alpha(f,s)=0$ implies that either $s=0$, or that the product is $0$, which is not possible, since it only has a finite number of non-trivial terms which are all non-zero. For the second part, first note that if $s\in J_n$
       \begin{equation}\label{eq:Alpha_Smaller_AlphaEpsilon}
           \alpha(f,s)\leq \alpha^{\frac{1}{2^n}}(f)
       \end{equation}
       Now, if $t\in \Real{\Delta^{\I}}$, then 
       \begin{itemize}
           \item 
       either $\sum_{j=l+1}^kt_{q_j}=0$, but then $t\in \Real{\Delta^{q_0<\dots<q_l}}$, and by \cref{lem:Properties_of_S}, in this case $S^l(t,u)=t$ for any $u\in [0,1]$. In particular, the expression in \cref{lem:Constructing_alpha_partition_unity} reduces to $f(t)$. Now, $f(t)$ is of the form $(\sigma,(\xi_0,\dots,\xi_n))$, but since $f$ is stratum preserving, it must satisfy $\xi_p=0$ for all $p>q_l$. Additionally, since $f$ takes value in $\RealP{K}^{\redu,l}$, we have $\xi_p=0\Rightarrow \xi_{p'}=0$ for $p'\geq p$ and $p<q_l$. Now, since we already now that $\xi_{p'}=0$ for $p'>q_l$, we trivially have the implication $\xi_{q_l}=0\Rightarrow \xi_{p'}=0$ for $p'\geq q_l$. In particular $f(t)\in \RealP{K}^{\redu,l+1}$.
       \item or $\sum_{j=l+1}^kt_{q_j}>0$. Then let $n$ be such that $\sum_{j=l+1}^kt_{q_j}\in J_n=\left[\frac{1}{2^n},\frac{1}{2^{n-1}}\right]$. By definition, $t\in \Real{\Delta^{\I}}^{\epsilon,l}$, for $\epsilon=\frac{1}{2^n}$. Combining equations \eqref{eq:Alpha_Smaller_AlphaEpsilon} and \eqref{eq:f_S_alpha_Subset_Xlplusone}, we get :
       \begin{equation*}
           f(S^l(t,\alpha(f,\sum_{j=l+1}^k t_{q_j})))\in f(S^l(\Real{\Delta^{\I}}^{\epsilon,l}\times [0,\alpha^{\epsilon}(f)]))\subset\RealP{K}^{\redu,l}
       \end{equation*}
       which concludes the proof.
       \end{itemize}
\end{proof}

\begin{proof}[Proof of \cref{lem:Last_Lemma_For_Strong_Holinks}]
Using the map $\alpha\colon \HolIP(\RealP{K}^{\redu,l})\times [0,1]\to [0,1]$ from \cref{lem:Constructing_alpha_partition_unity}, we define the following homotopy:
\begin{align*}
   H\colon \HolIP(\RealP{K}^{\redu,l})\times [0,1]&\to \HolIP(\RealP{K}^{\redu,l})\\
    (f,s)&\mapsto \left\{\begin{array}{cl}
         \RealP{\Delta^{\I}}&\to \RealP{K}^{\redu,l}  \\
         t=(t_{q_0},\dots,t_{q_k})&\mapsto f(S^l(t,(1-s)+s\alpha(f,\sum_{j={l+1}}^kt_{q_j}))) 
    \end{array}\right.
\end{align*}
We first check that $H(f,s)$ is a stratum preserving map. By \cref{lem:Properties_of_S}, $f(S^l(t,u))$ has the correct stratification, except maybe when $u=0$. In the latter case, one must have $s=1$ and $\alpha(f,\sum_{j={l+1}}^kt_{q_j})=0$, but by \cref{lem:Constructing_alpha_partition_unity}, this is only possible when $\sum_{j={l+1}}^kt_{q_j}=0$. This corresponds to $t\in \Real{\Delta^{[q_0<\dots<q_l]}}$, and by \cref{lem:Properties_of_S}, in this case one has $S^l(t,u)=t$ for all values of $u$. We conclude that $H(f,s)$ is indeed a stratum preserving map. Furthermore, its image lies in $\RealP{K}^{\redu,l}$, by construction, which means that $H$ is well-defined.
But now, $H_0$ is the identity map, since $S^l(t,1)=t$ for all $t\in \Real{\Delta^{\I}}$. By the second part of \cref{lem:Constructing_alpha_partition_unity}, for all $f\in \HolIP(\RealP{K}^{\redu,l})$, The image of $H(f,1)\colon \Delta^{\I}\to \RealP{K}^{\redu,l}$ lies in $\RealP{K}^{\redu,l+1}$.  This implies that $H_1$ lands in  $\HolIP(\RealP{K}^{\redu,l+1})$. On the other hand, if $f\in \HolIP(\RealP{K}^{\redu,l+1})$, then $H(f,s)$ lies in $\HolIP(\RealP{K}^{\redu,l+1})$ for all $s\in [0,1]$. This shows that the inclusion $\HolIP(\RealP{K}^{\redu,l+1})\hookrightarrow \HolIP(\RealP{K}^{\redu,l})$ defines a homotopy equivalence (with inverse induced by $H_1$), which concludes the proof.
\end{proof}

\subsection{Realizations characterize weak-equivalences}

We summarize the results of the previous subsections in the following theorem.

\begin{theorem}\label{theo:Link_Summary}
    Let $K$ be a stratifed simplicial set, $\I$ a regular flag and $b$ a point in the interior of $\Real{\Delta^{\I}}$. Then all maps in the following diagram are weak-equivalences:
    \begin{equation*}
    \begin{tikzcd}
   \HolINP(\RealNP{K})
   \arrow{dr}
    \arrow{r}
    &\HolIP(\RealP{K})
    \\
     \Real{\Link{\I}(K)}
    \arrow{r}
    &\Real{K}_b
    \end{tikzcd}
    \end{equation*}
\end{theorem}

\begin{corollary}\label{Cor:Realization_Preserve_Weak_Equivalences}
Let $f\colon K\to L$ be in $\sS_P$. The following assertions are equivalent:
\begin{itemize}
    \item $f$ is a weak-equivalence in $\sS_P$,
    \item $\RealNP{f}\colon \RealNP{K}\to\RealNP{L}$ is a weak-equivalence in $\Top_{N(P)}$,
    \item $\RealP{f}\colon \RealP{K}\to \RealP{L}$ is a weak-equivalence in $\Top_P$.
\end{itemize}
Furthermore, if $g\colon \RealNP{K}\to \RealNP{L}$ is a map in $\Top_{N(P)}$, then it is a weak-equivalence if and only if its image by the functor $\varphi_P\circ - \colon \Top_{N(P)}\to\Top_P$ is a weak-equivalence in $\Top_P$.
\end{corollary}

\begin{proof}
Let us first prove the second part. Let $g\colon \RealNP{K}\to\RealNP{L}$ be a map in $\Top_{N(P)}$. It is a weak-equivalence if and only if, for all regular flags $\I$, the maps induced by $g$, 
\begin{equation*}
    \HolINP(\RealNP{K})\to\HolINP(\RealNP{L})
\end{equation*}
are weak-equivalences. But, by \cref{theo:Strong_Holinks_are_Holinks}, this is equivalent to asking that the following maps are weak-equivalences:
\begin{equation*}
    \HolIP(\RealP{K})\to\HolIP(\RealP{L})
\end{equation*}
This, in turn, is equivalent to $(\varphi_P\circ-)(g)$ being a weak-equivalence in $\Top_P$.

    For the first part of the lemma, we know from \cite[Theorem 1.15]{douteau2021stratified}  that the functor $\RealP{\sd_P(-)}\colon \sS_P\to\Top_{P}$ is the left part of a Quillen-equivalence. In particular it characterizes weak-equivalences between cofibrant objects. Since all objects of $\sS_P$ are cofibrant, this implies that $f$ is a weak-equivalence if and only if $\RealP{\sd_P(f)}$ is a weak-equivalence. Now consider the following commutative diagram:
    \begin{equation*}
        \begin{tikzcd}[column sep=45]
        \RealNP{\sd_P(K)}
        \arrow{r}{\RealNP{\sd_P(f)}}
        \arrow[swap]{d}{\RealNP{\lv_P}}
        &\RealNP{\sd_P(L)}
        \arrow{d}{\RealNP{\lv_P}}
        \\
        \RealNP{K}
        \arrow{r}{\RealNP{f}}
        &\RealNP{L}
        \end{tikzcd}
    \end{equation*}
     By \cite[Lemma A.3]{douSimp}, we know that $\lv_P\colon\sd_P(K)\to K$ is a weak-equivalence for all $K\in\sS_P$, and by \cref{prop:RealNP_preserves_we}, we also know that $\RealNP{-}$ preserves weak-equivalences. This implies that the vertical arrows in the previous diagram are weak-equivalences. By two out of three, this means that $f$ is a weak-equivalence if and only if $\RealNP{f}$ is a weak-equivalence, which concludes the proof.
\end{proof}

\begin{corollary}\label{cor:Real_CP_reflects_WE}
The functor $\RealP{C_P(-)}\colon \Diag_P\to\Top_P$ characterizes all weak-equivalences.
\end{corollary}

\begin{proof}
A map in $\Diag_P$, $f\colon F\to G$ is a weak-equivalence if and only if, for all regular flags $\I$, $F(\I)\to G(\I)$ is a weak-equivalence. On the other hand, the map $\RealP{C_P(f)}\colon \RealP{C_P(F)}\to\RealP{C_P(G)}$ is a weak-equivalence in $\Top_P$ if and only if the map $\HolIP(\RealP{C_P(F)})\to\HolIP(\RealP{C_P(G)})$ is a weak-equivalence for all $\I$. In particular, it is enough to show that for any $F\in \Diag_P$, the natural map $F(\I)\to\Sing(\HolIP(\RealP{C_P(F)}))$ is a weak-equivalence for all regular flags $\I$. Consider the following commutative diagram, where $b$ is the barycenter of $\Real{\Delta^{\I}}$:
\begin{equation*}
    \begin{tikzcd}
    &&\HolIP(\RealP{C_P(F)})
    \\
    \Real{F(\I)}
    \arrow[bend left]{urr}
    \arrow[bend right]{drr}
    \arrow{r}
    &\HolINP(\RealNP{C_P(F)})
    \arrow{ur}
    \arrow{dr}
    \\
    &&\Real{C_P(F)}_b\spaceperiod
    \end{tikzcd}
\end{equation*}
We need to prove that the top map is a weak-equivalence, but by \cref{theo:Link_Summary}, it is enough to show that the bottom map is a weak-equivalence. Now, note that $\Real{C_P(F)}_b\cong\Real{F(\I)}\times\{b\}$, as can be computed directly from the definition of $C_P$. In particular, the bottom map is an isomorphism, which concludes the proof.
\end{proof}

\begin{corollary}\label{cor:CP_reflects_WE}
The functor $C_P\colon \Diag_P\to \sS_P$ characterizes all weak-equivalences.
\end{corollary}

\begin{proof}
Let $f\colon F\to G$ be a map in $\Diag_P$. By \cref{cor:Real_CP_reflects_WE}, $f$ is a weak equivalence if and only if $\RealP{C_P(f)}$ is a weak equivalence. But by \cref{Cor:Realization_Preserve_Weak_Equivalences}, the latter is true if and only if $C_P(f)$ is a weak-equivalence.
\end{proof}

\section{Equivalence of homotopy categories and applications}
\label{Section:Hammock}

The goal of this section is to show that the adjoint functors 
\begin{equation*}
    \RealStrat{-}\colon s\Strat\leftrightarrow\Strat\colon \SingStrat,
\end{equation*} as well as the fiberwise pairs $\RealP{-}\colon \sS_P\leftrightarrow\Top_P\colon\Sing_P$ descend to equivalences between the homotopy categories of stratified simplicial sets and stratified spaces.
We deduce from this that the naive homotopy theory of conically stratified spaces embeds fully faithfully in the homotopy theory of stratified spaces, as well as a simplicial approximation theorem.
\subsection{Equivalence between homotopy theories}
	In this subsection, we prove that the adjunctions $\RealP{-} \dashv \Sing_P$ and $\RealStrat{-} \dashv \SingStrat$ induce equivalences of homotopy categories. As neither is part of a Quillen adjunction, this is not to be understood in terms of derived functors in the sense of model categories. Instead, consider the homotopy categories as explicitly constructed by formally inverting the weak equivalences. All functors involved are shown to preserve weak equivalences and hence induce functors of the homotopy categories. The following theorem then states that these induced functors give equivalences of categories.
	
	\begin{theorem}\label{theo:Equivalence_Simplicial_Homotopy_Category}
	    The adjoint pairs $\RealP{-} \dashv \Sing_P$, and $\RealStrat{-} \dashv \SingStrat$ induce well defined equivalences between homotopy categories,
	    \begin{align*}
	    \RealP{-}\colon\Ho\sS_P&\leftrightarrow\Ho\Top_P\colon\Sing_P\\
	   \RealStrat{-}\colon\Ho s\Strat&\leftrightarrow\Ho\Strat\colon \SingStrat.
	    \end{align*}
	\end{theorem}
	We first prove that the functors of \cref{theo:Equivalence_Simplicial_Homotopy_Category} pass to the homotopy categories.
	\begin{lemma}\label{lem:Real_Sing_pass_homotopy_category}
	 The adjoint pairs $\RealP{-} \dashv \Sing_P$, and $\RealStrat{-} \dashv \SingStrat$ induce well defined functors at the level of homotopy categories.
	\end{lemma}
\begin{proof}
By \cref{Cor:Realization_Preserve_Weak_Equivalences}, the functor $\RealP{-}\colon \sS_P\to \Top_P$ preserves weak equivalences for all posets $P$. By construction of the model structures on $\Strat$ and $s\Strat$, this implies that $\RealStrat{-}$ also preserves all weak-equivalences. Since the homotopy categories are nothing more than the localization of the categories at the classes of weak equivalences,  this implies that $\RealP{-}$ and $\RealStrat{-}$ both induce functors between homotopy categories. By \cref{theo:SingP_Characterize_WeakEquivalences}, $\Sing_P\colon \Top_P\to\sS_P$ also preserve all weak-equivalences, and so the same is true for $\SingStrat\colon \Strat\to s\Strat$, which means that they also induce well defined functors at the level of homotopy categories.
\end{proof} 

To conclude the proof of \cref{theo:Equivalence_Simplicial_Homotopy_Category}, we need the following lemma.

\begin{lemma}\label{lem:Unit_Counit_RealP_SingP_WE}
 Let $X$ be a space stratified over $P$, then the co-unit of the adjunction $\RealP{-} \dashv \Sing_P$,
 \begin{equation*}
     \RealP{\Sing_P(X)}\to X
 \end{equation*}
 is a weak-equivalence in $\Top_P$.
 Let $K$ be a simplicial set stratified over $P$, then the unit of the adjunction $\RealP{-} \dashv \Sing_P$,
  \begin{equation*}
K\to \Sing_P(\RealP{K})
 \end{equation*}
 is a weak-equivalence in $\sS_P$.
 
 This also holds for the unit and co-unit of the adjunction $\RealStrat{-}\dashv \SingStrat$.
\end{lemma}
\begin{proof}
Consider first the adjunction $\sd_P\dashv \Ex_P$. For any simplicial set stratified over $P$, $K$, we have the commutative diagram
\begin{equation*}
    \begin{tikzcd}
    \sd_P(K)
    \arrow{rr}{\lv_P}
    \arrow[swap]{dr}{\sd_P(\iota_K)}
    &&K\spacecomma
    \\
    &\sd_P\Ex_P(K)
    \arrow[swap]{ur}{\epsilon_K}
    \end{tikzcd}
\end{equation*}
where the map $\epsilon_K$ is the co-unit. Now, by \cref{prop:LvP_WeakEquivalence}, $\lv_P$ is a weak-equivalence, and $\sd_P$ preserves weak-equivalences. By \cref{PropositionExPFibrantReplacement}, $\iota_K$ is a weak equivalence, which means that $\sd_P(\iota_K)$ is a weak-equivalence. By two out of three, this implies that $\epsilon_K\colon \sd_P\Ex_P(K)\to K$ is a weak-equivalence. 

Now let $X$ be a space stratified over $P$, and consider the commutative diagram:
\begin{equation*}
\begin{tikzcd}
\RealP{\sd_P\Ex_P\Sing_P (X)}
\arrow{rr}
\arrow[swap]{dr}{\RealP{\epsilon_{\Sing_P(X)}}}
&& X
\\
&\RealP{\Sing_P(X)}
\arrow{ur}
\end{tikzcd}    
\end{equation*}
By \cref{Cor:Realization_Preserve_Weak_Equivalences}, the functor $\RealP{-}$ preserves weak-equivalences, which means that the map $\RealP{\epsilon_{\Sing_P(X)}}$ is a weak-equivalence. Furthermore, the co-unit of the adjunction $\RealP{\sd_P(-)}\dashv \Ex_P\Sing_P$, is also a weak-equivalence, since the adjunction is a Quillen-equivalence \cite[Theorem 1.15]{douteau2021stratified}, and all objects in $\Top_P$ (resp. $\sS_P$) are fibrant (resp. cofibrant). This means that by two out of three the co-unit $\RealP{\Sing_P(X)}\to X$ is a weak-equivalence. 

Now, let $K$ be a simplicial set stratified over $P$, and consider the following composition:
\begin{equation*}
    \RealP{K}\to\RealP{\Sing_P(\RealP{K})}\to\RealP{K},
\end{equation*}
where the first map is the realization of the unit of the adjunction $\RealP{-}\dashv \Sing_P$, and the second map is the co-unit evaluated at $\RealP{K}$. The composition gives the identity, and we have proved that the second map is a weak-equivalence in $\Top_P$, which means that the map $\RealP{K}\to\RealP{\Sing_P(\RealP{K})}$ is a weak-equivalence in $\Top_P$, by two out of three. But since, by \cref{Cor:Realization_Preserve_Weak_Equivalences}, the realization functor characterizes weak-equivalences, this means that the unit $K\to \Sing_P(\RealP{K})$ is a weak-equivalence in $\sS_P$.

For the case of $\RealStrat{-}\dashv\SingStrat$, note that if $X$ is a space stratified over $P$, then $\RealStrat{\SingStrat(X)}=\RealP{\Sing_P(X)}$, by definition, which immediatly gives the proof.
\end{proof}	
	
\begin{proof}[Proof of \cref{theo:Equivalence_Simplicial_Homotopy_Category}]

Consider the natural transformations $\RealP{\Sing_P(-)}\to \Id_{\Top_P}$ and $\Id_{\sS_P}\to \Sing_P(\RealP{-})$. By \cref{lem:Unit_Counit_RealP_SingP_WE} they take value in weak-equivalences. Since the functors $\RealP{-}$ and $\Sing_P$ pass to the homotopy categories, so do the natural transformations $\RealP{\Sing_P(-)}\to \Id_{\Ho\Top_P}$ and $\Id_{\Ho\sS_P}\to \Sing_P(\RealP{-})$. Those now take value in isomorphisms, meaning that $(\RealP{-},\Sing_P)$ gives an equivalence between the homotopy categories. The same argument gives that $(\RealStrat{-},\SingStrat)$ induces an equivalence between the homotopy categories.
\end{proof}	
\begin{remark}
\label{rem:Equiv_Simplicial_Localization}
    \cref{theo:Equivalence_Simplicial_Homotopy_Category} is only stated in terms of homotopy categories because that is all that is needed for the applications of \cref{section:Embedding_classical_stratified_homotopy_theory,Section:Simp_Approx_Standalone}, but a much stronger version holds. Consider the simplicial localization defined by Dwyer and Kan in \cite{DwyerKanCalculating}, in terms of hammocks. \cref{lem:Real_Sing_pass_homotopy_category,lem:Unit_Counit_RealP_SingP_WE} give precisely the hypothesis needed to apply \cite[Corollary 3.6]{DwyerKanCalculating}. In particular, the functors $(\RealStrat{-},\SingStrat)$ and $(\RealP{-},\Sing_P)$ induce Dwyer-Kan equivalences between the simplicial localizations:     \begin{equation*}
        L^Hs\Strat\leftrightarrow L^H\Strat \ \text{ and } \ L^H\sS_P\leftrightarrow L^H\Top_P.
    \end{equation*}
 With the stronger result, one can say that the pairs of adjoint functors, $(\RealStrat{-},\SingStrat)$ and $(\RealP{-},\Sing_P)$ induce equivalences between the homotopy \textbf{theories} of stratified spaces and stratified simplicial sets. In particular, they induce equivalences between the underlying $\infty$-categories, which can be explicitly described as the Dwyer-Kan localizations.
\end{remark}

\subsection{Embedding the classical stratified homotopy category}
\label{section:Embedding_classical_stratified_homotopy_theory}
It follows from \cref{prop:Nonexistence_Appendix} that there exists no model structure on $\Top_P$ which is transported from $\sS_P$, along $\Sing_P\colon\Top_P \to \sS_P$. Nevertheless, it turns out that stratified spaces with fibrant $\Sing_P$ and realizations of stratified simplicial sets behave respectively much like fibrant and cofibrant objects of a model category.
\begin{recollection}\label{rec:ex_of_con}
Recall that particularly nice stratified spaces, such as pseudo manifolds or more generally homotopically and conically stratified spaces, have the right lifting property with respect to realizations of admissible horn inclusions (see \cref{theo:Conical_Fibrant}, and \cite[Theorem A.6.4]{HigherAlgebra}\cite[Proposition 8.1.2.6]{nand2019simplicial}). In other words, such spaces map to fibrant objects under $\Sing_P$.  
\end{recollection}

Recall that for $X,Y\in\Top_P$, $[X,Y]_P$ stands for the set of stratified homotopy classes of stratified maps between $X$ and $Y$ (\cref{def:strat_homotopies_spaces}). similarly, if $K,L\in\sS_P$, $[K,L]_P$ stands for the set of stratified homotopy classes of stratified simplicial maps between $K$ and $L$.
\begin{lemma}\label{lem:fibrants_are_like_simpl_fibrants}
Let $X\in\Top_P$ be a stratified space such that $\Sing_P(X) \in \sS_P$ is fibrant. Then, for any weak equivalence of stratified simplicial sets $f\colon K \to L$ in $\sS_P$, the induced map 
\[
[\RealP{L}, X]_P \to [\RealP{K}, X]_P
\]
is a bijection.
\end{lemma}
\begin{proof}
We have a commutative diagram 
\[
\begin{tikzcd}
  {[\RealP{L}, X]_P} \arrow[d, "\cong"] \arrow[r] & {[\RealP{K}, X]_P \arrow[d, "\cong"]}      \\
{[L, \Sing_P(X)]_P} \arrow[r, "f^*"] & {[K, \Sing_P(X)]_P} \spacecomma
\end{tikzcd}
\]
where the vertical maps are bijection, thanks to the fact that the adjunction $(\RealP{-},\Sing_P)$ is simplicial \cite[Proposition 4.9]{douSimp}.
Since $f\colon K \to L$ is a weak equivalence and $\Sing_P(X)$ is fibrant the lower horizontal is a bijection. Thus, by commutativity of the diagram, so is the upper horizontal, as was to be shown.
 \end{proof}
As an immediate corollary of this lemma we obtain:
\begin{theorem}\label{Theo:Fibrant_Cofibrant}
Let $K \in \sS_{P}$ and ${X} \in \Top_P$ such that $\Sing_P(X)$ is a fibrant object of $\sS_P$. Then, the natural map
\[ [\RealP{{K}}, {X}]_P \to \Ho \Top_P(\RealP{{K}}, {X}) \]
is a bijection. 
\end{theorem}
\begin{proof}
Consider the realization of the last vertex map $\RealP{\sd_P(K)} \xrightarrow{\RealP{\lv_P}} \RealP{K}$. Since $\lv_P$ is a weak equivalence (\cref{prop:LvP_WeakEquivalence}) by \cref{Cor:Realization_Preserve_Weak_Equivalences} so is its realization. Furthermore, it follows from \cite[Theorem 1.15]{douteau2021stratified}, that $\RealP{\sd_P(K)}$ is a cofibrant object in $\Top_P$. Thus, it follows that $\RealP{\lv_P}\colon \RealP{\sd_P(K)} \to \RealP{K}$ defines a cofibrant replacement of $\RealP{K}$. We obtain a commutative diagram
\[
\begin{tikzcd}
 {[\RealP{{K}}, {X}]_P} \arrow[r] \arrow[d, "\cong"']& \Ho \Top_P(\RealP{{K}}, {X}) \arrow[d, "\cong"] \\
  {[\RealP{{\sd_P(K)}}, {X}]_P} \arrow[r, "\cong"]    &     \Ho \Top_P(\RealP{{\sd_P(K)}}, {X}]_P)
\end{tikzcd}
\]
Since, $\RealP{\lv_P}$ is a weak equivalence, the right vertical map is a bijection. 
By \cref{lem:fibrants_are_like_simpl_fibrants}, the same holds for the left vertical map. Since $\RealP{\sd_P(K)}$ is cofibrant, and all object of $\Top_P$ are fibrant, the bottom horizontal map is also a bijection. Hence, by commutativity of the diagram, the upper horizontal is a bijection, as was to be shown. 
\end{proof}
By \cref{rec:ex_of_con} we obtain the following immediate corollary of \cref{Theo:Fibrant_Cofibrant}.
\begin{corollary}\label{Cor:Conic_embed}
   Let $\mathrm{Con}_P\subset \Top_P$ the full subcategory of conically $P$-stratified which are triangulable (stratum preserving homeomorphic to the realization of a stratified simplicial set). Denote by $\mathrm{Con}_P/{\simeq_P}$ the category obtained by identifying stratum preserving homotopic maps. Then, the inclusion  
        \[ \mathrm{Con}_P \hookrightarrow \Top_P\]
   induces a fully faithful embedding
        \[ \mathrm{Con}_P/{\simeq_P} \hookrightarrow \Ho \Top_P. \]
\end{corollary}

\begin{remark}
\label{rem:Embedding_Global}
    \cref{Theo:Fibrant_Cofibrant,Cor:Conic_embed} also hold over varying posets, i.e. in $\Strat$. Indeed, it follows from the equivalence between the homotopy categories of $\Strat$ and $\sStrat$ which is \cref{theo:Equivalence_Simplicial_Homotopy_Category}, and the fibrancy property of conically stratified spaces, which also holds in $s\Strat$ since fibrant objects of $s\Strat$ are characterized fiberwise.
\end{remark}
\begin{remark}\label{rem:simplicial_perspective}
    Note, that \cref{lem:fibrants_are_like_simpl_fibrants}, \cref{Theo:Fibrant_Cofibrant} and \cref{Cor:Conic_embed} admit a strengthening in terms of simplicial categories. More precisely, if for $X,Y\in \Top_P$ and $K,L\in\sS_P$ one replaces the set of homotopy classes of maps $[X,Y]_P$ and $[K,L]_P$, by the simplicial mapping spaces, and the sets $\Ho\Top_P(X,Y)$ and $\Ho\sS_P(K,L)$ by derived mapping spaces, the statements remain true with analogous proofs. In practice, this means that when working with spaces satisfying the hypothesis of \cref{Theo:Fibrant_Cofibrant} the classical mapping spaces has the correct homotopy type and there is no need to derive. This applies to triangulable conically stratified spaces as detailed in \cref{rem:infty_for_con}.
    \end{remark}

    \begin{remark}\label{rem:infty_for_con}
    The simplicial perspective described in \cref{rem:simplicial_perspective} can be used to strengthen \cref{Cor:Conic_embed} to a statement about infinity categories. Indeed, $\mathrm{Con}_P$, as a subcategory of $\Top_P$, inherits the structure of a simplicial category, while $\Top_P$ itself is a simplicial model category. This means that the inclusion of the full simplicial sub-category $\Con_P\hookrightarrow \Top_P$ descends to a functor
    \begin{equation*}
        \Con_P\hookrightarrow L^H_{\text{simp}}(\Top_P)
    \end{equation*}
    Where $L^H_{\text{simp}}(\Top_P)$ is the diagonal hammock localization of a simplicial model category, described in \cite[Prop. 4.8]{FunctionComplexesDwyerKan}. \cref{Cor:Conic_embed} then generalizes to the statement that the above map is a fully faithful embedding of simplicial categories. Combining this with the fact that $L^H_{\text{simp}}(\Top_P)\simeq L^H(\Top_P)$ (\cite[Prop. 4.8]{FunctionComplexesDwyerKan}), this means that the "naive" infinity category of conically $P$-stratified and triangulable stratified spaces embeds fully faithfully in that of $P$-stratified spaces. The analogous statement for $\Con$ and $\Strat$ also holds by the same argument. 
    \end{remark}
\subsection{A simplicial approximation theorem}
\label{Section:Simp_Approx_Standalone}
    Exposing $\Ex^{\infty}_P$  as a fibrant replacement functor in $\sS_P$ allows one to study the homotopy category of $\sS_P$ through actual maps in $\sS_P$ from some subdivision of the domain. Using \cref{theo:Equivalence_Simplicial_Homotopy_Category}, one can then transport those results to $\Top_P$ and its homotopy category to obtain stratified versions of the classical simplicial approximation theorems.
	\begin{proposition}
    \label{theo:sdP_Approximation}
	    	Let $ K \in \sS_P$ be finite and $ L \in \sS_P$. 
	    	Then, for any morphism
	    	$\phi\colon K \to {L}$  in the homotopy category $\Ho{\sS}_P$ 
	    	there exists an $n \in \mathbb N$ and a morphism $f\colon \sd^n_P(K)\to L$ such that the diagram
	    	\begin{center}
	    	\begin{tikzcd}
	    	&\arrow[ld, "\lv_P^n"'] \sd^n_P{{K}} \arrow[rd, "f"]& \\
	    	{K}  \arrow[rr, "\phi"]& & {L}
	    	\end{tikzcd}
	    	\end{center}
	    	commutes in $\Ho{\sS_P}$.
	\end{proposition}
	\begin{remark}\label{rem:relative_version_sdP_approximation}
	    \cref{theo:sdP_Approximation} also holds as a relative version. To be more precise this involves the following replacements:
	    Replace ${K}$ by a pair ${A} \hookrightarrow {K}$ and ${L}$ by some stratum preserving simplicial map $g: {A} \to {L}$. The role of $\sS_P$ is then taken by the under category $\sS_P^{A}$ with the induced model structure from \cite[Theorem 7.6.5]{hirschhornModel}. However, one needs to be mindful of the fact that, for the relative version, $\sd^n_P$ maps into $\sS_P^{\sd_P^n(A)}$. In particular, the final commutativity condition holds in the homotopy category of the latter. Aside from this, the proof is formally identical. 
	\end{remark}
	As an immediate corollary of this result (its relative version), and \cref{theo:Equivalence_Simplicial_Homotopy_Category} we obtain the following corollary, which we prove first. 
    \begin{proposition}
    \label{cor:cof_approx}
    Suppose we are given a commutative diagram in $\Top_P$:
    \begin{equation*}
	    	    \begin{tikzcd}
	    	    \RealP{A}
	    	    \arrow{r}{\RealP{g}}
	    	    \arrow[hookrightarrow,swap]{d}{\RealP{i}}
	    	    &\RealP{L}
	    	    \\
	    	    \RealP{K}
	    	    \arrow[swap]{ur}{\phi} \spacecomma
	    	    \end{tikzcd}
	    	\end{equation*}
	    	where $A,K$ and $L$ are stratified simplicial sets, with $A$ and $K$ finite, $i\colon A\hookrightarrow K$ is an inclusion, $g\colon A\to L$ some arbitrary map in $\sS_P$ and $\phi\colon \RealP{K}\to\RealP{L}$ some map in $\Top_P$.
Then, there exists an $n \in \mathbb N$ and a map $\hat f\colon \sd^n_P(K) \to  L\in\sS_P$ such that:
	     \begin{itemize}
	         \item 	     the following diagram commutes in $\sS_P$\begin{equation*}
	         \begin{tikzcd}
	         \sd_P^n(A)
	         \arrow[swap]{d}{\lv^n_P}
	         \arrow{dr}{\hat{f}_{|\sd_P^n(A)}}
	         \\
	          A
	         \arrow[swap]{r}{g}
	         &L
	         \end{tikzcd}
	     \end{equation*}
	     \item the following diagram commutes in $ \Ho\Top_P^{\RealP{\sd^n_P(A)}}$
	     \begin{equation}\label{Eq:DiagramSimplicialApprox}
	         \begin{tikzcd}
	         \RealP{\sd_P^n(K)}
	         \arrow[swap]{d}{\RealP{\lv^n_P}}
	         \arrow{dr}{\RealP{\hat{f}}}
	         \\
	         \RealP{K}
	         \arrow[swap]{r}{\phi}
	         &\RealP{L}
	         \end{tikzcd}
	     \end{equation}
	     \end{itemize}

	    In particular, for $n \geq 1$, Diagram (\ref{Eq:DiagramSimplicialApprox}) can even be assumed to commute up to stratified homotopy relative to $\RealP{\sd_P^n(A)}$, since then $\RealP{\sd_P^n(i)}$ is a cofibrant object in $\Top^{\RealP{\sd_P^n(A)}}_P$.

	\end{proposition}
	\begin{proof}
	    Note that, by subdividing the source once, we may without loss of generality assume that $i$ is such that its realization is cofibrant in $\Top_P$. As every object is fibrant, this means that morphisms in the homotopy category $\Ho\Top_P^{|A|}$ from $\RealP{i}$ to $\RealP{g}$ agree with homotopy classes of maps in $\Top_P^{\RealP{A}}$ (i.e. homotopy classes rel $\RealP{A}$).
	   Using this and (a relative version of) \cref{theo:Equivalence_Simplicial_Homotopy_Category} we obtain
	    \[ 
	    [\RealP{i}, \RealP{g}]^{\RealP{A}}_P = \Ho{\Top_P^{\RealP{A}}}(\RealP{i}, \RealP{g}) \cong \Ho{\sS_P}^{A}(i,g).
	    \]
	    where the left hand side denotes relative stratum preserving homotopy classes and the right hand side bijection is given by realization.
	    Now, apply \cref{theo:sdP_Approximation}, to obtain the result.
	\end{proof}

    We now move on to the proof of \cref{theo:sdP_Approximation}. We are going to prove the nonrelative version. The relative proof is structurally almost identical. First, we need two equations involving subdivision and $\Ex_P$ which are easily verified. Recall, that we denote $\iota^n: {L} \hookrightarrow {Ex}^n_P(L)$ the natural inclusion induces by pulling back a simplex along $\lv_P^n$. 

    \begin{lemma}\label{lem:Rel_Ex_Sd}
    Denote by $\eta$ and $\varepsilon$ the unit and counit of $\sd^n_P \dashv \Ex^n_P$ respectively. Then, the equations
    \begin{align*}
        \varepsilon \circ \sd_P^n(\iota^n) &= \lv_P^n \\
        \Ex^n_P(\lv_P^n) \circ \eta &= \iota^n
    \end{align*}
    hold.
    \end{lemma}
    As an immediate corollary we obtain:
    \begin{corollary}\label{lem:tricky_Ex_diag}
   Let ${K}, {L} \in \sS_P$. Then, the following diagram of bijections commutes:
    \begin{center}
    \begin{tikzcd}[column sep= tiny]
    &\Ho{\sS_P}({K}, {L}) \arrow[rd] \arrow[ld]& \\
     \Ho{\sS_P}({K}, \Ex^n_P(L)) \arrow[rr] && \Ho{\sS_P}(\sd_P^n(K), {L})
    \end{tikzcd},
    \end{center}
     where the right diagonal is given by pre-composing with $\lv^n$, the left diagonal by post-composing with $\iota^n$ and the bottom horizontal is the adjunction map.

        \end{corollary}
    \begin{proof}
        Let $\phi \in \Ho\sS_P({K}, {L})$
        Let $\hat \phi$ be the adjoint morphism to $\iota^n \circ \phi$. It is given by $\varepsilon \circ \sd^n(\iota^n \circ \phi)$. By \cref{lem:Rel_Ex_Sd} we have 
        \begin{align*}
            \varepsilon \circ \sd^n(\iota^n \circ \phi)   &= \lv_P^n \circ \sd_P^n( \phi) \\
                                                    &= \phi \circ \lv_P^n.
        \end{align*}
        In particular, this shows the required commutativity.
        \end{proof}
    We now have everything necessary available to derive \cref{theo:sdP_Approximation}.
    \begin{proof}[Proof of \cref{theo:sdP_Approximation}]\label{proof:approximation_A}
      By \cref{cor:Exi_P_fibrant_replacement}, $\iota^\infty: {L} \hookrightarrow \Ex_P^\infty {L}$ defines a fibrant replacement of ${L}$ Now, consider the following commutative diagram:
    \begin{center}
        \begin{tikzcd}
        \sS_P( {K}, {L}) \arrow[d] \arrow[r]& \Ho{\sS_P}( {K}, {L}) \arrow[d]\\
        \sS_P({K},\Ex^\infty_P(L)) \arrow[r]& \Ho{\sS_P({K}, \Ex^\infty_P(L))}
        \end{tikzcd}
    \end{center}
    with the verticals given by postcomposition with $\iota^{\infty}$. 
    Since $\Ex_P^\infty$ is fibrant, the lower horizontal is surjective.  
    As ${K}$ is finite, $\sS_P({K}, \Ex_{P}^\infty (L)) = \varinjlim \sS_P(K, \Ex_P^n(L))$. In particular, for any $\phi\in\Ho\sS_P(K,L)$,  we find some $f'\colon {K} \to \Ex_P^n(L)$ mapping to the same element as $\phi$ in $\Ho{\sS_P({K},\Ex_P^\infty {L})}$. As $\iota^{\infty}$ is a weak equivalence and thus postcomposing with it gives a bijection in the homotopy category, we get that in particular $f' = \iota^n \circ \phi$ in $\Ho{\sS}_P(K,\Ex_P^n(L))$. 
    Next, consider the following diagram, which is commutative by \cref{lem:tricky_Ex_diag}:
    \begin{center}
        \begin{tikzcd}
        \sS_P(  K,{L}) 
        \arrow[rr] 
        \arrow[d] 
        &&  \Ho{\sS_P}({K},{L}) 
        \arrow[ld, "\cong"] 
        \arrow[dd,"\cong"] 
        \\
        \sS_P(\sd_P^n(K), {L})
        \arrow[r]
        \arrow[d, "\cong"] 
        &  \Ho{\sS}_P(\sd_P^n(K), {L}) 
        \arrow[rd, "\cong"]
        \\
        \sS_P({K}, \Ex^n_P(L)) 
        \arrow[rr] 
        &&\Ho{\sS_P}({K}, \Ex^n_{P}(L))
        \end{tikzcd}
    \end{center}
    Bijections are marked by $\cong$.
    We have shown that $f'\in\sS_P(K,\Ex_P^n(L))$ and $\phi\in\Ho\sS_P(K,L)$ have the same image in $\Ho\sS_P(K,\Ex_P^n(L))$. But the commutativity of the diagram gives that the map $f\colon \sd^n_P(K)\to L\in\sS_P$, adjoint to $f'$, and $\phi\in\Ho\sS_P(K,L)$ must have the same image in $\Ho\sS_P(\sd_P^n(K),L)$, which concludes the proof. 
    \end{proof}

\section{Simplicial Homotopy links and vertically stratified complexes}
\label{Section:LastHolink}
Consider again the following diagram of categories with weak equivalences:
\begin{center}
\begin{tikzcd}
& \Diag_P \arrow[ld] \arrow[rd]& \\
\sS_P \arrow[ru, shift left = 2, "D_P"] \arrow[rr, "\RealP{-}", shift left = 1] & & \Top_P \arrow[ll, "\Sing_P", shift left = 1] \arrow[ul, shift right = 2, "D_P^{\Top}", swap ] \spaceperiod
\end{tikzcd}
\end{center}
In \cref{Section:Real_Char_WE} we have seen that weak equivalences in $\sS_P$ are detected by (combinatorial) links and that (up to realization), those naturally have the same weak homotopy type as the homotopy links in $\Top_P$. As a particular consequence, we obtained the fact that the realization functor $\RealP{-}$ characterizes weak equivalences. The functor $D_P^{\Top}$ characterizes weak equivalences by definition and it is the content of \cref{Section:FSAE}  that $\Sing_P$ characterizes weak equivalences. We have also shown (\cref{cor:CP_reflects_WE,cor:Real_CP_reflects_WE}) that $C_P$ and thus $C_P^{\Top}=\RealP{C_P(-)}$ characterize weak-equivalences. To complete this picture, it remains to investigate the functor $D_P\colon\sS_P \to \Diag_P$. Rephrasing this question in terms of links, we need to compare the combinatorial homotopy link $\Hol$ to the other notions of (homotopy) link, which we have already shown to agree (up to the Quillen equivalence $\Real{-} \dashv \Sing$). There is an obvious natural comparison map
\[ \Hol_{\I}(K) = \Map(\Delta^{\I}, K)  \to 
\Map(\RealP{\Delta^{\I}}, \RealP{K}) = \Sing(\HolIP( \RealP{K}) \]
given by realization.  The goal of this section is to prove the following theorem.
\begin{theorem}\label{theo:simHolink_v_Top_Hol}
Let $K \in \sS_P$. Then the natural inclusion
\[\Hol_\I(K) \hookrightarrow \Sing(\HolIP(\RealP{K})) \]
is a weak equivalence of simplicial sets. Hence, equivalently, so is the adjoint map
\[ \Real{\Hol_\I(K)} \to \HolIP(\RealP{K}).\]
\end{theorem}
As an immediate corollary of this result and \cref{Cor:Realization_Preserve_Weak_Equivalences}, one obtains:

\begin{corollary}\label{Cor:Diagram_Preserve_WE}
   The functor $D_P: \sS_P \to \Diag_P$ characterizes weak equivalences. 
\end{corollary}

\begin{proof}
By \cref{Cor:Realization_Preserve_Weak_Equivalences}, a map $f\colon K\to L$ in $\sS_P$ is a weak-equivalence if and only if $\RealP{f}\colon \RealP{K}\to\RealP{L}$ is a weak-equivalence in $\Top_P$. This implies that $f$ is a weak-equivalence if and only if for all regular flags $\I$, $f$ induces weak-equivalences $\HolIP(\RealP{K})\to\HolIP(\RealP{L})$. By \cref{theo:simHolink_v_Top_Hol} this is equivalent to asking that $f$ induces weak-equivalences  $D_P(K)(\I)=\Hol_{\I}(K)\to\Hol_{\I}(L)=D_P(L)(\I)$, which concludes the proof.
\end{proof}

\begin{remark}\label{rem:Henriques}
\cref{Cor:Diagram_Preserve_WE} applies in particular to fibrant replacement maps $K\to K^{\fib}$ (for example, the map $K\hookrightarrow \Exi_P(K)$). This means that the map $\Hol_{\I}(K)\to\Hol_{\I}(K^{\fib})$ is a weak-equivalence, or in other words, that the naïve simplicial homotopy link and the fibrantly defined homotopy link coincide. This result provides insight as to why the model structure on $\sS_P$ described in \cite{Henriques} and the one under studied in this paper coincide, even though they have very different descriptions. Philosophically, the former has a class of weak-equivalences defined from the naïve homotopy links $\Hol_{\I}(K)$ while the latter has a class of weak-equivalences defined from $\Hol_{\I}(K^{\fib})$, and the weak-equivalence between those mapping spaces imply that the two model structures indeed have the same class of weak-equivalences. In fact, they coincide.
\end{remark}

\subsection{Sketch of proof of \cref{theo:simHolink_v_Top_Hol}}
We are going to prove \cref{theo:simHolink_v_Top_Hol} through a comparison of homotopy groups using simplicial approximation style results. Note that using simplicial approximation to produce homotopies between simplicial maps leads to an uncommon type of cylinder object. To handle these cylinder objects, we define sd-homotopies.  \cref{rem:Weird_Cylinders}, at the end of this subsection, provides some insight into sd-homotopies.

\begin{definition}
Let $S,S'$ be simplicial sets, $k\geq 0$, and $f,f'\colon \sd^k(S)\to S'$ be two simplicial maps. Then $f$ and $f'$ are \define{sd-homotopic}, if there exists a simplicial map
\begin{equation*}
    H\colon \sd^k(S\times \Delta^1)\to S'
\end{equation*}
such that $H$ restricts on either end of the cylinder to $f$ and $f'$. Such a map is called an \define{sd-homotopy}. If $f$ and $f'$ are pointed maps, a \define{pointed sd-homotopy} is instead a pointed map
\begin{equation*}
    H\colon \sd^k(S\wedge\Delta^1_{+})\to S',
\end{equation*}
where $\wedge$ stands for the smash product, and $\Delta^1_+$ stands for the $1$-simplex with a freely adjoined base-point.\\
Now, let $K\in \sS_P$ be a stratified simplicial set, $\I\in R(P)$ a regular flag and $f,f'\colon \sd^k(S)\times\Delta^{\I}\to K$ two stratified simplicial maps. The maps $f$ and $f'$ are called \define{(stratified) sd-homotopic} if there exists a stratum preserving simplicial map
\begin{equation*}
    H\colon \sd^k(S\times \Delta^1)\times\Delta^{\I}\to K,
\end{equation*}
 called a \define{(stratified) sd-homotopy in $\sS_P$}, such that $H$ restricts on either end of the cylinder to $f$ and $f'$. If in addition $*\in S$ and $\Delta^{\I}\to K$ are pointed, and $f$ and $f'$ are pointed maps, then a \define{stratified, pointed sd-homotopy} is instead a pointed stratified map
\begin{equation*}
    H\colon \sd^k(S\wedge\Delta^1_{+})\times\Delta^{\I}\to K.
\end{equation*}
(Stratified) pointed sd-homotopies generate equivalence relations on the sets of pointed maps $\sS^*(\sd^k(S),S')$ and $\sS_P^{\Delta^{\I}}(\sd^k(S)\times\Delta^{\I},K)$. We write $\sim_*$ for both equivalence relations defined this way.
\end{definition}

\begin{remark}
\label{rem:Homotopies_are_Quasi_Homotopies}
    Let $f,f'\colon \sd^k(S)\to S'$ be two simplicial map, related by a simplicial homotopy $H\colon \sd^k(S)\times\Delta^1\to S'$. Note that there is a simplicial map
    \begin{equation*}
        Q\colon\sd^k(S\times\Delta^1)\to\sd^k(S)\times\sd^k(\Delta^1)\xrightarrow{\Id\times\lv^k}\sd^k(S)\times\Delta^1.
    \end{equation*}
    Precomposing $H$ with $Q$ gives a sd-homtopy 
\begin{equation*}
    H\circ Q\colon \sd^k(S\times\Delta^1)\to S'
\end{equation*}
between $f$ and $f'$. In particular, if two maps are homotopic, they are also sd-homotopic. The analogous statement holds for stratified homotopies as well as in the pointed case.
\end{remark}

\begin{lemma}\label{lem:compute_strat_ho_gr}
Let $\Delta^{\I}\to K \in \sS_P^{\Delta^{\I}}$ be an $\I$ pointed stratified simplicial set. Consider $\Hol_{\I}(K)$ with the induced pointing.
Then, for each $n\geq 0$ there are natural bijections
\begin{align*}
\varinjlim \sS_P^{\Delta^{\I}}(\sd^k(S^n) \times \Delta^{\I}, K)/{\sim_*} &\cong \varinjlim\sS^*(\sd^k(S^n), \Hol_{\I}(K))/{\sim_*} \\ & \cong \pi_n(\Hol_{\I}(K))
\end{align*}
where $S^n$ is some triangulation of the $n$-sphere.
The first bijection is induced by the adjunction $- \times \Delta^{\I} \dashv \Map(\Delta^{\I}, -)$, and the second by composing with an inverse to $\lv^k$  (in the homotopy category).
\end{lemma}
\begin{proof}
The first map is a bijection since, by construction, the simplicial adjunction 
$-\times \Delta^{\I} \dashv \Map(\Delta^{\I}, -)$ preserve sd-homotopies. The second follows from the adjunction $\sd\dashv\Ex$, the fact that $\Exi$ defines a fibrant replacement in $\sS$, and that $S^n$ is a finite simplicial set (see \cref{rem:Weird_Cylinders} for more details). Note, that the commutativity conditions needed to ensure that the map from the colimit is well defined were already checked in the proof of \cref{theo:sdP_Approximation}, specifically, in \cref{lem:tricky_Ex_diag}.
\end{proof}
We can now conclude the proof of \cref{theo:simHolink_v_Top_Hol}.
\begin{proof}[proof of \cref{theo:simHolink_v_Top_Hol}]
Let $K$ be a stratified simplicial set, together with some pointing $\phi\colon  \Delta^{\I} \to K$ corresponding to a basepoint in $\Hol_{\I}({K})$. Consider the commutative diagram 
\[
    \begin{tikzcd}
    \varinjlim \sS_P^{\Delta^{\I}}(\sd^k(S^n) \times \Delta^{\I}, K )/{\sim_*}
    \arrow[r] 
    \arrow[d, "\cong"]
    & \left[\RealP{S^n \times \Delta^{\I}}, \RealP{K} \right ]_P^{\RealP{\Delta^{\I}}}
    \arrow[d, "\cong"]
    \\
    \varinjlim \sS^*(\sd^k(S^n), \Hol_{\I}(K))/{\sim_*} 
    \arrow[d, "\cong" ]
    & \left[ \Real{S^n}, \HolIP(\RealP{K}) \right ]^{*} 
    \arrow[d, "\cong"]
    \\
     \pi_n(\Hol_{\I}(K)) 
     \arrow{r}
    &  \pi_n(\HolIP(\RealP{K}))\spacecomma
    \end{tikzcd}
\]
where the bottom horizontal map is induced by $\Hol_{\I}({K}) \to \Sing(\HolIP(\RealP{K}))$ and the top horizontal
is given by first realizing and then precomposing with a homotopy inverse to $\Real{\lv^k}\times 1_{\Delta^{\I}}$. Note that the top horizontal map is well-defined since (stratified) sd-homotopies realize to (stratified) homotopies, by the fact that $\Real{\sd^k(S^n\times \Delta^1) \times \Delta^{\I}}$ is a cylinder for $\Real{\sd^k(S^n) \times \Delta^{\I}}$. 

All the vertical maps are already known to be bijections. To finish the proof of \cref{theo:simHolink_v_Top_Hol}, it suffices to show that the top horizontal map is a bijection. This is a consequence of \cref{prop:another_approximation}. Indeed, the direct statement of \cref{prop:another_approximation} gives surjectivity, while the homotopy statement shows that if two pointed maps $f,f'\colon \sd^k(S^n)\times\Delta^{\I}\to K$ realize to homotopic maps, then they are related by a pointed sd-homotopy $\sd^{k'}(sd^k(S^n)\wedge\Delta^1_+)\times\Delta^{\I}\to K$. By \cref{rem:Homotopies_are_Quasi_Homotopies}, this also means that $f\circ(\lv^{k'}\times\Id_{\Delta^{\I}})\sim_*f'\circ(\lv^{k'}\times\Id_{\Delta^{\I}})$, which proves injectivity.
\end{proof}

\begin{proposition}\label{prop:another_approximation}
Let $S \in \sS^\star$ be a pointed finite simplicial set, $\I$ a regular flag and $K \in \sS_P^{\Delta^\I}$ a pointed stratified simplicial set.  Then, for any pointed stratum preserving map $\phi\colon\RealP{S \times \Delta^\I} \to \RealP{K}$ and $k \gg 0$, there exists a pointed stratified simplicial map $f\colon\sd^k(S) \times \Delta^\I \to  K$ such that \[
    \phi \circ \RealP{\lv^k \times 1_{\Delta^\I}} \simeq_P \RealP{ f} \text{ rel }*\times\RealP{\Delta^{\I}}.
\]
Conversely, if any two pointed stratified simplicial maps $f_0,f_1: S \times \Delta^\I \to {K}$ fulfill 
\[\RealP{f_0} \simeq_P \RealP{f_1}\text{ rel }*\times\RealP{\Delta^{\I}},
\]
then, for $k \gg 0$, there exists a pointed sd-homotopy $H\colon\sd^k(S \wedge \Delta^1_+) \times \Delta^{\mathcal I} \to  K$ between $f_0\circ(\lv^k\times\Id_{\Delta^{\I}})$ and $f_1\circ(\lv^k\times\Id_{\Delta^{\I}})$.
\end{proposition}
Note, that in comparison to \cref{cor:cof_approx}, the left hand side only subdivides in the nonstratified part of $S^n \times \Delta^{\I}$. That such a more efficient subdivision suffices, is a consequence of the particularly simple shape of this stratified simplicial set. It is a special example of a vertically stratified object.
We will study those in details in \cref{Section:Vertical}.
This will serve a dual purpose. First off, we use these objects to obtain a proof of \cref{prop:another_approximation} in \cref{Section:Last_Approximation}. Secondly, they give a simple and convenient model for the homotopy category of stratified spaces (see \cref{Cor:CW_Model}).

\begin{remark}
\label{rem:Weird_Cylinders}
Simplicial approximation theorems, such as \cref{cor:cof_approx}, allow one to produce a simplical maps $f\colon \sd^k(S)\to S'$ from the data of a continuous map $\phi\colon \Real{S}\to\Real{S'}$. The two maps will then be related by a (topological) homotopy
\begin{equation*}
    H\colon\Real{\sd^k(S)}\times [0,1]\to \Real{S'}
\end{equation*}
relating $\Real{f}$ and $\phi\circ\Real{\lv^k}$.
In its relative version, one can in addition assume that if $\phi$ was already simplicial on some subobject $A\subset S$, i.e. $\phi_{|\Real{A}}=\Real{g}$, with $g\colon A\to S'$, then $f$ can be chosen such that $f_{\sd^k(A)}=g\circ\lv^k$, and the homotopy $H$ can then be taken relative to $\Real{\sd^k(A)}$. \\
One common use of the relative statement is when one has a pair of simplicial maps $f,f'\colon S\to S'$ whose realization happen to be homotopic, through some map $H'\colon \Real{S\times\Delta^1}\cong\Real{S}\times [0,1]\to \Real{S'}$. Through relative simplicial approximations, one gets a map
\begin{equation*}
    H\colon \sd^k(S\times\Delta^1)\to S',
\end{equation*}
which restricts to $f\circ\lv^k$ and $f'\circ \lv^k$ on either side of the cylinder $\sd^k(S\times\Delta^1)$. Now, note that $H$ is not an elementary simplicial homotopy (see \cref{def:Strat_Homotopies_simplicial}), but instead an sd-homtopy. In particular, unless one is given an appropriate simplicial map $\sd^{k'}(S)\times\sd^{k''}(\Delta^1)\to\sd^k(S\times\Delta^1)$, one can not deduce that, after enough subdivisions, the maps $f\circ \lv^k$ and $f'\circ\lv^k$ become simplicially homotopic in the usual sense. It does not seem unreasonable to assume that such a map can generally be exposed, in fact in the unordered simplicial complex case, such a result follows from the material presented in \cite[Chapter 3, Section 4]{spanier1989algebraic}.
We have no need for such a result in the following proof, however, and will work with the alternative cylinder instead.
Of course, $f$ and $f'$ are still equal in the homotopy category, and if $S'$ is fibrant, this is enough to show that they must be related through a simplicial homotopy. 

One way to derive these type of simplicial approximation results, which we already employed in the proof of \cref{theo:sdP_Approximation}, is by leveraging the $\Exi$ functor. Let us again illustrate this for homotopies. Let $f,f'\colon S\to S'$ be two maps who become equal in the homotopy category (this is equivalent to asking that they realize to homotopic maps). $\iota\colon S'\to \Exi(S')$ be the usual inclusion. Then, since $\Exi$ is a fibrant replacement functor, $f\circ\iota$ and $f'\circ \iota$ must be related through a simplicial homotopy $\hat H\colon S\times\Delta^1\to \Exi(S')$. Now, if $S$ is finite, the map $\hat H$ must factor through $\Ex^n(S')$ for some $n\geq 0$. Using the adjunction $\sd^n\dashv\Ex^n$, we get an sd-homtopy
\begin{equation*}
    H\colon \sd^n(S\times \Delta^1)\to S'
\end{equation*}
which restricts on either side of the cylinder $\sd^n(S\times \Delta^1)$ to $f\circ\lv^n$ and $f'\circ\lv^n$. Again, note that, a priori, this is not enough to conclude that after enough subdivisions, $f\circ \lv^n$ and $f'\circ\lv^n$ become simplicially homotopic in the usual sense. 
\end{remark}

\subsection{Vertically stratified objects}
\label{Section:Vertical}

Throughout this subsection, if $\sigma\colon \Delta^n\to S$ is a possibly degenerate simplex, $\widehat{\sigma}$ stands for the unique non-degenerate simplex of $S$, of which $\sigma$ is a degeneracy.
\begin{definition}
A \define{$P$-labelled simplicial set} is the data of
\begin{itemize}
    \item a simplicial set $S$;
    \item a labelling map, $\lambda_S\colon S_{\nd}\to N(P)_{\nd}$,
\end{itemize}
such that, for any $\sigma\subset\tau$ in $S$, $\lambda_S(\tau)\subset\lambda_S(\sigma)$. A \define{label preserving map} $f\colon (S,\lambda_S)\to (S',\lambda_{S'})$ is a simplicial map $f\colon S\to S'$, such that for all $\sigma\in S_{\nd}$, $\lambda_S(\sigma)\subset \lambda_{S'}(\nondegen{f(\sigma)})$. 
We denote by $\PsS$ the category of $P$-labelled simplicial sets with label preserving maps. 
\end{definition}

\begin{example}\label{ex:Naive_Subdivision_is_P_labelled}
Let $K\in \sS_P$ be a stratified simplicial set. Consider the (non-stratified) simplicial set $\sd(K)$. Define a labelling map, $\lambda\colon \sd(K)_{\nd}\to N(P)_{\nd}$ as follows. If $(\mu,\sigma)$ is a vertex in $\sd(K)$, with  $\mu \subset \Delta^n$ and $\sigma\colon \Delta^n\to K$, let $\lambda(\mu,\sigma)=\nondegen{\varphi_K\circ\sigma\circ\mu}$. For higher dimensional simplices, set $\lambda((\mu_0,\dots,\mu_k),\sigma)=\lambda(\mu_0,\sigma)$. Then $(\sd(K),\lambda)$ is a $P$-labelled simplicial set, and for any stratum preserving simplicial map $f\colon K\to L$, $\sd(f)\colon (\sd(K),\lambda)\to(\sd(L),\lambda)$ is a label-preserving map.
\end{example}
\begin{definition}
Let $\lab{S}$ be a $P$-labelled simplicial set. Its \define{verticalization} is the inclusion of stratified simplicial sets $V(S,\lambda_S)\hookrightarrow S\times N(P)$, where $V(S,\lambda_S)$ is defined as the following subset:
\begin{equation*}
    V\lab{S}=\bigcup\limits_{\sigma\colon \Delta^n\to S \textnormal{ , n.d.} }\Im(\sigma)\times\lambda_S(\sigma)\subset S\times N(P) \spacecomma
\end{equation*}
where the union is taken over all non-degenerate simplices of $S$. 
If $f \colon \lab{S} \to \lab{S'}$ is a label preserving map, then the map $f \times 1_{N(P)}$ restricts to a stratified map $V(f)\colon V\lab{S} \to V\lab{S'}$ and $V$ defines a functor $\PsS \to \sS_P$ in this fashion.
\end{definition}

Next, let us pay some more attention to the stratified simplicials sets which lie in the essential image of the verticalization functor.
\begin{definition}\label{def:pre-vert-ss}
A \define{pre-verticalization} on a stratified simplicial set $K$ is the data of 
\begin{itemize}
    \item a simplicial set $\bar{K}$;
    \item a monomorphism in $\sS_P$, $K\hookrightarrow \bar{K}\times N(P)$.
\end{itemize}
A \define{vertical map} between stratified simplicial sets equipped with pre-verticalizations $K\subset \bar{K}\times N(P)$ and $L\subset\bar{L}\times N(P)$ is a stratum preserving map $f\colon K\to L$  such that there exists a simplicial map $\bar{f}\colon \bar{K}\to\bar{L}$ making the following diagram commute:
\begin{equation*}
    \begin{tikzcd}[column sep=30 pt]
    K
    \arrow{r}{f}
    \arrow[hookrightarrow]{d}
    &L
    \arrow[hookrightarrow]{d}
    \\
    \bar{K}\times N(P)
    \arrow{r}{\bar{f}\times \Id_{N(P)}}
    &\bar{L}\times N(P) \spaceperiod
    \end{tikzcd}
\end{equation*}
\end{definition}
\begin{definition}\label{def:VerticalSimplicialSets}
A \define{vertically stratified simplicial set} is a stratified simplicial set, $K$, equipped with a pre-verticalization, $K\hookrightarrow \bar{K}\times N(P)$, which is isomorphic to the verticalization of a $P$-labelled simplicial set through some vertical map:
\begin{equation*}
    \begin{tikzcd}
    K
    \arrow[hookrightarrow]{d}
    \arrow{r}{\simeq}
    &V(S,\lambda_S)
    \arrow[hookrightarrow]{d}
    \\
    \bar{K}\times N(P)
    \arrow{r}{\simeq}
    &S\times N(P)\spaceperiod
    \end{tikzcd}
\end{equation*}
\end{definition}

\begin{remark}
Note that for a stratified simplicial set, having a pre-verticalization is not enough to be a vertically stratified simplicial set. All stratified simplicial sets admit a tautological pre-verticalization, $K\to K\times N(P)$. However, we are only interested in those pre-verticalizations that come from taking the verticalizations of $P$-labelled simplicial sets. Furthermore, when considering vertical maps between those, one need not keep track of the associated simplicial map $\bar{f}\colon \bar{K}\to \bar{L}$ since in those cases, if such a map exists, it is unique. To see this, observe that for a vertically stratified simplicial set $K\to \bar{K}\times N(P)$, the composition $K\to\bar{K}\times N(P)\to\bar{K}$ must be surjective, leaving at most one choice for $\bar{f}$ (see also the fully faithfulness part of \cref{prop:relation_lab_diag_strat}).
\end{remark}

\begin{example}\label{Ex:Sd_P_Vertically_Filtered_sS}
Let $K\in \sS_P$ be a stratified simplicial set, and $(\sd(K),\lambda)$ the $P$-labelled simplicial set from \cref{ex:Naive_Subdivision_is_P_labelled}. Then $V(\sd(K),\lambda)=\sd_P(K)\subset \sd(K)\times N(P)$. Furthermore, if $f\colon K\to L$ is a stratified map, then $\sd_P(f)=V(\sd(f))$ is a vertical map. See \cref{fig:Subdivision_As_Vertical_Stuff}.
\end{example}
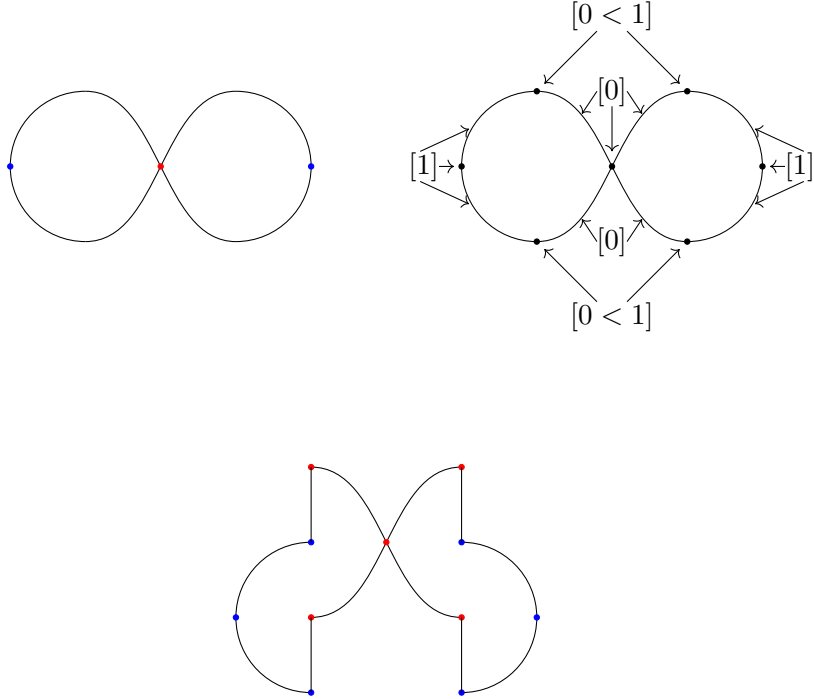
\begin{figure}
\begin{tikzpicture}
\draw[black] (0,0) arc (90 :270 : 1cm);
\draw[black] (0,0) ..controls (1,0) and (1,-2)..(2,-2);
\draw[black] (0,-2) ..controls (1,-2) and (1,0)..(2,0);
\draw[black] (2,-2) arc (-90 :90 : 1cm);
\filldraw[red](1,-1) circle (1pt);
\filldraw[blue](-1,-1) circle (1pt);
\filldraw[blue](3,-1) circle (1pt);

\draw[shift={(6,0)},black] (0,0) arc (90 :270 : 1cm);
\draw[shift={(6,0)},black] (0,0) ..controls (1,0) and (1,-2)..(2,-2);
\draw[shift={(6,0)},black] (0,-2) ..controls (1,-2) and (1,0)..(2,0);
\draw[shift={(6,0)},black] (2,-2) arc (-90 :90 : 1cm);
\filldraw[shift={(6,0)},black](1,-1) circle (1pt);
\filldraw[shift={(6,0)},black](-1,-1) circle (1pt);
\filldraw[shift={(6,0)},black](3,-1) circle (1pt);
\filldraw[shift={(6,0)},black] (0,0) circle (1pt) (2,0) circle (1pt) (0,-2) circle (1pt) (2,-2) circle (1pt);
\node[shift={(6,0)}] at (-1.5,-1) {$[1]$};
\draw[shift={(6,0)},->] (-1.55,-1.2)--( -0.9,-1.5);
\draw[shift={(6,0)},->] (-1.55,-0.8)--( -0.9,-0.5);
\draw[shift={(6,0)},->] (-1.3,-1)--( -1.1,-1);

\node[shift={(6,0)}] at (3.5,-1) {$[1]$};
\draw[shift={(6,0)},->] (3.55,-1.2)--( 2.9,-1.5);
\draw[shift={(6,0)},->] (3.55,-0.8)--( 2.9,-0.5);
\draw[shift={(6,0)},->] (3.3,-1)--( 3.1,-1);

\node[shift = {(6,0)}] at (1, 0) {$[0]$};
\draw[shift = {(6,0)},->] (0.8,0)--(0.6,-0.3);
\draw[shift = {(6,0)},->] (1.2,0)--(1.4,-0.3);
\draw[shift = {(6,0)},->] (1,-0.2)--(1,-0.8);

\node[shift = {(6,0)}] at (1, -2) {$[0]$};
\draw[shift = {(6,0)},->] (0.8,-2)--(0.6,-1.7);
\draw[shift = {(6,0)},->] (1.2,-2)--(1.4,-1.7);

\node[shift = {(6,0)}] at (1, 1) {$[0<1]$};
\draw[shift = {(6,0)},->] (0.8,0.8)--(0.1,0.1);
\draw[shift = {(6,0)},->] (1.2,0.8)--(1.9,0.1);

\node[shift = {(6,0)}] at (1, -3) {$[0<1]$};
\draw[shift = {(6,0)},->] (0.8,-2.8)--(0.1,-2.1);
\draw[shift = {(6,0)},->] (1.2,-2.8)--(1.9,-2.1);

\draw[shift={(3,-6)},black] (0,0) arc (90 :270 : 1cm);
\draw[shift={(3,-5)},black] (0,0) ..controls (1,0) and (1,-2)..(2,-2);
\draw[shift={(3,-5)},black] (0,-2) ..controls (1,-2) and (1,0)..(2,0);
\draw[shift={(3,-6)},black] (2,-2) arc (-90 :90 : 1cm);
\filldraw[shift={(3,-5)},red](1,-1) circle (1pt);
\filldraw[shift={(3,-6)},blue](-1,-1) circle (1pt);
\filldraw[shift={(3,-6)},blue](3,-1) circle (1pt);
\filldraw[shift={(3,-6)},blue] (0,0) circle (1pt) (2,0) circle (1pt) (0,-2) circle (1pt) (2,-2) circle (1pt);
\filldraw[shift={(3,-5)},red](0,0) circle (1pt) (2,0) circle (1pt) (0,-2) circle (1pt) (2,-2) circle (1pt);
\draw[shift={(3,-6)}] (0,0)--(0,1) (0,-2)--(0,-1) (2,0)--(2,1) (2,-2)--(2,-1);
\end{tikzpicture}
\caption{The figure $8$ as a stratified simplicial set over $P=\{0<1\}$, its subdivision as a $P$-labelled simplicial set, and the associated vertically stratified simplicial set, $\sd_P(K)$.}
\label{fig:Subdivision_As_Vertical_Stuff}
\end{figure}

\begin{remark}\label{rem:Vertical_Are_Diagrams}
    Vertically stratified simplicial sets can be characterized in another way. They are precisely the images of cofibrant objects in $\Diag_P$ under the functor $C_P\colon \Diag_P\to \sS_P$. Indeed, one can go from a $P$-labelled simplicial set, $\lab{S}$, to a (cofibrant) diagram, $F$, by setting $F(\I)=\{\sigma\in S\ |\ \Delta^{\I}\subset \lambda_S(\nondegen{\sigma})\}$, for regular flags $\I$,. Conversely, given a cofibrant diagram $F$, set $S=\cup_{\I}F(\I)$, and $\lambda_{S}(\sigma)=\max\{\Delta^{\I}\ |\ \sigma\in F(\I)\}$. To see why this labelling map is well-defined, recall from \cref{prop:Cofibrant_Diagrams} that cofibrant objects in $\Diag_P$ are precisely the diagrams, $F$, such that:
    \begin{itemize}
        \item For all $\I'\subset \I$, the map $F(\I)\to F(\I')$ is a monomorphism.
        \item If $\I_0\subset \I_1,\I_1'$, then $F(\I_0)\supset F(\I_1)\cap F(\I_1')\not=\emptyset$ if and only if there exist $\I_2$ such that $\I_1,\I_1'\subset \I_2$, and in this case, $F(\I_1)\cap F(\I_1')=F(\I_2)$, for the smallest such $\I_2$. 
    \end{itemize}
     One then checks that the verticalization process corresponds to applying $C_P$ to the corresponding diagram.

     This relation between vertically stratified simplicial sets and diagrams fits into a larger picture, which is explored in more details in \cref{section:appendixB}, more specifically \cref{prop:properties_of_U_functors}.
\end{remark}

\begin{proposition}\label{prop:simplicial_vertical_approx}
Let $f\colon S\times\Delta^{\I}\to L$ in $\sS_P$ be a stratum preserving simplicial map between vertically stratified simplicial sets. There exists a stratum preserving homotopy, $H\colon (S\times\Delta^{\I})\times\Delta^1\to L$ such that $H_0=f$ and $H_1$ is a vertical map. Furthermore, if $f$ is already vertical on $A\times\Delta^{\I}\subset S\times\Delta^{\I}$, then the homotopy $H$, can be taken relative to $A\times\Delta^{\I}$.
\end{proposition}
\begin{proof}
We give an explicit definition of the homotopy for the case $S = \Delta^n$ which is natural in $n$ and hence extends to the general case.
Note that $f$ factors through $S' \times \Delta^{\I} \subset L$, where $L=V\lab{\bar{L}}$, and $S'\subset \bar{L}$ is spanned by the simplices $\sigma$ such that $\Delta^{\I}\subset\lambda_{\bar{L}}(\sigma)$.
Hence, we may without loss of generality assume $L$ to be of the form  $S' \times \Delta^{\I}$ for some $S' \in \sS$. 
In particular, as $f$ is stratum preserving, it is uniquely determined by the map $\hat f: \Delta^n \times \Delta^{\I} \to S'$ obtained by composing with the projection to $S'$, and we may equivalently show the existence of a simplicial map (natural in $\Delta^n$ and $S'$) $\hat H\colon \Delta^n \times \Delta^{\I} \times \Delta^1 \to S'$ such that $\hat H_0 = \hat f$ and $\hat H_1$ factors through $\Delta^n \times \Delta^{\I} \to \Delta^n$. 
Any $k$-simplex in $\Delta^n \times \Delta^\I \times \Delta^1$ is given by the data of a flag 
\[(x_0, p_0) \leq \dots \leq (x_k,p_k)\] 
with $x_i \in [n]$ and $p_i \in \I$, together with some $l \in [k+1]$ indicating whether the vertices $(x_i, p_i)$ project to $0$ or $1 \in \Delta^1$. More explicitely, $l$ stands for the $k$-simplex of $\Delta^1$, $[0,\dots,0,1,\dots 1]$, where the first $1$ appears at the $l$-th position. Note that a simplex $ ([(x_0,p_0), \dots (x_{k},p_{k})],l)$ lies in the sub-object   $\Delta^n \times \Delta^\I \times \{0\}$ (resp. $\Delta^n \times \Delta^\I \times \{1\}$) if and only if  $l=k+1$ (resp. $l=0$).
Using this, consider the (not stratum preserving) simplicial map 
\begin{align*}
    R^n\colon\Delta^n  \times \Delta^\I \times \Delta^1 &\to \Delta^n \times \Delta^\I \\
     ([(x_0,p_0), \dots (x_{k},p_{k})],l) &\mapsto  [ (x_0,p_0), \dots , (x_l,p_l), (x_{l+1},p_m) , \dots , (x_{k},p_m) ], 
     \end{align*}
  where $p_m$ is the maximum of $\I$. It is immediate from the definition, that $R^n$ is natural in $n$ and furthermore, that $R^n_0 = 1$ and $R^n_1$ factors through $\Delta^n \times \Delta^{\I} \to \Delta^n$. 
Now, finally define $\hat H$ as the composition \[\hat H: \Delta^n  \times \Delta^\I \times \Delta^1 \xrightarrow{R^n} \Delta^n \times \Delta^\I \xrightarrow{\hat f} S',\] which fufills the requirements by the respective properties of $R^n$, and is natural in n. Finally, for the relative statement, note that if $\hat f$ already factors through $\Delta^n  \times \Delta^\I \to \Delta^n$, i.e. $f$ is vertical on $\Delta^n \times \Delta^\I$, then we have a commutative diagram 
\[
\begin{tikzcd}
\Delta^n \times \Delta^\I \times \Delta^1 \arrow[r, "R^n"] \arrow[d, "\pi_{\Delta^n \times \Delta^\I}"] & \Delta^n \times \Delta^\I \arrow[d] \arrow[rd, "\hat f"] & \\
\Delta^n \times \Delta^\I  \arrow[r]& \Delta^n \arrow[r, dashed]& S'
\end{tikzcd}
\]
In other words, $\hat H$ factors through $\pi_{\Delta^n \times \Delta^\I}: \Delta^n \times \Delta^\I \times \Delta^1 \to \Delta^n \times \Delta^\I$ and thus is a constant homotopy, which shows the relative statement.
\end{proof}

The theory of $P$-labelled and of vertically stratified simplicial sets has a topological counterpart, which we briefly introduce here.

\begin{definition}
A \define{$P$-labelled CW-complex} is the data of 
\begin{itemize}
    \item a CW-complex, $T$;
    \item a labelling map, $\lambda_T\colon \{\text{cells of $T$}\}\to N(P)_{\nd}$.
\end{itemize}
Such that, for any pair of cells $e_{\alpha},e_{\beta}$ such that $e_{\alpha}\cap \overline{e_{\beta}}\not = \emptyset$, $\lambda_T(e_{\beta})\subset \lambda_T(e_{\alpha})$.
A \define{label preserving map} $f\colon \lab{T}\to\lab{{T'}}$ is a continuous map $f\colon T\to {T'}$, such that for any cells $e_{\alpha}\in T, e_{\beta}\in {T'}$ such that $f(e_{\alpha})\cap e_{\beta}\not = \emptyset$, one has $\lambda_T(e_{\alpha})\subset \lambda_{T'}(e_{\beta})$. We denote the category of $P$-labeled CW-complexes by $\PCW$.
\end{definition}
\begin{example}\label{Ex:RealizingPlabeledGivesPlabeled}
Given a $P$-labelled simplicial set $\lab{T}$, its realization admits the structure of a $P$-labelled CW-complex with a cell for each non-degenerate simplex, which is given the same label as the simplex.
\end{example}

As for $P$-labelled simplicial sets, we are interested in the verticalization of a $P$-labelled CW-complex.
\begin{definition}
Let $(T,\lambda_T)$ be a $P$-labelled CW-complex. Its \define{verticalization} is the inclusion $V(T,\lambda_T)\hookrightarrow T\times \RealP{N(P)}$, where $V(T,\lambda_T)$ is the following subset:
\begin{equation*}
    V\lab{T}=\bigcup\limits_{e_{\alpha}\in \{\text{cells of $T$}\}}e_{\alpha}\times\RealP{\lambda_T(e_{\alpha})}\subset T\times\RealP{N(P)}. 
\end{equation*}
Extending to morphisms in the obvious way, this construction defines a functor $V\colon \PCW \to \Top_P$, which factors through $V\colon \PCW\to \Top_{N(P)}$.
\end{definition}

Again, it can be useful to study the objects in the essential image of the verticalization functor. These lead to a particularly convenient class of stratified spaces which admit a cell decomposition similar to classical CW-complexes (see \cref{Rem:Vertical_CW_Cell_Complexes}). We give a brief outlook into the resulting theory here, without going too much into detail. The definitions require a definition of a \define{pre-verticalization} and \define{vertical map}, which we omit here, since they are entirely analogous to the simplicial definition in \cref{def:pre-vert-ss}, replacing $N(P)$ by $\RealP{N(P)}$.

\begin{definition}\label{def:vert_CW}
A \define{vertically stratified CW-space} is a stratified space equipped with a pre-verticalization, which is vertically isomorphic to the verticalization of a $P$-labeled CW-complex. When such an isomorphism is fixed we speak of a \define{vertically stratified CW-complex}. We will consider those objects in the categories $\Top_P$ and $\Top_{N(P)}$.
\end{definition}

\begin{remark}\label{Rem:Vertical_CW_Cell_Complexes}
    Vertically stratified CW-spaces can also be characterized explicitly through the existence of a cell decomposition. Though, in this case, one has to be careful to also impose verticality conditions on the gluing of cells. Since the cells are of shape $B^n\times \RealNP{\Delta^{\I}}$, for $n\geq 0$, and regular flags $\I$, and are glued along their boundaries $S^{n-1}\times \RealNP{\Delta^{\I}}$, this implies that vertically stratified CW-spaces are actual cell complexes in $\Top_{N(P)}$ (and in $\Top_P$). In particular, they are cofibrant objects.
\end{remark}

\begin{example}\label{ex:real_sd_P_is_vert}
    Following \cref{Ex:Sd_P_Vertically_Filtered_sS,Ex:RealizingPlabeledGivesPlabeled}, given a stratified simplicial set $K\in\sS_P$, we have a $P$-labelled CW-complex $(\Real{\sd(K)},\Real{\lambda})$. Its verticalization gives a vertically stratified CW-structure on $\RealP{\sd_P(K)}$:
    \begin{equation*}
        \RealP{\sd_P(K)}\hookrightarrow\Real{\sd(K)}\times\RealP{N(P)}.
    \end{equation*}
    See \cref{prop:relation_lab_diag_strat} for a complete picture of the relationship between labelled objects, vertical objects and diagrams indexed over $R(P)$.
\end{example}

\begin{proposition}\label{prop:vertical_approx}
    Let $f\colon X\to Y$ be a map in $\Top_{N(P)}$ between vertically stratified CW-spaces, $X \cong V \lab{\bar{X}}$ and $Y\cong V\lab{\bar{Y}}$. Then $f$ is strongly stratified homotopic to a vertical map. Furthermore, if $\bar{A} \subset \bar{X}$ is a subcomplex equipped with the induced labelling, then the homotopy can be taken relative to $A=V\lab{\bar{A}}\subset X$. The analogous statement also holds for maps and homotopies in $\Top_P$.
\end{proposition}

\begin{proof}
    Here, we give a proof using the cellular structure of vertically stratified CW-spaces. A proof going through abstract homotopy theory is given in \cref{proof:alternative_proof_in_appendix}.
    We first prove the statement for maps in $\Top_{N(P)}$ by the usual induction on skeletons argument. Let us write $X^n=V(\bar{A}\cup \bar{X}^n,\lambda_{\bar{X}})$, for $n\geq -1$. We will construct a sequence of maps, $g_i\colon X\to Y$, $i\geq -1$, with $g_{-1}=f$, and $g_n$ vertical on $X^n$, and $g_{n}$ and $g_{n-1}$ homotopic relative to $X^{n-1}$, for $n\geq 0$, through some homotopy in $\Top_{N(P)}$, $H_n$. By subdividing the interval, one can then concatenate the homotopies $H_n$, to a homotopy $H\colon X\times [0,1)\to Y $. By construction, the map $g\colon x\mapsto \lim_{t\to 1} H(x,t)$ is well defined, vertical and extends $H$ to $X \times [0,1]$. Hence, $g$ is as required in the statement of the proposition.

    Now, let $n\geq 0$, and assume $g_{n-1}$ has been constructed. Let $e_{\alpha}$ be some $n$-cell of $X$ of shape $B^n\times\Delta^{\I}$ and define the homotopy
    \begin{align*}
        H'_{n,\alpha}\colon B^n\times\RealNP{\Delta^{\I}}\times[0,1]&\to Y\subset{\bar{Y}}\times\RealNP{N(P)}\\
        ((x,\xi),s)&\mapsto (\pr_{\bar{Y}}(g_{n-1}(x,s(1,0,\dots,0)+(1-s)\xi)),\xi)
    \end{align*}
    Assembling the $H'_{n,\alpha}$ gives a homotopy on the $n$-skeleton $H'_n\colon X^n\times [0,1]\to Y$, which is constant on the $n-1$ skeleton, by the induction hypothesis. But now, by \cref{Rem:Vertical_CW_Cell_Complexes}, the inclusion $X^n\to X$ is a cofibration in $\Top_{N(P)}$, and thus has the homotopy extension property. This means that $H'_n$ extends to the desired homotopy $H_n\colon X\times [0,1]\to Y$, between $g_{n-1}$ and some map $g_n$ with the required properties, which concludes the proof for $\Top_{N(P)}$.

    For maps in $\Top_P$, first note that, as noticed in Remark \ref{Rem:Vertical_CW_Cell_Complexes}, $X$ and $Y$ are cofibrant objects, in $\Top_P$, and in fact come from cofibrant objects in $\Top_{N(P)}$. Since all objects in $\Top_P$ and $\Top_{N(P)}$ are fibrant, and since $\varphi_P\circ -\colon \Top_{N(P)}\to \Top_P$ is a Quillen equivalence between the two model categories, this implies that we have canonical bijections
\begin{align*}
    \Hom_{\Top_P}(X,Y)/ {\simeq_P}&\cong \Hom_{\Ho\Top_P}(X,Y)\\
    &\cong \Hom_{\Ho\Top_{N(P)}}(X,Y)\\
    &\cong \Hom_{\Top_{N(P)}}(X,Y)/ {\simeq_{N(P)}}.
\end{align*}
Where $\simeq_P$ and $\simeq_{N(P)}$ stand for the homotopy relations in $\Top_P$ and $\Top_{N(P)}$ respectively. In particular, this implies that there must exist $f'\colon X\to Y$ a map in $\Top_{N(P)}$ homotopic to $f$ as a map in $\Top_P$. 
For the relative statement, consider the relative categories $\Top_P^A$, and $\Top^A_{N(P)}$. The Quillen equivalence between $\Top_P$ and $\Top_{N(P)}$ passes to the relative categories, and the above argument then gives the existence of a map in $\Top_{N(P)}^A$, $f'\colon X\to Y$, homotopic to $f$ as a map of $\Top_P^A$. In other words, it is homotopic to $f$ relative to $A$.
We then conclude the proof by applying the first part of the proposition to $f'$.
\end{proof}

\begin{definition}
Two vertical maps between vertically stratified CW-spaces, $f,g\colon X\to Y$, are said to be \define{vertically homotopic} if there exists a vertical map $H\colon X\times [0,1]\to Y$ which is a homotopy between $f$ and $g$. Write $f\simeq_{V}g$ when $f$ and $g$ are vertically homotopic, and let $\mathrm{VCW}_P/{\simeq_{V}}$ be the category whose objects are vertically stratified CW-spaces and whose set of morphisms between $X$ and $Y$ is the set of vertical maps quotiented by the equivalence relation $\simeq_{V}$. 
\end{definition}
Now, as a corollary of \cref{prop:vertical_approx}, we obtain that the homotopy category $\Ho \Top_P$ may equivalently be described as the homotopy category of vertically stratified CW-spaces $\mathrm{VCW}_P /{\simeq_{V}}$, and vertical maps and homotopies. The comparison is given by the forgetful functor, sending a vertically stratified CW-space $X\hookrightarrow \bar{X}\times \RealP{N(P)}$ to the stratified space $X$.
\begin{theorem}\label{Cor:CW_Model}
the forgetful functor induces an equivalence of categories
\begin{equation*}
    \mathrm{VCW}_P/{\simeq_{V}}\cong \Ho\Top_P.
\end{equation*}
\end{theorem}

\begin{proof}
By classical results, we know that $\Ho\Top_P\cong \Top_P^{\cof}/{\simeq_P}$. Furthermore, for all $X\in \Top_P$, $X$ is weakly equivalent to $\RealP{\sd_P(\Sing_P(X))}$, which admits the structure of a vertically stratified CW-complex by \cref{ex:real_sd_P_is_vert}. Since stratified spaces admitting the structure of a vertically stratified CW-complexes are cofibrant objects, when $X$ is also a cofibrant object it is in fact homotopy equivalent to the vertically stratified CW-complex  $\RealP{\sd_P(\Sing_P(X))}$.
Thus, one can restrict from the subcategory of cofibrant objects, to the subcategory of vertically stratified CW-spaces (while still retaining all stratum preserving maps and homotopies). But then, by \cref{prop:vertical_approx}, it is enough to consider only vertical maps, and, from the relative case, we see that it is also enough to consider only vertical homotopies between vertical maps, which concludes the proof.
\end{proof}

\begin{remark}
\cref{Cor:CW_Model} implies that vertically stratified CW-complexes and vertical maps (and equivalently labeled CW-complex with label preserving maps, see \cref{section:appendixB}, specifically \cref{Cor:lab_CW_Model}) are a model for the homotopy category of stratified spaces. This is particularly convenient for a few reasons.
    \begin{itemize}
        \item As illustrated in the proof of \cref{prop:vertical_approx}, vertically stratified CW-complexes allow for inductive arguments on their skeletons, giving access to reasonably elementary and topological proofs.
        \item Vertically stratified CW complexes allow for a nice interpretation of the homotopy links, and the stratified homotopy groups. Indeed, let $X\cong V\lab{\bar{X}}\subset \bar{X}\times\RealP{N(P)}$ be a vertically stratified CW-complex, let $\I$ be a regular flag, and let $\bar{X}_{\I}\subset \bar{X}$ be the subcomplex containing all cells $e_{\alpha}\subset \bar{X}$, such that $\Delta^{\I}\subset\lambda_{\bar{X}}(e_{\alpha})$. Then, there is a homotopy equivalence
        \begin{equation*}
        \HolIP(X)\simeq \bar{X}_{\I}
        \end{equation*}
        where the map from $\HolIP(X)$ to $\bar{X}_{\I}$ is given by evaluating at the $0$-th vertex, then projecting to $\bar{X}$, while its homotopy inverse is the map sending a point $x\in\bar{X}_{\I}$ to the composition $\{x\}\times\RealP{\Delta^{\I}}\hookrightarrow \bar{X}_{\I}\times\RealP{\Delta^{\I}}\hookrightarrow X$. In particular, the $p$-stratum of $X$ is homotopy equivalent to the subobject $\bar{X}_p\subset \bar{X}$, and for all regular flags $\I=\{p_0<\dots<p_n\}$, the above equivalence can be rewritten as 
        \begin{equation*}
        \HolIP(X)\simeq\bar{X}_{p_0}\cap\dots\cap \bar{X}_{p_n}.
        \end{equation*}
        In particular, for a vertically stratified CW-complex, its strata - and homotopy links - can be interpreted as subobjects - and intersections of those subobjects - of the associated $P$-labelled CW-complex.
         Even more, a pointing $x$ of $\bar{X}$ in a cell with label $\Delta^{\I}$ gives a stratified pointing $\phi\colon \RealP{\Delta^{\I}}\to X$ (and all pointing of $X$ come from restrictions of such pointings, up to homotopy). Thus, the stratified homotopy groups of \cite{douteauEnTop} associated to the pointing $\phi$ are nothing more than the data of the homotopy groups of the $\bar{X}_{\I'}$, for $\I'\subset \I$, with respect to the pointing $x$.
        \item Given a vertically stratified CW-complex $X \cong V\lab{\bar X}$ its underlying (non-stratified) homotopy type is that of the CW-complex $\bar{X}$. Indeed, to see this, consider the diagram of spaces associated to $V\lab{X}$, $F\colon \I\mapsto \bar{X}_\I$ defined as in the previous bulletpoint (in the appendix this construction is denoted $U^{\Top}$, see \cref{prop:properties_of_U_functors} for more details).
        Then, the underlying space of $X \cong V\lab{\bar X}$ is a simplicial model for the homotopy colimit of $F$ (compare for ex \cite[Sec. 18.1]{hirschhornModel}). At the same time $\bar{X}$ is the regular colimit of $F$. Furthermore, it follows immediately from the construction of $F$, that it is a cofibrant diagram in the projective model structure, on the category of functors from $R(P)^{\op}$ to $\Top$ \cref{prop:properties_of_U_functors}. 
        In particular (see for example \cite[Prop 9.11]{Dugger2008APO}), we have weak equivalences of topological spaces:
        \[
        X = \hocolim F \xrightarrow{\sim} \colim F = \bar{X}.
        \]
        The analogous result holds for vertically stratified simplicial sets.
    \end{itemize}
\end{remark}
\subsection{One last approximation theorem}
\label{Section:Last_Approximation} We now move on to the proof of \cref{prop:another_approximation}.
We first need a technical lemma on how to concatenate sd-homotopies.
\begin{lemma}\label{lem:concat_sd_ho}
Let $S,S'$ be simplicials sets and
Let $H\colon \sd^k( S \times \Delta^1) \to S'$ as well as $H':\sd^k( S \times \Delta^1) \to S'$ be sd-homotopies between simplicial maps $f$ and $g$, and $g$ and $h$ respectively. Then, there exists a concatenated sd-homotopy $H'' \colon \sd^{k+2}( S \times \Delta^1) $ from $f \circ \lv^2$ to $h \circ \lv^2$. The analogous result for pointed and stratified sd-homotopies also holds. 
\end{lemma}
\begin{proof}
We prove the non-pointed non-stratified case. The other cases work completely analogously. 
First consider a gluing \begin{equation*}
H\cup H'\colon \sd^{k}(S \times \Delta^1) \cup_{\sd^{k}(S)}\sd^{k}(S \times \Delta^1) \to S' \spacecomma
\end{equation*}
gluing $H$ to $H'$ along $g$.
The left hand side of this map is equivalently given by 
\[
\sd^{k}(S \times ( \Delta^1 \cup_{1,0} \Delta^1))= \sd^{k}(S \times ( \Delta^1 \cup_{1,0} \Delta^1))  \spaceperiod
\]
Note, that we may naturally identify
\[
\sd^2(\Delta^1) = \Delta^1 \cup_{1,1} \Delta^1 \cup_{0,0} \Delta^1 \cup_{1,1} \cup \Delta^1.
\]
Using this identification, one obtains a map
\[
c\colon \sd^2(\Delta^1) \to ( \Delta^1 \cup_{1,0} \Delta^1)
\]
given by collapsing the second and the fourth interval of $\sd^2(\Delta^1)$ to a point. In particular, this map maps left endpoint to left endpoints and right endpoints to right endpoints. From $c$, we in turn now obtain the composition
\begin{equation*}
   c' \colon \sd^{2}(S \times  \Delta^1)\to\sd^{2}(S)\times \sd^{2}(\Delta^1) \xrightarrow{\lv^2 \times c} S \times ( \Delta^1 \cup_{1,0} \Delta^1 )  
\end{equation*}
mapping the left boundary of the subdivided cylinder $\sd^{2}(S \times \Delta^1)$ to the left boundary of the glued cylinder $S \times ( \Delta^1 \cup_{1,0} \Delta^1)$ and analogously for the right boundary. Precomposing $H' \cup H'$, with $\sd^k(c')$ produces an sd-homotopy $H''$ as required.
\end{proof}
Next, we show the following.
\begin{lemma}\label{lem:vertical_simplicial_approx}
Let $S\in \sS$ be a finite simplicial set with simplicial subset $S_0 \subset S$. Further, let $K \cong V(\bar K, \lambda_{\bar K}) \in \sS_P$ be a vertically stratified simplicial set, and $\I$ a regular flag of $P$. Assume that we are given a map in $\Top_{P}$, $\phi\colon \RealP{S\times\Delta^{\I}}\to \RealP{K}$ and a vertical map in $\sS_P$ $g\colon S_0\times\Delta^{\I}\to K$, such that 
\[
\phi|_{\RealP{{S_0} \times \Delta^\I}} = \RealP{g}.
\]

Then for $k\gg 0$ there exists a vertical, simplicial map $ f\colon \sd^k(S)\times\Delta^{\I}\to {K} $, such that 
\[
f|_{\sd^k(S_0) \times \Delta^\I} = g \circ (\lv^k \times 1_{\Delta^\I}) 
\]
and such that $\RealP{f}$ is stratum preserving homotopic to $\phi \circ\RealP{\lv^k\times\Id_{\Delta^{\I}}}$ relative to $\RealP{\sd^k({S_0})\times\Delta^{\I}}$.
\end{lemma}

\begin{proof}
By \cref{prop:vertical_approx}, there exists a vertical map $h\colon \RealP{S\times\Delta^{\I}}\to \RealP{K}$, that is stratified homotopic to $\phi$ relative to $\RealP{{S_0}\times\Delta^{\I}}$.  
The image of $\bar{h}$ must be contained in the cells with labels containing $\Delta^{\I}$. We write $K_{\I}$ for the corresponding subsimplicial set of  $\bar K$. The situation can be summed up by the following commutative diagram:
\begin{equation*}
\begin{tikzcd}
\Real{{S_0}}
\arrow{r}{\Real{\bar{g}}}
\arrow[hookrightarrow]{d}
&\Real{K_{\I}}
\\
\Real{S}
\arrow[swap]{ur}{\bar{h}}
\end{tikzcd}
\end{equation*}
Now, by the non-stratified, relative version of \cref{theo:sdP_Approximation} (see also \cref{rem:relative_version_sdP_approximation}), there exists some $k\geq 0$ and some simplicial map $h'\colon \sd^k(S)\to K_{\I}$ such that the following diagram commutes:
\begin{equation*}
\begin{tikzcd}
\sd^k({S_0})
\arrow{r}{\bar{g}\circ\lv^k}
\arrow[hookrightarrow]{d}
&K_{\I}
\\
\sd^k(S)
\arrow[swap]{ur}{h'}
\end{tikzcd}
\end{equation*}
and such that $\Real{h'}$ is homotopic to $\bar{h}\circ\Real{\lv^k}$ relative to $\sd^k({S_0})$. But now the composition \[\sd^k(S)\times\Delta^{\I} \xrightarrow{h' \times 1_{\Delta^\I}} K_{\I} \times \Delta^{\I} \hookrightarrow {K} \] is the desired map.
\end{proof}
We may now combine the previous two lemmas, to obtain the following proposition.
\begin{proposition}\label{prop:existence_ell}
Let $S$ be a finite pointed simplicial set and $\I$ some regular flag in $P$. Then, for any $S_0\subset S$ a simplicial subset containing the base point, there exists a pointed vertical simplicial map
$\ell \colon \sd^k(S)\times\Delta^{\I}\to \sd_P(S\times\Delta^{\I})$
which restricts to a pointed vertical simplicial map
$ \ell_0\colon\sd^k({S_0})\times\Delta^{\I}\to \sd_P({S_0}\times\Delta^{\I}) $
such that the diagram
\begin{equation}
\label{eq:Diagram_l}    
\begin{tikzcd}
\sd^k(S_0)\times\Delta^{\I} \arrow[rd, "\lv \times 1_{\Delta^\I}"'] \arrow[r, "\ell_0"] & \sd_P(S_0\times\Delta^{\I}) \arrow[d, "\lv_P"]\\
 & S_0 \times \Delta^\I 
\end{tikzcd}
\end{equation}
commutes up to pointed sd-homotopy. 
\end{proposition}
Before we begin with the proof, let us quickly remark on the somewhat specific wording of \cref{prop:existence_ell}.
\begin{remark}
\cref{prop:existence_ell} is at the heart of the proof of \cref{prop:another_approximation}. We effectively use it in two ways. First, for the existence part of \cref{prop:another_approximation}, one uses the case $S_0=S$ guaranteeing the existence of a global $l$ with certain commutativity properties. Then, for the uniqueness part (up to sd-homotopy) one considers the case where $S$ is a cylinder with boundary $S_0$. In that case, it is crucial that the restriction of the global map $\ell$ to $\sd^k(S_0)\times\Delta^{\I}$ has its image in $\sd_P(S_0\times\Delta^{\I})$, however, the commutativity property is then only needed for that restriction.
\end{remark}

\begin{proof}[Proof of \cref{prop:existence_ell}]
To simplify notation, we omit all exponents from last vertex maps. Note, that $\RealP{\lv_P}\colon \RealP{\sd_P(S \times \Delta^\I)} \to \RealP{ S \times \Delta^\I}$ is a weak equivalence between cofibrant, fibrant objects with respect to the model structure on $\TopP^{\RealP{\Delta^\I}}$. In particular, it has an inverse, up to stratum preserving pointed homotopy equivalence. Denote these by $\gamma_S$ and $\gamma_{S_0}$ respectively. By naturality of $\lv_P$ the diagram 
    \[
    \begin{tikzcd}
    \RealP{{S_0} \times \Delta^\I} \arrow[r, "\gamma_{S_0}"] \arrow[d, hook] &\RealP{\sd_P{({S_0} \times \Delta^\I})} \arrow[d, hook] \\
        \RealP{S \times \Delta^\I} \arrow[r, "\gamma_S"] &\RealP{\sd_P(S \times \Delta^\I)}
    \end{tikzcd}
     \]
     is commutative in $\Ho\Top_P^{\Delta^{\I}}$. Now, since all objects involved are cofibrant, the diagram is commutative up to pointed stratum preserving homotopy. Now, first apply \cref{lem:vertical_simplicial_approx} to $\gamma_{S_0}$ with the role of the simplicial subset taken by the basepoint. For some $k' \geq 0$, we obtain a pointed vertical simplicial map $ \ell^1_{S_0}\colon \sd^{k'}({S_0}) \times \Delta^\I \to \sd_P({S_0} \times \Delta^\I)$ such that the following solid diagram commutes up to pointed, stratum preserving homotopy. To reduce the overload of notation, we omit the indices $P$ for stratified realizations.
    \[
    \begin{tikzcd}[column sep = large]
    \Real{\sd^{k'+k''}({S_0}) \times \Delta^\I}
    \arrow[d, hook, dotted] 
    \arrow[r, "\Real{\lv \times 1_{\Delta^\I}}", dotted]  
    &  \Real{\sd^{k'}({S_0}) \times \Delta^\I} 
    \arrow[d, hook]  
    \arrow[rr, bend left = 30, "\Real{\ell^1_{{S_0}}}"]{}{}
    \arrow[r, "\Real{\lv \times 1_{\Delta^\I}}"] 
    &\Real{{S_0} \times \Delta^\I} 
    \arrow[r, "\gamma_{S_0}"] 
    \arrow[d, hook] 
    &[-5 pt]\Real{\sd_P{({S_0} \times \Delta^\I})} 
    \arrow[d, hook] \\
    \Real{\sd^{k'+k''}(S) \times \Delta^\I} \arrow[rrr, bend right = 30, "\Real{\ell^2_S}", dotted]   \arrow[r, "\Real{\lv \times 1_{\Delta^\I}}", dotted]   & \Real{\sd^{k'}(S) \times \Delta^\I} \arrow[rr, bend right = 30, "\tilde \gamma_S", dashed ] \arrow[r, "\Real{\lv \times 1_{\Delta^\I}}"]    &   \Real{S \times \Delta^\I} \arrow[r, "\gamma_S"] &\Real{\sd_P{(S \times \Delta^\I})}.
    \end{tikzcd}
    \]
    Now, in this diagram the vertical left solid arrow is a cofibration in $\Top_P^{\RealP{\Delta^\I}}$ and all objects are cofibrant fibrant. Since the outer solid arrow diagram is commutative up to stratum preserving pointed homotopy, the homotopy extension property of cofibrations against fibrant objects gives the existence of a dashed map $\tilde \gamma_S$, making the square commute on the nose, and making the lower right triangle commute up to stratum preserving homotopy.
    Now, apply \cref{lem:vertical_simplicial_approx} to $\tilde \gamma_S$ and $\ell^1_{S_0}$ to obtain a pointed vertical simplicial map $\ell^2_S$ in the dotted part of the diagram, for $k''$ sufficently large. Again, the most outer part of the complete diagram commutes on the nose (in $\sS_P$) and the lower left triangle commutes up to stratum preserving pointed homotopy. 
    
    Now $\ell^2_{S}$ is not yet the map we are looking for, since we still need to check that \eqref{eq:Diagram_l} commutes up to pointed sd-homotopy.
    Write $\ell^2_{S_0}=\ell^1_{S_0}\circ (\lv^{k''}\times\Id_{\Delta^{\I}})$, for the restriction of  $\ell^2_{S}$ to $\sd^{k'+k''}({S_0}) \times \Delta^\I$. 
    We want to construct a simplicial homotopy between $\lv_P\circ\ell^2_{S_0}$ and $\lv\times\Id_{\Delta^{\I}}$, but to apply \cref{lem:vertical_simplicial_approx} we need two vertical simplicial maps. What we will do instead is, first replace $\lv_P\circ\ell^2_{S_0}$ by a vertical map, and then, exhibit a sd-homotopy from the latter to $\lv\times\Id_{\Delta^{\I}}$. Finally we will concatenate the two sd-homotopies using \cref{lem:concat_sd_ho}.
    By \cref{prop:simplicial_vertical_approx} there exists a simplicial pointed stratum preserving homotopy 
    \begin{equation*}
        H_1\colon \sd^{k'+k''}(S_0)\wedge\Delta^1_+\times\Delta^{\I}\to S_0\times\Delta^{\I}
    \end{equation*}
    from $\lv_P \circ \ell^2_{S_0}$ to a basepoint preserving vertical simplicial map $f\colon \sd^{k'+k''}(S_0)\times\Delta^{\I}\to S_0\times\Delta^{\I}$.  Note that, by construction, $\lv_P\circ\ell^2_{S_0}$ realizes to a map that is stratum preserving pointed homotopic to the realization of $\lv \times 1_{\Delta^\I}$, and so it must also be true of $f$. More specifically, there must exist a pointed stratified homotopy 
    \begin{equation*}
    H'\colon \RealP{\sd^{k'+k''}(S_0)\wedge\Delta^1_+\times\Delta^{\I}}\to \RealP{S_0\times\Delta^{\I}}     
    \end{equation*}
    Between $\RealP{f}$ and $\RealNP{\lv\times\Id_{\Delta^{\I}}}$. Since $f$ and $\lv\times\Id_{\Delta^{\I}}$ are both pointed vertical maps, they define a vertical map
    \begin{equation*}
        f\vee (\lv\times\Id_{\Delta^{\I}})\colon \sd^{k'+k''}(S_0)\times\Delta^{\I}\vee \sd^{k'+k''}(S_0)\times\Delta^{\I}\to S_0\times\Delta^{\I}
    \end{equation*}
     We can now apply the relative version of \cref{lem:vertical_simplicial_approx}
    to $H'$ and $f\vee (\lv \times\Id_{\Delta^{\I}})$ to get a simplicial map
    \begin{equation*}
        H_2\colon sd^{k'''}(\sd^{k'+k''}(S_0)\wedge\Delta^1_+)\times\Delta^{\I}\to S_0\times\Delta^{\I}
    \end{equation*}
    Now consider boundary preserving maps
    \begin{align*}
    \alpha_1\colon &\sd^{k'+k''}(S_0\wedge\Delta^1_+) \to \sd^{k'+k''}(S_0)\wedge\Delta^1_+\\
    \alpha_2 = \sd^{k'''}(\alpha_1)\colon & \sd^{k'+k''+k'''}(S_0\wedge\Delta^1_+)\to \sd^{k'''}(\sd^{k'+k''}(S_0)\wedge\Delta^1_+).
    \end{align*}
 The compositions $H_1'=H_1\circ(\alpha_1\times\Id_{\Delta^{\I}})\circ (\lv\times\Id_{\Delta^{\I}})$ and $H_2'=H_2\circ (\alpha\times\Id_{\Delta^{\I}})$ give pointed sd-homotopies respectively between $\lv_P\circ \ell^2_{S_0}\circ(\lv\times\Id_{\Delta^{\I}})$ and $f\circ (\lv\times\Id_{\Delta^{\I}})$ and between $f\circ (\lv\times\Id_{\Delta^{\I}})$ and  $\lv\times\Id_{\Delta^{\I}}$. Using \cref{lem:concat_sd_ho} we may concatenate these sd-homotopies
to a sd-homotopy
 \begin{equation*}
     H\colon \sd^{k'+k''+k'''+2}(S_0\wedge\Delta^1_+)\times \Delta^{\I}\to S_0\times\Delta^{\I},
 \end{equation*}
 between $\lv_P\circ\ell^2_{S_0}\circ (\lv\times\Delta^{\I})$ and $\lv\times\Delta^{\I}$.
 
  Finally, we take $k=k'+k''+k'''+2$, and $\ell$ to be the composition
 \begin{equation*}
     \begin{tikzcd}
     \sd^{k}(S)\times\Delta^{\I}
     \arrow{r}{\lv\times\Id_{\Delta^{\I}}}
     &\sd^{k'+k''}(S)\times\Delta^{\I}
     \arrow{r}{\ell^2_S}
     &\sd_P(S\times\Delta^{\I}).
     \end{tikzcd}
 \end{equation*}
 As, we have just proven, the restriction of this map to $\sd^{k}(S_0)\times\Delta^{\I}$ satisfies the homotopy commutativity of \eqref{eq:Diagram_l}.
    \end{proof}
   We now have all the necessary technical results to prove \cref{prop:another_approximation}.
   
    \begin{proof}[proof of \cref{prop:another_approximation}]
    We first prove the direct statement. Note that a pointing  $\Delta^{\I}\to K$ canonically lifts along $\lv_P$ to a pointing $\Delta^{\I}\to\sd_P(K)$. The lift is given by 
    the composition 
    \[
    \Delta^{\I} \to \sd_P (\Delta^{\I}) \to \sd_P(K)
    \]
    where the first map is specified by sending the the maximal non-degenate simplex $\mu$ of $\Delta^{\I}$ to $[(\mu,p_0),\dots,(\mu,p_n)]$, 
    for $\I = [ p_0 < \cdots < p_n]$.
    We then have a diagram in the pointed category 
    \[
    \begin{tikzcd}
    \RealP{S \times \Delta^\I} \arrow[rd, "\phi"] \arrow[r, dashed, "\hat \phi"]& \RealP{\sd_P(K)} \arrow[d, "\RealP{\lv_P}"]\\
    & \RealP{{K}}.
    \end{tikzcd}
    \]
    Note, that $\RealP{\lv_P}$ is a weak equivalence between fibrant objects in $\TopP^{\RealP{\Delta^\I}}$, and thus, $\phi$ admits a lift,
    $\hat \phi$ , making the above diagram commutative up to stratum preserving pointed homotopy. Now, apply the relative version of \cref{lem:vertical_simplicial_approx} to $\hat \phi$, to produce a pointed map $\hat f\colon \sd^k(S)\times \Delta^\I \to \sd_P(K)$. The map $f=\lv_P\circ\hat f\colon \sd^k(S)\times\Delta^{\I}\to K$ is then the desired map. Indeed, we have the following pointed stratum preserving homotopies:
    
    \[\RealP{f}=\RealP{\lv_P \circ \hat f}  \simeq_{P} \RealP{\lv_P} \circ \hat \phi \circ \RealP{ \lv \times 1_{\Delta^\I}}   \simeq_{P} \phi \circ  \RealP{ \lv \times 1_{\Delta^\I}}.
       \]
       
    Now, for the homotopy statement, take ${S'_0}:=S \vee S \hookrightarrow S \wedge \Delta^{1}_{+} =: S'$, and set 
    \[ g'=f_0 \cup_{\Delta^{\I}} f_1: {S_0}' \times \Delta^{\I} \to {K} \]
    and let  $ H'\colon \RealP{S'\times\Delta^{\I}}\to \RealP{K}$ be the stratified pointed homotopy between $\RealP{f_0}$ and $\RealP{f_1}$. Next, consider the commutative diagram
    \[
    \begin{tikzcd}[column sep = 60pt]
    \RealP{\sd_P ({S'_0} \times \Delta^\I)} 
    \arrow[r, "\RealP{\sd_P(g')}"] 
    \arrow[d, hook] 
    & \RealP{\sd_P{ {K}}} 
    \arrow[d, "\RealP{\lv_P}"] \\
    \RealP{\sd_P( S' \times \Delta^\I) } 
    \arrow[ru, dashed, "\phi'"]
    \arrow[r, "H'\circ \RealP{\lv_P}"] 
    & \RealP{{K}} \spaceperiod
    \end{tikzcd}
    \]
    Since the left hand vertical is a cofibration in $\Top_P$ and the right hand vertical is a weak equivalence between fibrant objects, there must exist some lift up to homotopy, $\phi'$, making the upper left triangle commute on the nose and the lower right triangle commute up to stratified homotopy relative to $\RealP{\sd_P ({S'_0} \times \Delta^\I)}$. Now, precompose this diagram with $\ell$ from \cref{prop:existence_ell}, for $k\gg 0$ to obtain the following commutative diagram:
    \begin{equation*}
        \begin{tikzcd}[column sep = 60pt]
        \RealP{\sd^{k}(S'_0)\times\Delta^{\I}} 
        \arrow{r}{\RealP{\ell_0}}
        \arrow[hookrightarrow]{d}
    &\RealP{\sd_P ({S'_0} \times \Delta^\I)} 
    \arrow[r, "\RealP{\sd_P(g')}"] 
    \arrow[d, hook] 
    & \RealP{\sd_P{ {K}}} 
    \\
    \RealP{\sd^{k}(S')\times\Delta^{\I}} 
        \arrow{r}{\RealP{\ell}}
    &\RealP{\sd_P( S' \times \Delta^\I) } 
    \arrow[ru, "\phi'"]
    \end{tikzcd}
    \end{equation*}
    
    Set $\phi= \phi' \circ \RealP{\ell}$ and $g = \sd_P(g') \circ \ell_0$. Now, apply \cref{lem:vertical_simplicial_approx} to $\phi$ and $g$. For $k' \gg 0$, we obtain a simplicial map 
    \begin{equation*}
            H^1\colon \sd^{k+k'}(S')\times \Delta^\I \to \sd_P(K)
    \end{equation*}
    whose restriction to $\sd^{k+k'}(S_0')\times\Delta^{\I}$ is given by $\sd_P(g') \circ \ell_0 \circ ( \lv \times \Id_{\Delta^\I})$. By naturality of $\lv_P$ one has $\lv_P\circ\sd_P(g')=g'\circ\lv_P$. In particular, $\lv_P\circ H^1\colon \sd^{k+k'}(S')\times \Delta^\I \to K$ restricts on the boundary to $g'\circ\lv_P\circ \ell_0\circ(\lv\times\Id_{\Delta^{\I}})$.
    By \cref{prop:existence_ell}, this map is sd-homotopic to $g' \circ (\lv \times \Id_{\Delta^\I})$ through a map of the form
    \begin{equation*}
     H^2\colon \sd^{k+k'} ({S'_0} \wedge \Delta^1_+) \times \Delta^{\I}  \to K
    \end{equation*}
      Now, the two sd-homotopies $\lv_P \circ H^1$ and $H^2$ can be concatenated. More explicitely, consider the inclusion $i_1\colon S'_0\to S'$ and $i_2\colon S'_0=S'_0\wedge\{0\}_+\to S'_0\wedge\Delta^1_+$. These maps induce inclusions \begin{align*}
          i'_1\colon &\sd^{k+k'} (S'_0)\times\Delta^{\I}\hookrightarrow \sd^{k+k'} (S')\times\Delta^{\I}\\
          i'_2\colon &\sd^{k+k'}(S'_0)\times\Delta^{\I}\hookrightarrow \sd^{k+k'}(S'_0\wedge\Delta^1_+)\times\Delta^{\I}
      \end{align*}
      The restrictions $\lv_P\circ H^1\circ i'_1$ and $H^2\circ i'_2$ are both equal to $g'\circ\lv_P\circ\ell_0\circ(\lv\times\Id_{\Delta^{\I}})$, by construction. In particular, $\lv_P\circ H^1$ and $H^2$ can be glued along $\sd^{k+k'} (S'_0)\times\Delta^{\I}$ to obtain a map
      \[ H^3\colon \sd^{k+k'}(S' \cup_{{S'_0}} (S'_0\wedge\Delta^1_+) ) \times \Delta^\I \to {K},
    \] 
     which restricts to $g'\circ(\lv\times\Id_{\Delta^{\I}})=(f_0\cup_{\Delta^{\I}}f_1) \circ (\lv \times 1_{\Delta^\I})$ on the boundary. Now to turn this map into an sd-homotopy, consider the following composition:
     \begin{equation*}
        \alpha\colon\sd^{k+k'+2}(S\wedge \Delta^1_+)\to \sd^{k+k'}(S\wedge\sd^2(\Delta^1_+))\to \sd^{k+k'}(S'\cup_{S'_0}(S'_0\wedge\Delta^1_+)),
    \end{equation*}
    where the left hand map comes from the natural map $\sd^2(S\wedge \Delta^{1}_+)\to S\wedge\sd^2(\Delta^1_+)$ and the right hand map is the one suggested by \cref{fig:Cylinders}. To construct the latter explicitely, it is enough to note that there is an isomorphism \[ 
      S' \cup_{{S'_0}} (S'_0\wedge\Delta^1_+) \cong S \wedge (\Delta^1\cup_{1,0} \Delta^1\cup_{1,1} \Delta^1)_+,
      \]
      and a boundary preserving map $\sd^2(\Delta^1)\to \Delta^1\cup_{1,0} \Delta^1\cup_{1,1} \Delta^1$. Now, the desired pointed sd-homotopy is given by the composition
      \begin{equation*}
        H=H^3\circ(\alpha\times\Id_{\Delta^{\I}})\colon \sd^{k+k'+2}(S\wedge\Delta^1_+)\times\Delta^{\I}\to K.
    \end{equation*}
\end{proof}

\begin{figure}[h]
\begin{tikzpicture}
\draw (0,0)--(0,1);
\filldraw[black] (0,0) circle (1pt);

\draw[shift= {(1.5,0)}] (0.5,1)--(0,0)--(-0.5,1);
\filldraw[shift= {(1.5,0)},black] (0,0) circle (1pt);

\filldraw[shift= {(3,0)}, black,opacity=0.1](0.5,1)--(0,0)--(-0.5,1);
\draw[shift= {(3,0)}] (0.5,1)--(0,0)--(-0.5,1);
\filldraw[shift= {(3,0)},black] (0,0) circle (1pt);
\draw[shift= {(3,0)},->] (-0.2,0.8)--(0.2,0.8);
\node[shift= {(3,0)}] at (0,-0.5) {$H_1$};

\filldraw[shift= {(5,0)}, black,opacity=0.1](-1,0)--(0,0)--(-0.5,1);
\filldraw[shift= {(5,0)}, black,opacity=0.1](1,0)--(0,0)--(0.5,1);
\draw[shift= {(5,0)}] (0.5,1)--(0,0)--(-0.5,1);
\draw[shift= {(5,0)}] (-1,0)--(0,0)--(1,0);
\filldraw[shift= {(5,0)},black] (0,0) circle (1pt);
\draw[shift= {(5,0)},->] (-0.75,0.15)--(-0.5,0.65);
\draw[shift= {(5,0)},->] (0.75,0.15)--(0.5,0.65);
\node[shift= {(5,0)}] at (0,-0.5) {$H_2$};

\filldraw[shift= {(7.5,0)}, black,opacity=0.1](0.5,1)--(0,0)--(-0.5,1);
\filldraw[shift= {(7.5,0)}, black,opacity=0.1](-1,0)--(0,0)--(-0.5,1);
\filldraw[shift= {(7.5,0)}, black,opacity=0.1](1,0)--(0,0)--(0.5,1);
\draw[shift= {(7.5,0)}] (0.5,1)--(0,0)--(-0.5,1);
\draw[shift= {(7.5,0)}] (-1,0)--(0,0)--(1,0);
\filldraw[shift= {(7.5,0)},black] (0,0) circle (1pt);
\draw[shift= {(7.5,0)},->] (-0.75,0.15)--(-0.5,0.65);
\draw[shift= {(7.5,0)},->] (0.75,0.15)--(0.5,0.65);
\draw[shift= {(7.5,0)},->] (-0.2,0.8)--(0.2,0.8);
\node[shift= {(7.5,0)}] at (0,-0.5) {$H_3$};

\filldraw[shift={(10,0)}, black, opacity=0.1](-1,0)--(-0.8,0.8)--(0,1)--(0.8,0.8)--(1,0);
\draw[shift={(10,0)}](-1,0)--(0,0)--(-0.8,0.8)--(0,0)--(0,1)--(0,0)--(0.8,0.8)--(0,0)--(1,0);
\filldraw[shift= {(10,0)},black] (0,0) circle (1pt);
\draw[shift= {(10,0)},->] (-0.8,0.15)--(-0.7,0.5);
\draw[shift= {(10,0)},->] (-0.15,0.8)--(-0.5,0.7);
\draw[shift= {(10,0)},->] (0.15,0.8)--(0.5,0.7);
\draw[shift= {(10,0)},->] (0.8,0.15)--(0.7,0.5);
\node[shift= {(10,0)}] at (0,-0.5) {$H_3\circ\alpha$};

\end{tikzpicture}

\caption{From left to right, schematic pictures of the simplicial sets, $S$, $S'_0$, $S'$, $S'_0\wedge\Delta^1_+$, the union $S'\cup_{S'_0}(S'_0\wedge\Delta^1_+)$ and the smash product $S\wedge\sd^2(\Delta^1_+)$. Pointing are indicated by large dots, and arrows indicate direction of homotopies.}
\label{fig:Cylinders}
\end{figure}
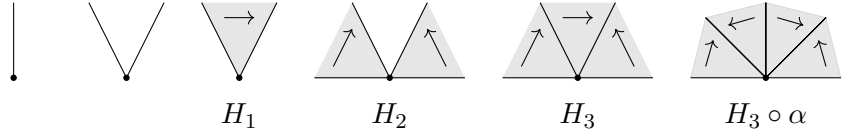

\section*{Acknowledgements}

The first author was supported by a grant from the Fondation Sciences Mathématiques de Paris, while a post-doc at IRIF, UMR 8243, Université de Paris,  during part of the writing of this article, and is now receiving support from Dan Petersen's Wallenberg Academy Fellowship. The second author was supported by a grant from the Landesgraduiertenförderung Baden-Württemberg.
The authors would also like to thank David Chataur for helpful discussions during the writing of this article.

\appendix

\section{Incompatible criteria for model structures over $\Top_P$}
\label{section:appendix}

The goal of this section is to prove the following proposition, and its corrolary.

\begin{proposition}\label{prop:Nonexistence_Appendix}
For $P$ a non-discrete poset, there exists no model structure on $\Top_P$ satisfying all of the following properties:
\begin{enumerate}
    \item \label{Item:Mono_Are_Cof_Appendix}Realizations of monomorphisms are cofibrations.
    \item \label{Item:Strat_HE_WE_Appendix}Stratified homotopy equivalences are weak equivalences.
    \item \label{Item:WE_Induce_WE_Appendix} For a weak equivalence, the induced map between classical homotopy links is also
a weak-equivalence.
\end{enumerate}
\end{proposition}

\begin{remark}\label{Rem:StrongerNonExistence}
In fact, we will prove a slightly stronger version of \cref{prop:Nonexistence_Appendix}, where condition \ref{Item:WE_Induce_WE_Appendix} is replaced by the following:
\begin{enumerate}
    \item [(3')] \label{Item:WeakerCondition}  For some fixed flag, $\I= \{ p <q \} $, the homotopy link functor $\HolIP$ sends weak equivalences to weak equivalences of topological spaces. 
\end{enumerate}
\end{remark}

\begin{corollary}
If $P$ is a non-discrete poset, then the model structure on $\Top_P$ transported from $\sS_P$ along the adjunction $\RealP{-}\colon\sS_P\leftrightarrow\Top_P\colon\Sing_P$ does not exist.
\end{corollary}

Let us first deduce the corollary from the proposition.
\begin{proof}
Assume that the transported model structure on $\Top_P$ exists. Then $\RealP{-}\dashv \Sing_P$ is a Quillen adjunction and so the functor $\RealP{-}$ must preserve cofibrations, which implies that the model structure satisfies \eqref{Item:Mono_Are_Cof_Appendix}. Furthermore, $\Sing_P$ preserves stratified homotopy equivalences, so the model structure must satisfy \eqref{Item:Strat_HE_WE_Appendix}. Finally, if $f\colon X\to Y$ is a weak-equivalence, then, by definition, $\Sing_P(f)$ must be a weak-equivalence in $\sS_P$. By \cref{Cor:Diagram_Preserve_WE}, this implies that $f$ induces weak-equivalences  $\Hol_{\I}(\Sing_P(X))\to\Hol_{\I}(\Sing_P(Y))$ for all regular flags $\I$. By \cref{Rem:Comparing_Holink_Intro}, this implies that $f$ induces weak-equivalences $\HolIP(X)\to\HolIP(Y)$ for all regular flags $\I$. In particular, the model category must also satisfy \eqref{Item:WE_Induce_WE_Appendix}. But then, it satisfies simultaneously \eqref{Item:Mono_Are_Cof_Appendix}, \eqref{Item:Strat_HE_WE_Appendix} and \eqref{Item:WE_Induce_WE_Appendix}, which is a contradiction.
\end{proof}

\cref{prop:Nonexistence_Appendix} also has implications for other model structures.
\begin{remark}
Consider Haine's model stucture on the category of stratified simplicial sets, $\sS_P^{\text{Joyal-Kan}}$ \cite{haine2018homotopy}. It can not be transported along the adjunction $(\RealP{-},\Sing_P)$ to a model structure on $\Top_P$. Indeed, supposing this is possible, this hypothetical model structure would satisfy \eqref{Item:Mono_Are_Cof_Appendix}, since the cofibrations in $\sS_P^{\text{Joyal-Kan}}$ are the monomorphisms. Furthermore, it would satisfy \eqref{Item:Strat_HE_WE_Appendix}, since $\Sing_P$ preserves stratified homotopy equivalences, and those are weak-equivalences in $\sS_P^{\text{Joyal-Kan}}$. But then, it can not satisfy \eqref{Item:WE_Induce_WE_Appendix}, by \cref{prop:Nonexistence_Appendix}. This implies that if such a model structure existed, the classical homotopy links would not be invariants of the stratified homotopy type. 

However, we can show that \textbf{some} homotopy links would have to be preserved, preventing the existence of this model structure altogether, by the strengthened version of \cref{prop:Nonexistence_Appendix} (see \cref{Rem:StrongerNonExistence}). 
Indeed, let us first assume that $P$ admits a successor pair, that is $p<q\in P$ such that there exists no $m\in P$ satisfying $p<m<q$. Then, by \cref{lem:haines_pairwise_links}, and recalling that 
\begin{equation*}
    \Sing(\HolIP(X))=\Hol_{\I}(\Sing_P(X)),
\end{equation*}
we see that $\HolIP$ sends weak equivalences to weak equivalences of spaces, for $\I=[p<q]$. The argument is slightly more subtle if $P$ does not admit successor pairs. Let $p<q\in P$, and consider stratified spaces of the form $X\to P$ such that the $m$ stratum of $X$ is empty for all $p<m<q$. Then, $\HolIP$ sends weak equivalences between such stratified spaces to weak equivalences of topological spaces, for $\I=[p<q]$. But this also leads to a contradiction when considering \cref{Ex:Appendix_pathological} as stratified over $\{p<q\}\subset P$.
\end{remark}

\begin{remark}
Similarly, let $\mathrm{Strat}^{N}$ be the category of such stratified spaces, which have no empty strata, as defined in \cite{nand2019simplicial}. In \cite{nand2019simplicial} the author asked the question, whether the structure on the model category for quasi categories $\sS^{Joyal}$ can be transferred to $\mathrm{Strat}^{N}$ along an adjunction defined using the functors $(\Real{-},\Sing_{\Strat})$. Arguing as in the proof of \cref{prop:Nonexistence_Appendix}, and using \cref{Ex:Appendix_pathological}, we may show that it is indeed not possible.

Assume that the transported model structure exists. Then, since cofibrations in $\sS^{\text{Joyal}}$ are monomorphisms, the model category would satisfy the analogue of \eqref{Item:Mono_Are_Cof_Appendix}. Furthermore, assume that $H\colon X\times [0,1]\to Y$ is a stratified  homotopy, between two stratified maps $f,g\colon X\to Y$, and consider the following composition 
\begin{equation}
    \Sing_{\text{Strat}}(X)\times \Sing([0,1])\to \Sing_{\text{Strat}}(X\times [0,1]))\to \Sing_{\text{Strat}}(Y).
    \end{equation}
This composition is a homotopy (in the Joyal model structure) between $\Sing_{\text{Strat}}(f)$ and $\Sing_{\text{Strat}}(g)$. This follows from the fact that $\Sing([0,1])$ is a cylinder in the Joyal model structure. In particular, in the transported model structure on $\Strat^N$, stratum preserving homotopy equivalences are weak-equivalences, i.e. the model category satisfies the analoge of \eqref{Item:Strat_HE_WE_Appendix}. 
Now, assume that $f \colon X \to Y$, is a stratum preserving map over the poset $P = \{ 0 < 1 \}$ such that $\Sing_{\text{Strat}}(f)$ is a Joyal equivalence in $\sS$. Then, the map $\Sing_P(f)\colon \Sing_P(X)\to\Sing_P(Y)$ is a map in $\sS_P$, and is a weak equivalence in the structure $\sS_P^{\text{Joyal}}$ (that is, the slice model category). But since $\sS_P^{\text{Joyal-Kan}}$ is a Bousfield localization of the former, $\Sing_P(f)$ is also a weak equivalence in $\sS_P^{\text{Joyal-Kan}}$. Finally, note that for $P=\{0<1\}$, the model structure on $\sS_P^{\text{Joyal-Kan}}$ coincides with the one studied in this paper, giving that $\Sing_P(f)$ must be a weak equivalence in $\sS_P$, and hence, must induce weak-equivalences
\begin{equation*}
    \HolIP(X)\to \HolIP(Y).
\end{equation*}
Applying the above to the map of \cref{Ex:Appendix_pathological} gives a contradiction, hence the transported model structure on $\Strat^N$ may not exist.
\end{remark}

\cref{prop:Nonexistence_Appendix} is a direct consequence of the following example.
\begin{example}\label{Ex:Appendix_pathological}
Consider the simplicial complex $Y$ with vertices $a,b,c,d,e,f$ generated by the simplices $\{a,b,c,d\}$, $\{a,c,e\}$ and $\{a,d,f\}$ and with the stratification over $P=\{0<1\}$ defined as follows. The entire segment $\{a,b\}$ is sent to $0$, then, consider two segments between $a$ and the middle of $\{d,f\}$ and between $a$ and the middle of $\{e,c\}$, which we will promptly identify with two copies of the interval $[0,1]$ (with $a$ at $0$). In those intervals, send all the points of the sequence $\frac{1}{2^{n+1}}$, for $n\geq 0$, to $0$ (any decreasing sequence converging to $0$ will give an isomorphic result). The remaining points are mapped to $1$.

Note that the stratified space obtained in this way is not isomorphic in $\Top_P$ to the realization of a stratified simplicial set. In particular, $Y$ does not come from a strongly stratified space. Nevertheless one can compute its homotopy link with respect to the flag $\I=[0<1]$. We will also consider the subspace $X\subset Y$ obtained by deleting the maximal simplex $\{a,b,c,d\}$ and its face $\{a,c,d\}$, see \cref{Fig:Example_Appendix}. 

We will only be interested in $\pi_0(\HolIP(X))$, and $\pi_0(\HolIP(Y))$ or in other words, in exit-paths up to stratum preserving homotopies. Note that any exit path starting from one of the isolated points is entirely determined - up to stratum preserving homotopy - by its starting point. Similarly, any point starting in the interval $(a,b]$ is equivalent to the exit-path spanning the segment $[b,d]$. 

On the other hand, there are a vast number of inequivalent classes of exit-paths starting from $a$. A particular set of exit-paths that can be easily described is those that get away from $a$ steadily (as in \cref{Fig:Example_Appendix}). For such an exit-path in $X$, it is enough to specify which face ($\{a,d,f\}$ or $\{a,c,e\}$) it lies on, and then for each isolated singular point if the path passes under, or over. This can be summarized as a binary sequence (where the first few terms, corresponding to points away from $a$ might be ill-defined).  One then checks that two such exit-paths in $X$ are equivalent if and only if they lie on the same face and their associated sequence differ in only finitely many places. This implies in particular that $\pi_0(\HolIP(X))$ is uncountable. 

Now consider $\pi_0(\HolIP(Y))$. The paths we just described in $X$ also exist as exit-paths in $Y$, but there also exist paths going back on forth between the faces $\{a,c,e\}$ and $\{a,d,f\}$. If one restricts to those paths that are moving away from $a$, it is still possible to parametrize those by ternary sequences, where the $n$-th entry indicates if the path passes to the left, right or middle of the $n$-th pair of isolated singular points. One can show again that two such paths are equivalent up to stratum-preserving homotopies if and only if their associated sequences differ only at finitely many places. But now, note that an exit-path in $Y$ zig-zaging infinitely many times between the left side of $\{a,c,e\}$ and the right side of $\{a,d,f\}$ can not be in the image of $\pi_0(\HolIP(X))\to\pi_0(\HolIP(Y))$. This implies that $\HolIP(X)\to\HolIP(Y)$ is not a weak-equivalence, and in turn that $X\to Y$ is not a weak-equivalence in $\Top_P$.
\end{example}
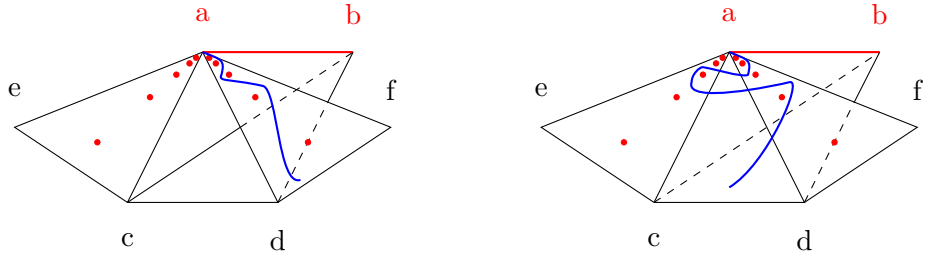
\begin{figure}[h]
\begin{tikzpicture}

\draw[black](0,-1)--(2,-1)--(1,1)--(-1.5,0)--(0,-1);
\draw[black](0,-1)--(1,1);
\draw[black](2,-1)--(3.5,0)--(1,1);
\draw[black](0,-1)--(1.5,0);
\draw[black,dashed](1.5,0)--(3,1);
\draw[red,thick](1,1)--(3,1);
\draw[black,dashed](2,-1)--(2.66,0.33);
\draw[black](2.66,0.33)--(3,1);

\filldraw[red] (1.0875,0.925) circle (1pt);
\filldraw[red] (1.175,0.85) circle (1pt);
\filldraw[red] (1.350,0.7) circle (1pt);
\filldraw[red] (1.700,0.40) circle (1pt);
\filldraw[red] (2.400,-0.20) circle (1pt);

\filldraw[red] (0.9125,0.925) circle (1pt);
\filldraw[red] (0.825,0.85) circle (1pt);
\filldraw[red] (0.650,0.7) circle (1pt);
\filldraw[red] (0.3,0.40) circle (1pt);
\filldraw[red] (-0.4,-0.20) circle (1pt);

\node at (1,1.5) {\begin{red}a\end{red}};
\node at (3,1.5) {\begin{red}b\end{red}};
\node at (0,-1.5) {c};
\node at (2,-1.5) {d};
\node at (-1.5,0.5) {e};
\node at (3.5,0.5) {f};

\draw[blue,thick](1,1).. controls (1.25,0.9) and (1.3,0.9).. (1.25,0.7)..controls (1.2,0.6) and (1.35,0.65).. (1.8,0.55)..controls (2,0.5) and (2,-0.8)..(2.3,-0.7);

\draw[shift={(7,0)},black](0,-1)--(2,-1)--(1,1)--(-1.5,0)--(0,-1);
\draw[shift={(7,0)},black](0,-1)--(1,1);
\draw[shift={(7,0)},black](2,-1)--(3.5,0)--(1,1);
\draw[shift={(7,0)},black,dashed](0,-1)--(3,1);
\draw[shift={(7,0)},red,thick](1,1)--(3,1);
\draw[shift={(7,0)},black,dashed](2,-1)--(2.66,0.33);
\draw[shift={(7,0)},black](2.66,0.33)--(3,1);

\filldraw[shift={(7,0)},red] (1.0875,0.925) circle (1pt);
\filldraw[shift={(7,0)},red] (1.175,0.85) circle (1pt);
\filldraw[shift={(7,0)},red] (1.350,0.7) circle (1pt);
\filldraw[shift={(7,0)},red] (1.700,0.40) circle (1pt);
\filldraw[shift={(7,0)},red] (2.400,-0.20) circle (1pt);

\filldraw[shift={(7,0)},red] (0.9125,0.925) circle (1pt);
\filldraw[shift={(7,0)},red] (0.825,0.85) circle (1pt);
\filldraw[shift={(7,0)},red] (0.650,0.7) circle (1pt);
\filldraw[shift={(7,0)},red] (0.3,0.40) circle (1pt);
\filldraw[shift={(7,0)},red] (-0.4,-0.20) circle (1pt);

\node[shift={(7,0)}] at (1,1.5) {\begin{red}a\end{red}};
\node[shift={(7,0)}] at (3,1.5) {\begin{red}b\end{red}};
\node[shift={(7,0)}] at (0,-1.5) {c};
\node[shift={(7,0)}] at (2,-1.5) {d};
\node[shift={(7,0)}] at (-1.5,0.5) {e};
\node[shift={(7,0)}] at (3.5,0.5) {f};

\draw[shift={(7,0)},blue,thick](1,1).. controls (1.25,0.9) and (1.3,0.9).. (1.25,0.7).. controls (1.2,0.65) and (0.7,0.8).. (0.65,0.8)..controls (0.6,0.8) and (0.5,0.7)..(0.5,0.6)..controls (0.4,0.3) and (1.6,0.55)..(1.8,0.6)..controls (2,0.65) and (1.5,-0.5)..(1,-0.8);

\end{tikzpicture}
\caption{the simplicial complexes $X$ and $Y$ with their pathological stratifications over $P=\{0<1\}$. The isolated singular points admit a as an accumulation point. Two exit-paths starting from $a$ have been represented in blue. }
\label{Fig:Example_Appendix}
\end{figure}

\begin{proof}[Proof of \cref{prop:Nonexistence_Appendix}]
Assume that there exists a model structure on $\Top_P$ satisfying \eqref{Item:Mono_Are_Cof_Appendix} and \eqref{Item:Strat_HE_WE_Appendix}. This implies that  $\RealP{\Lambda^{\J}_k}\to\RealP{\Delta^{\J}}$ is a trivial cofibration in this model structure for any admissible horn inclusion $\Lambda^{\J}_k\to \Delta^{\J}$ since it is simultaneously the realization of a cofibration and a stratified homotopy equivalence (see \cite[Proposition 1.13]{douSimp}). Now note that the inclusion $X\hookrightarrow Y$ from \cref{Ex:Appendix_pathological} can be obtained as the following pushout
\begin{equation*}
    \begin{tikzcd}
         \RealP{\Lambda^{\J}_k}
         \arrow{r}
         \arrow{d}
         &X
         \arrow{d}
         \\
         \RealP{\Delta^{\J}}
         \arrow{r}
         &Y
    \end{tikzcd}
\end{equation*}
where $\Lambda^{\J}_k\to\Delta^{\J}$ is an admissible horn inclusion. In particular, $X\to Y$ is the pushout of a trivial cofibration, so it must be a trivial cofibration. But as demonstrated in \cref{Ex:Appendix_pathological}, the inclusion $X\to Y$ does not induce weak-equivalences on all classical homotopy links. In particular, the model structure can not satisfy \eqref{Item:WE_Induce_WE_Appendix}, which concludes the proof.
\end{proof}

 \begin{remark}
One might object, that \cref{Ex:Appendix_pathological} is pathological insofar, as the $0$-stratum considered as a subspace in the usual subspace topology is not $\Delta$-generated. It may be possible to pass to an appropriate tamer subcategory of $\Top_P$, which would then allow for the structure of a model category fulfilling all of the requirements of \cref{prop:Nonexistence_Appendix}. The main point of this article however, is that even without attempting to change the underlying point-set topological framework and defining such a model structure, many of the consequences of its existence may nevertheless be obtained.
\end{remark}
\begin{lemma}\label{lem:haines_pairwise_links}
Let $\I = \{ p <q\}$ be a flag in $P$ and let $X\in \sS_P^{Joyal-Kan}$ be such that $X_m$ is empty, for all $ p < m < q$. Then, for any weak equivalence $X \to Y$ in $\sS_P^{Joyal-Kan}$, the induced map on simplicial homotopy links, $\Hol_{\I}(X) \to \Hol_{\I}(Y)$, is a weak equivalence in the Kan model structure on simplicial sets.
\end{lemma} 

\begin{proof}
Recall that the fibrant objects in $\sSJK_P$ are those that admit the right lifting property against all stratified horn inclusions $\Lambda^{\J}_k$ which are either inner horn inclusions, or admissible horn inclusions. Call those horn inclusions weakly admissible. Now, given a weak-equivalence in $\sSJK_P$, $f\colon X\to Y$, where the $m$-stratum of $X$ is empty for all $p<m<q$, consider fibrant replacements for $X$ and $Y$ obtained by the small object argument applied to the set of weakly admissible horn inclusions. We get a commutative diagram of weak equivalences
\begin{equation*}
    \begin{tikzcd}
         X
         \arrow{r}{f}
         \arrow[hookrightarrow]{d}
         &Y
         \arrow[hookrightarrow]{d}
         \\
         X^{\text{fib}}
         \arrow{r}{f^{\text{fib}}}
         & Y^{\text{fib}}
    \end{tikzcd}
\end{equation*}
Now, since $\sSJK_P$ is a left Bousfield localization of $\sS_P$ (see \cref{rem:Haine_Structure_is_localization}), and since the map $f^{\text{fib}}$ is a weak equivalence between fibrant object in $\sSJK_P$, it must also be a weak equivalence in $\sS_P$. But then, by \cref{Cor:Diagram_Preserve_WE}, $f^{\text{fib}}$ must induce a weak equivalence $\Hol_\I(X^{\text{fib}})\to \Hol_\I(Y^{\text{fib}})$. By two out of three, it is thus enough to show \cref{lem:haines_pairwise_links} for maps of the form $X\hookrightarrow X^{\text{fib}}$ (noting that the $m$-strata of $Y$ must also be empty for $p<m<q$, since weak equivalences in $\sSJK_P$ preserve the homotopy type of strata).
Now, by \cref{theo:Intro_Holink}, together with the (pseudo)-naturality of $\Link{\I}$ with respect to monomorphisms (see \cref{prop:Links_and_Links}), we can reduce to comparing $\Link{\I}(X)$ and $\Link{\I}(X^{\fib})$. Finally, since the map $X\hookrightarrow X^{\text{fib}}$
is a transfinite composition of pushouts along weakly admissible horn inclusions, and since $\Link{\I}$ preserves colimits, it is enough to prove \cref{lem:haines_pairwise_links}
for maps $X\hookrightarrow Y$ obtained by pushing out along a single weakly admissible horn inclusion, $\Lambda^{\J}_k\to \Delta^{\J}$. Now, note that since $X$ has empty $m$ strata for all $p<m<q$, so must have $\Lambda^{\J}_k$, and $\Delta^{\J}$. Finally, it is enough to show that 
\begin{equation*}
    \Link{\I}(\Lambda^\J_k)\hookrightarrow \Link{\I}(\Delta^\J)
\end{equation*}
is a weak equivalence of simplicial sets, whenever $\Delta^{\J}$ has empty $m$-strata, for $p<m<q$.
Let $\J=[p_0\leq \dots \leq p_n]$, and consider the following cases.
\begin{itemize}
\item if $\Lambda^{\J}_k$ is an admissible horn, then this follows from \cref{lem:Link_preserve_trivial_cofibrations}.
\item else, $\Lambda^{\J}_k$ is an inner horn, and thus $n\geq 2$. But then, the only vertices that can possibly be in $\Link{\I}(\Delta^\J)$ but not in $\Link{\I}(\Lambda^\J_k)$ are those corresponding to the simplex $\Delta^\J$
and to its face, $\Delta^{\J'}=d_k(\Delta^{\J})$. Those would be in $\Link{\I}(\Delta^\J)$ if and only if they degenerate from $\Delta^{\I}$. If the former degenerates from $\Delta^{\I}$, $\Lambda^{\J}_k$ is an admissible horn. If the latter degenerates from $\Delta^{\I}$, 
then $\Delta^{\J'}=[p_0,\dots,\widehat{p_k},\dots,p_n]$ with $p_i=p$ or $q$ for all $i\not =k$. But since $p_{k-1}\leq p_k\leq p_{k+1}$,
and since $\Delta^{\J}$ has empty $m$-strata for all $p<m<q$, then $p_k=p$ or $q$, and the horn is also admissible. Finally, in all the remaining cases, we have $\Link{\I}(\Lambda^{\J}_k)=\Link{\I}(\Delta^{\J})$.
\end{itemize}
\end{proof}

\section{Relating labellings, vertical stratifications and diagrams.}
\label{section:appendixB}

In \cref{rem:Vertical_Are_Diagrams} we have already hinted at the fact that
$P$-labelled simplicial sets can be thought of as a particularly concise description of certain cofibrant diagrams in $\Diag_P$. Let us now expand on this and make the relationship between labelled objects, diagrams and vertical objects precise. 
Before we do so, let us quickly remark on the topological counterpart of $\Diag_P$. We denote $\Diag_P^{\Top} := \mathrm{Fun}(R(P)^{\mathrm{op}}, \Top)$.
\begin{remark}\label{rem:diag_top}
Just like its simplicial counterpart $\Diag_P^{\Top}$ can be equipped with the projective model structure (use for example \cite[Thrm 11.6.1]{hirschhornModel}). Assigning to a $P$-stratified space $X$ the diagram given by its (topological) homotopy links $\I \mapsto \HolIP(X)$, and conversely sending a diagram $F$ to $\int^{\I} F(\I) \times \RealP{\Delta^{\I}}$ then defines a Quillen adjunction 
\[D'_P \colon \Diag_P^{\Top} \leftrightarrow \Top_P \colon C'_P.\]
We obtain a diagram of Quillen functors
\[
\begin{tikzcd}
         & \arrow[ld, shift left = 4pt ]\Diag^{\Top}_P \arrow[rd, "C'_P"]& \\
         \Diag_P \arrow[ru] \arrow[rr, "C^{\Top}_P"]&& \arrow[ll, shift left = 4pt, "D^\Top_P"] \arrow[lu, "D'_P",  shift left = 4 pt ]\Top_P,
\end{tikzcd}
\]
where the left diagonals are induced by the adjunction $\Real{-} \dashv \Sing$. Both the left adjoint, as well as the right adjoint part of this diagram commute up to natural isomorphism. Hence, it follows from the two out of three property for Quillen equivalences that $C'_P \dashv D'_P$ is also a Quillen equivalence.
\end{remark}
Next, let us describe the relationship between labelled objects and diagrams defined on $R(P)^{\op}$.
    Both $P$-labeled simplicial sets, as well as CW-complexes can readily be equipped with a functor 
    \begin{align*}
        U\colon \PsS &\to \Diag_P, \\
        U^{\Top}\colon \PCW &\to \Diag_P^{\Top}
    \end{align*}
    respectively.
    In case of a $P$-labeled CW-complex $\lab{T}$, the diagram $U^{\Top}\lab T \in \Diag_P^{\Top}$, at a regular flag $\I$, is given by the unions of cells
    \[U^{\Top}\lab T(\I) = \bigcup_{e_{\alpha},\ {\Delta^{\I}} \subset \lambda_T(e_\alpha)} e_{\alpha},\]
    and structure maps are given by inclusions. One extends this construction to morphisms in the obvious way.
    The definition for simplicial sets is analogous, replacing open cells by non-degenerate simplices. The precise behavior of the $U$ functors is then described in the following proposition:
    \begin{proposition}\label{prop:properties_of_U_functors}
     The functors $U$ and $U^{\Top}$ are fully faithful, and fit into a commutative diagram 
    \[
    \begin{tikzcd}
        \PsS\arrow[d] \arrow[r, "U"]& \Diag_P \arrow[d]\\
        \PCW \arrow[r, "U^\Top"]& \Diag_P^{\Top}
    \end{tikzcd}
    \]
    with verticals induced by realization. 
    Both $U$ and $U^{\Top}$ factor thought the respective subcategories of cofibrant diagrams.
    Furthermore, $U$ induces an equivalence of categories $\PsS \xrightarrow{\sim} \Diag_P^{\cof}$.
    \end{proposition}
    \begin{proof}
    Commutativity is easily verified from the constructions, while fully faithfulness is an immediate consequence of the definition of labelled maps. That both $U$ functors have image in the subcategories of cofibrant objects follows from the characterization of generating cofibrations in a projective model structure in \cite[Thm. 11.6.1]{hirschhornModel}. Finally, that every cofibrant diagram in $\Diag_P$ is in fact (up to natural isomorphism) of the form $U\lab S$, for $\lab S \in \PsS$, follows from the characterization of cofibrant diagrams in \cref{prop:Cofibrant_Diagrams} (see also \cref{rem:Vertical_Are_Diagrams}).
    \end{proof}
    In this sense, $P$-labeled simplicial sets are a particularly concise, but equivalent, description of cofibrant diagrams. Since not every absolute cell complex is a CW-complex, the analogous essential surjectivity fails for $P$-labelled CW-complexes. However, as a consequence of \cref{Cor:CW_Model}, essential surjectivity is restored when passing to homotopy categories. \\
    \\
    Next, let us study the precise relationship of diagrams with verticalization.
    One easily verifies that, the two diagrams
   \[ \begin{tikzcd}
        \PsS \arrow[rd, "V"'] \arrow[rr, "U"]& & \Diag_P \arrow[ld, "C_P"] \spacecomma \\ 
        & \sS_P&
    \end{tikzcd}
    \begin{tikzcd}
          \PCW \arrow[rd, "V"']\arrow[rr, "U^\Top"]& & \Diag^{\Top}_P \spacecomma \arrow[ld, "C'_P"] \\ 
        & \Top_P&
    \end{tikzcd}
    \]
    commute up to natural isomorphism. In this sense, we can think of verticalization as an alternative description of the $C_P$ functors. Now, clearly the verticalization functors are not full. We can amend this by passing to the vertical setting. Denote by $\mathrm{V}\sS_P$ and $\mathrm{VCW}_P$ respectively the categories of vertically stratified simplicial sets and CW-spaces, with vertical maps. Using the natural pre-verticalizations which the verticalization functors come with, we obtain lifts 
    \begin{align*}
        V \colon \PsS &\to \mathrm{V}\sS_P \\
        \lab{S} &\to \{ V\lab{S} \hookrightarrow S \times N(P) \}, \\
        V \colon \PCW &\to \mathrm{VCW}_P \\
        \lab{T} &\to \{ V\lab{T} \hookrightarrow T \times \RealP{N(P)}\},
    \end{align*}
    denoted the same by abuse of notation. It turns out, that both these functos induce equivalences of categories. One may immediately verify from the definition of label preserving maps, that these two functors are fully faithful. They are essentially surjective, by construction of the categories of vertical objects. \\
    Hence, verticalization defines equivalences between $P$-labeled and vertical objects. In this sense, using \cref{prop:properties_of_U_functors}, the vertical categories give a description of (certain) cofibrant diagrams, which is more intrinsic to $\sS_P$ or $\Top_P$.
    Together with \cref{prop:properties_of_U_functors} we may summarize the whole situation in the following proposition. 
    \begin{proposition}\label{prop:relation_lab_diag_strat}
    The following diagram of functors commutes up to natural isomorphism.
        \begin{equation}\label{diag:overview_over_vert_and_lab}
        \begin{tikzcd}
        \sS_P  \arrow[d]& \arrow[l] \mathrm{V}\sS_P \arrow[d]& \arrow[l,"V"', "\sim"] \PsS \arrow[r, "U", "f.f."'] \arrow[d]& \Diag_P^{\cof} \arrow[lll, bend right = 20, "C_P"]  \arrow[d] \\
        \Top_P   & \arrow[l] \mathrm{VCW}_P    & \arrow[l, "V"', "\sim"] \PCW \arrow[r, "U^{\Top}", "f.f"']& \Diag_P^{\Top,\cof} \arrow[lll, bend left = 20, "C'_P"']
    \end{tikzcd}
    \end{equation}
    Here, all verticals are induced by realization and the left two horizontals are given by the forgetful functors forgetting the pre-verticalizations. Both $V$ functors define equivalences of categories. The $U$ functors are both fully faithful, and the upper one is even an equivalence of categories.
    \end{proposition}
The fact that $P$-labeled CW-complexes embed fully faithfully into $\Diag^{\Top,\cof}_P$ (\cref{prop:properties_of_U_functors}), allows one to model the homotopy category $\Ho \Top_P$ through these objects. To see this, we first need a definition of label preserving homotopy.
\begin{definition}
Let $f,g\colon \lab{T}\to \lab{T'}$
be two label preserving maps between $P$-labelled CW-complexes. A \define{label preserving homotopy} from $f$ to $g$ is a map,
\begin{equation*}
    H\colon T\times [0,1]\to T',
\end{equation*}
such that for all cells $e_{\alpha}\in T$ and $e_{\beta}\in T'$, if $H(e_{\alpha}\times [0,1])\cap e_{\beta}\not=\emptyset$, then $\lambda_T(e_{\alpha})\subset \lambda_{T'}(e_{\beta})$. Equivalently, $H$ is a label-preserving map for the induced labelling on $T\times [0,1]$. If such a homotopy exists, we say $f$ and $g$ are label preserving homotopic, writing $f \simeq_{\text{lab}}g$.
We write $\PCW/{\simeq_{\text{lab}}}$ for the category whose objects are $P$-labelled CW-complex and whose sets of morphisms between $\lab{T}$ and $\lab{T'}$ is $\PCW(\lab{T},\lab{T'})/{\simeq_{\text{lab}}}$. 
\end{definition}
\begin{remark}\label{rem:desc_of_homotopies}
    Note, that under the fully faithful functor $V\colon \PCW \to \mathrm{VCW}_P$, $P$-labelled homotopies correspond one to one to vertical homotopies. Furthermore, under the fully faithful functor $U^{\Top}\colon \PCW \to \Diag^{\Top,\cof}$ they correspond one to one to those homotopies in $\Diag_P^\Top$ which are defined through the cylinder given by $(F \times [0,1])( \I) =  F( \I) \times [0,1]$.
\end{remark}
We may use \cref{prop:relation_lab_diag_strat} to obtain a more high-level proof of part of \cref{prop:vertical_approx}. We focus on the $\Top_P$ part here. A proof for $\Top_{N(P)}$ works analogously.
\begin{proof}[Alternative proof of \cref{prop:vertical_approx}]\label{proof:alternative_proof_in_appendix}
Let us first focus on the absolute case.
By \cref{rem:desc_of_homotopies} and \cref{prop:relation_lab_diag_strat}, we may equivalently show that in the following following commutative diagram of categories
\[
\begin{tikzcd}
    \PCW \arrow[d] \arrow[rr, bend left = 20, "V"] \arrow[r, "U^{\Top}", "f.f."']& \Diag_P^{\Top, \cof} \arrow[r, "C'_P"] \arrow[d]  & \Top^{\cof}_P \arrow[d] \\
     \PCW/{\simeq_{\text{lab}}} \arrow[r,]  & \Diag_P^{\Top, \cof}/_{\sim} \arrow[r] &               \Ho\Top_P ,
\end{tikzcd}
\]
the lower horizontal composition is fully faithful. Here, by $\Diag_P^{\Top, \cof}/_{\sim}$ we denote the category obtained by identifying morphisms which are homotopic through the cylinder given by $(F \times [0,1])(\I) = F(\I) \times [0,1]$. By \cref{prop:relation_lab_diag_strat}, the left upper horizontal is fully faithful. Using the second part of \cref{rem:desc_of_homotopies}, we also obtain that the left lower horizontal is fully faithful. 
Since every object in $\Diag^{\Top}_P$ is fibrant, we have a natural equivalence $\Diag_P^{\Top, \cof}/_{\sim} \cong \Ho \Diag^{\Top}_P$. Now, finally note that $C'^{\Top}_P \colon \Diag^{\Top}_P \to \Top_P$ is the left part of a Quillen equivalence (see \cref{rem:diag_top}). It follows, that the last remaining lower horizontal arrow is also fully faithful, finishing the proof of the absolute case. The relative case follows analogously, after verifying that $U^{\Top}$ maps inclusion $\lab{A} \hookrightarrow \lab{T}$ (where $\lambda_{A}$ is the induced labelling) to a cofibration in the projective model structure on $\Diag_P^{\Top}$.
\end{proof}
Finally, we may also use \cref{prop:relation_lab_diag_strat} together with \cref{rem:desc_of_homotopies} to obtain the following alternative version of \cref{Cor:CW_Model}, which allows one to perform stratified homotopy theory in the labelled setting.
\begin{corollary}\label{Cor:lab_CW_Model}
Verticalization induces an equivalence of categories
\begin{equation*}
    \PCW/{\simeq_{\textnormal{lab}}}\cong \Ho\Top_P.
\end{equation*}
\end{corollary}

\bibliographystyle{alpha}
	\bibliography{aqebib}
\end{document}